\documentclass[a4paper,11pt]{article}

\input{styleart.sty}

\begin{document}

\author{Chlo\'e Perin}
\title{Elementary embeddings in torsion-free hyperbolic groups}
\maketitle

\begin{abstract}We consider embeddings in a torsion-free hyperbolic group which are elementary in the sense of first-order logic. We give a description of these embeddings in terms of Sela's hyperbolic towers. We deduce as a corollary that subgroups elementarily embedded in finitely generated free groups are free factors. 
\end{abstract}

\section{Introduction}

Tarski's problem asks whether any two finitely generated non abelian free groups are elementary equivalent, namely whether they satisfy the same closed first-order formulas over the language of groups. In a series of articles starting with \cite{Sel1} and culminating in \cite{Sel6}, Sela answered this question positively (see also the work of Kharlampovich and Myasnikov \cite{KharlampovichMyasnikov}). Sela's approach is very geometric, and thus enables him in \cite{Sel7} to tackle problems on the first-order theory of torsion-free hyperbolic groups as well.

Another notion of interest in first-order theory is that of an elementary subgroup, or elementary embedding. Informally, a subgroup $H$ of a group $G$ is elementary if any tuple of elements of $H$ satisfies the same first-order properties in $H$ and $G$ (see Section \ref{LogicSec} for a definition).

Denote by $\F_n$ the free group on $n$ generators. To prove that finitely generated free groups of rank at least $2$ all have the same elementary theory, Sela shows in fact the following stronger result:
\begin{thm} \cite[Theorem 4]{Sel6} \label{FreeFactIsElementary} Suppose $2 \leq k \leq n$.
The standard embedding $\F_k \hookrightarrow \F_n$ is elementary.
\end{thm}

In this paper, we use some of Sela's techniques to give a description of elementary subgroups of torsion-free hyperbolic groups. Our main result is
\begin{thm}\label{MainResult}
Let $G$ be a torsion-free hyperbolic group. Let $H \hookrightarrow G$ be an elementary embedding. Then $G$ is a hyperbolic tower based on $H$.
\end{thm}

Hyperbolic towers are groups built by successive addition of hyperbolic floors, which can be described as follows. A group $G$ has a hyperbolic floor structure over a subgroup $G'$ if it is the fundamental group of a complex $X$ built by gluing some surfaces $\Sigma_1, \ldots, \Sigma_m$ along their boundary to the disjoint union of complexes $X'_1, \ldots, X'_l$, such that $G'$ is the fundamental group of a subcomplex $X'$ which contains the subcomplexes $X'_i$, and whose intersection with each surface $\Sigma_j$ is contractible (in particular $G'$ is isomorphic to the free product of the groups $\pi_1(X'_i)$). We require moreover the existence of a retraction $r:G \to G'$ which sends the fundamental groups $\pi_1(\Sigma_j)$ to non abelian images.

A hyperbolic tower over $H$ is built by successively adding hyperbolic floors to a ``ground floor'' which is the free product of $H$, closed surface groups and a free group (see Figure~\ref{HypTowerEx}). For a precise definition, see Definition \ref{HypTower}. 

Hyperbolic towers are defined by Sela in \cite{Sel1}, and enable him to give in \cite{Sel6} a description of finitely generated groups which are elementary equivalent to free groups. This structure is also used in \cite{Sel7} to give a classification of elementary equivalence classes of torsion-free hyperbolic groups. In fact, Proposition 7.6 of \cite{Sel7} shows that some particular subgroups of a torsion-free hyperbolic group $\Gamma$ (its ``elementary cores''), over which $\Gamma$ has a structure of hyperbolic tower, are elementarily embedded in $\Gamma$. According to Sela, the specific properties of these subgroups (apart from the structure of hyperbolic tower $\Gamma$ admits over them) are not used in the proof, which in fact shows that the converse of Theorem \ref{MainResult} holds \cite{SelaPrivate}.

\begin{figure}[!ht]
\begin{center}
\input{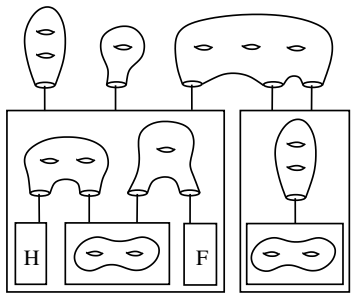}
\caption{A hyperbolic tower over $H$.}
\label{HypTowerEx}
\end{center}
\end{figure}

In the particular case where $G$ is a free group, we show that Theorem \ref{MainResult} implies the converse of Theorem \ref{FreeFactIsElementary}, so that we have
\begin{thm} \label{CaseFreeGroups} Let $H$ be a proper subgroup of $\F_n$. The embedding of $H$ in $\F_n$ is elementary if and only if $H$ is a non abelian free factor of $\F_n$.
\end{thm}

Let us consider the case where $G$ is the fundamental group of a closed hyperbolic surface. In the example represented on Figure~\ref{StupidEx}, the element corresponding to $\gamma$ can be written as a product of two commutators in $\pi_1(\Sigma)$, though not in $\pi_1(\Sigma_1)$. This is a property which can be expressed by a first-order formula, and that $\gamma$ satisfies on $\pi_1(\Sigma)$, though not in $\pi_1(\Sigma_1)$: the embedding of $\pi_1(\Sigma_1)$ in $\pi_1(\Sigma)$ is not elementary. 

\begin{figure}[H]
\begin{center}
\input{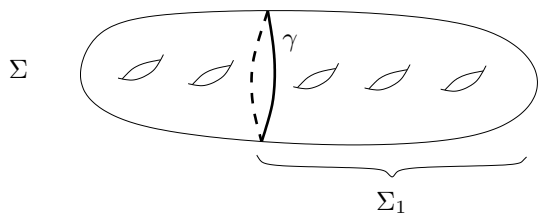}
\caption{$\pi_1(\Sigma_1)$ is not elementarily embedded in $\pi_1(\Sigma)$.}
\label{StupidEx}
\end{center}
\end{figure}

This example seems to suggest that an elementary subgroup of the fundamental group of a hyperbolic surface cannot be too big. In fact, we show that applying Theorem \ref{MainResult} gives
\begin{thm} \label{SurfaceCaseIntro} Let $S$ be the fundamental group of a closed hyperbolic surface $\Sigma$. Suppose $H$ is a proper subgroup of $S$ whose embedding in $S$ is elementary.

Then $H$ is a non abelian free factor of the fundamental group of a connected subsurface $\Sigma_0$ of $\Sigma$ whose complement in $\Sigma$ is connected, and which satisfies $|\chi(\Sigma_0)| \leq |\chi(\Sigma)|/2$ (with equality if and only if $\Sigma$ is the double of $\Sigma_0$).
\end{thm}

To prove Theorem \ref{MainResult}, we need to uncover a decreasing sequence $G=G_0 \geq G_1 \geq G_2 \geq \ldots$ of subgroups of $G$ which contain $H$, each of the subgroups $G_i$ forming a floor of a hyperbolic tower above the next subgroup $G_{i+1}$. 

Let us give an idea of the proof in the special case where $G$ is freely indecomposable relative to $H$. For each decomposition of $G$ as an amalgamated product or an HNN extension above a cyclic group for which $H$ lies in one of the factors, we consider the Dehn twists which fix the factor containing $H$. We then define the modular group $\Mod_H(G)$ as the group of automorphisms of $G$ generated by all such Dehn twists. 

The shortening argument of Rips and Sela gives the following result (this is a special case of one of the two key results for the construction of restricted Makanin-Razborov diagrams in \cite{Sel7}):
\begin{thm}\label{MRGammaRelPartialIntro} Let $G$ be a hyperbolic group which is freely indecomposable with respect to a non abelian subgroup $H$. 
There exists a finite set of proper quotients of $G$, such that for any non injective morphism $h:G \to G$ which fixes $H$, there is an element $\sigma$ of $\Mod_H(G)$ such that $h \circ \sigma$ factors through one of the corresponding quotient maps. 
\end{thm}

The key idea in the proof of Theorem \ref{MainResult} is to try and express this factorization result by a first-order formula that $H$ satisfies. 

Suppose now that $G$ does not admit any non trivial splitting over a cyclic subgroup in which $H$ is elliptic, so that the modular group $\Mod_H(G)$ is trivial. Theorem \ref{MRGammaRelPartialIntro} then implies that any non injective morphism $G \to G$ fixing $H$ factors through one of a finite set of quotients. Note that  if $H$ is a proper subgroup of $G$, a morphism $G \to H$ fixing $H$ is in particular a non injective morphism $G \to G$ fixing $H$.  Let $U$ be a finite set containing one non trivial element in the kernel of each one of the quotient maps given by Theorem \ref{MRGammaRelPartialIntro}. If $H$ has a finite generating set $\{h_1, \ldots, h_n\}$, we can express by a first-order formula $\phi(h_1, \ldots, h_n)$ satisfied by $H$ the following statement: ``any morphism $G \to H$ fixing the elements $h_i$ sends one of the elements of $U$ to a trivial image''. Since $H$ is elementary in $G$, the formula $\phi(h_1, \ldots, h_n)$ is satisfied by $G$, and its interpretation on $G$ is the following statement: `` any morphism $G \to G$ fixing each of the elements $h_i$ sends one of the elements of $U$ to a trivial image''. This is obviously not true of the identity map $G \to G$, so we get a contradiction. Thus if $H$ is elementary in $G$, there is at least one non trivial splitting of $G$ over a cyclic group in which $H$ is elliptic. 

In general, we do not know a priori that $H$ is finitely generated (though this is a consequence of Theorem \ref{MainResult}), and thus we cannot express the fact that a morphism fixes $H$ in a first-order formula. We can in fact generalize Theorem \ref{MRGammaRelPartialIntro} to show that any non injective morphism $G \to G$ which fixes a large enough finitely generated subgroup $H_0$ of $H$ factors after precomposition by a modular automorphism through one of finitely many quotients. We also need to show that a morphism $G \to H$ which fixes a large enough finitely generated subgroup $H_0$ of $H$ cannot be injective. For this, we prove that if $G$ is freely indecomposable with respect to a subgroup $H$, then it is freely indecomposable with respect to a finitely generated subgroup $H_0$ of $H$, and we combine this with the following result: 
\begin{thm} \label{CoHopfRelIntro} Let $G$ be a torsion-free hyperbolic group. Let $H_0$ be a non cyclic subgroup of $G$ relative to which $G$ is freely indecomposable. Then $G$ is co-Hopf relative to $H_0$, that is, if a morphism $\phi: G \rightarrow G$ is injective and fixes $H_0$ then it is an isomorphism.
\end{thm}

A more difficult problem is that of expressing precomposition by a modular automorphism even when the modular group is not trivial. To overcome this, we consider the cyclic JSJ decomposition of $G$ relative to $H$: it is a decomposition of $G$ as a graph of group with cyclic edge stabilizers in which $H$ is elliptic, and which is maximal in some sense among all such decompositions of this type that $G$ admits. Some of the vertex groups of this decomposition are fundamental group of surfaces with boundary. The point is that the modular group preserves some of the structure of the JSJ. The JSJ decomposition gives some combinatorial structure to both $G$ and $\Mod_H(G)$, which enables us to partially express precomposition by an element of the modular group in first-order. 

If the JSJ decomposition of $G$ relative to $H$ does not contain any surface groups, precomposition by a modular element can be completely expressed in a first-order formula, so that the argument we gave in the trivial modular group case generalizes and we get a contradiction. Thus if $H$ is elementary in $G$, the JSJ decomposition of $G$ relative to $H$ must contain some surface groups. It is precisely some of these surface groups which will appear in the hyperbolic floor structure we will uncover. 

As soon as the JSJ contains some surface group, we can only partially express precomposition by a modular automorphism, and the first-order formula $\phi(h_1, \ldots, h_n)$ satisfied by $H$ that we build expresses a result weaker than the factorization result. Its interpretation on $G$ does not give a contradiction, but a \textbf{non injective} morphism $G \to G$ which preserves some properties of the JSJ: we call such a morphism a preretraction. 

We then show that a non injective preretraction can be modified to get a retraction of some of the surface groups of the JSJ onto a proper subgroup of $G$ so as to form a floor of hyperbolic tower. This result is stated in Proposition \ref{Retraction}, and is proved using arguments of low dimensional topology and Bass-Serre theory. We then proceed by induction. 

Section \ref{LogicSec} recalls the definition of an elementary embedding. Section \ref{SurfaceGroupsSec} gives some basic but key results about surfaces with boundary and maps between such surfaces. These results will be used later to deduce Theorem \ref{SurfaceCaseIntro} from Theorem \ref{MainResult}, and are crucial in the proof of Proposition \ref{Retraction}. Section \ref{FactorSetsSec} recalls known results about the existence of factor sets, such as Theorem \ref{MRGammaRelPartialIntro}, and gives an outline of their proof. We also indicate how the same techniques can prove Proposition \ref{CoHopfRelIntro}. In Section \ref{MainResultSec}, we prove Theorem \ref{MainResult}, \ref{CaseFreeGroups} and \ref{SurfaceCaseIntro} assuming that both Proposition \ref{Retraction} and Proposition \ref{RetractionPlus} (another result concerning preretraction needed for the induction step) hold. In Section \ref{IsoSec}, we prove that a preretraction $A \to A$ which satisfies some injectivity conditions on some subgroups of $A$ must in fact be an isomorphism. This will be useful in the last section, which is devoted to the proof of Propositions \ref{Retraction} and \ref{RetractionPlus}.

\begin{rmk} This paper is a corrected version of \cite{PerinElementaryAnn}. The original version contained a mistake in the statement as well as in the proof of its Proposition 5.11 (correponding to Proposition 5.11 also in this version), which claimed that the existence of a preretraction implies the existence of a hyperbolic floor. This is not strictly true, and this is the reason why we introduce here the notion of \textbf{extended} hyperbolic floor. More details about the mistake and the modifications brought here can be found in \cite{PerinElementaryErratum}.
\end{rmk}

Most of the results presented here are part of the author's Ph.D. thesis. Many thanks are due to Zlil Sela for suggesting this problem as well as for the many discussions that followed, to Gilbert Levitt for his constant assistance, and to Thomas Delzant, Panos Papasoglu and Vincent Guirardel for their helpful advice. We are also grateful to the referee for his helpful remarks.

\section{Elementary embeddings} \label{LogicSec}
We give only an informal definition of first-order formulas, for a precise definition and more detailed background, the reader is referred to \cite{Chatzidakis}, or to \cite{ChangKiesler}.

The language of groups ${\cal L}$ is the set of symbols consisting of 
\begin{itemize}
	\item the symbols $\cdot$, $^{-1}$, and $1$, standing for the group multiplication, the inverse and the identity element respectively (these are specific to the language of groups);
	\item the usual first-order symbols: $=$, $\neq$, $\wedge$ (meaning ``and''), $\vee$ (meaning ``or''), $\Rightarrow$, and the quantifiers $\forall$ and $\exists$;
	\item variables $x_1, x_2, \ldots$
\end{itemize}
A first-order formula in the language of groups is a finite formula using these symbols. Note that variables always stand for elements of the group, so that in a first-order formula, we can only quantify on elements of the group and not on integers, say, or on subsets of the group. 

A variable $x$ is free in the formula $\phi$ if it is not bound by any quantifier, that is, if neither $\forall x$ nor $\exists x$ appear in $\phi$. We usually denote a first-order formula by $\phi(x_1, \ldots, x_n)$ if the variables which appear in it and are free are exactly $x_1, \ldots, x_n$. 

If $\phi(x_1, \ldots, x_n)$ is a first-order formula and $g_1, \ldots, g_n$ are elements of a group $G$, we say that $\phi(g_1, \ldots, g_n)$ is satisfied by $G$ if its interpretation  is true on $G$. This is denoted by $G \models \phi$. 

\begin{ex} \label{CentreElementary} Consider the following formula:
$$\phi(x) : \; \forall y \; xy = yx,$$
Given an element $g$ in a group $G$, we have $G \models \phi(g)$ if and only if $g$ lies in the centre of $G$.
\end{ex}

Let $H$ be a subgroup of a group $G$. Given a first-order formula $\phi(x)$ and an element $h$ of $H$, we could have $H \models \phi(h)$ and $G \not \models \phi(h)$, that is, $h$ could have different first-order properties in $H$ and in $G$. If this never happens, we say that $H$ is elementary in $G$. More precisely:
\begin{defi} \emph{(elementary embedding)} We say that the embedding of $H$ in $G$ is elementary (or that $H$ is an elementary subgroup of $G$) if for any first-order formula $\phi(x_1, \ldots, x_n)$ and for any elements $h_1, \ldots, h_n$, we have 
$$ H \models \phi(h_1, \ldots, h_n) \iff G \models \phi(h_1, \ldots, h_n).$$
This is denoted by $H \preceq G$.
\end{defi} 

\begin{ex} Suppose $H$ is an elementary subgroup of a group $G$. By considering the formula given in Example \ref{CentreElementary}, we see that an element $h$ of $H$ is in the centre of $H$ if and only if it is in the centre of $G$.
\end{ex}

Note that if the embedding $H \hookrightarrow G$ is elementary, in particular formulas without free variables are satisfied by $H$ if and only if they are satisfied by $G$ (i.e. $H$ and $G$ are elementary equivalent). So for example, we can see by considering the formula $\forall x \; \forall y \; xy=yx$ that $H$ is abelian if and only if $G$ is.

\section{Some preliminary results on surface groups} \label{SurfaceGroupsSec}
All the surfaces we consider are, unless otherwise stated, compact, connected, and with (possibly empty) boundary.

Let $\Sigma$ be a surface, and denote by $S$ its fundamental group. To each boundary component of $\Sigma$ corresponds a conjugacy class of maximal cyclic subgroups of $S$: we call these subgroups \emph{maximal boundary subgroups}, and their generators \emph{maximal boundary elements}. A non trivial element in a maximal boundary subgroup is called a \emph{boundary element}, and the cyclic group it generates is called a \emph{boundary subgroup}.

The group $S$ endowed with the set of its maximal boundary subgroups is called a \emph{surface group}. If $\Sigma$ and $\Sigma'$ are surfaces, a morphism $f$ between their fundamental groups $S$ and $S'$ is a \emph{morphism of surface groups} if it sends boundary elements to boundary elements.

We will denote by $\rk(F)$ the rank of a finitely generated free group $F$. We will need a notion of complexity for surfaces with non empty boundary such that if we cut a surface along a simple closed curve, the connected components we get have complexity strictly smaller than that of the original surface. The rank of the fundamental group will not be sufficient for this, as cutting a surface along a non separating simple closed curve preserves it, we thus give 
\begin{defi} \emph{(topological complexity)} \label{TopComplexity} Let $\Sigma$ be a surface with non empty boundary, denote by $S$ its fundamental group. The topological complexity $k(\Sigma)$ of $\Sigma$ is the pair $(\rk(S), -n)$, where $n$ is the number of boundary components of $\Sigma$. We order topological complexities lexicographically.
\end{defi}

Note that complexity determines a surface up to orientability.

\subsection{Surface groups acting on simplicial trees}

\begin{defi} \emph{(splitting $\Delta(S, {\cal C})$ dual to a set of simple closed curves)} \label{DualSplitting} Let $\Sigma$ be a surface, and let ${\cal C}$ be a set of non null-homotopic, two-sided, disjoint simple closed curves on $\Sigma$ (note that we do not exclude pairs of parallel curves). We denote $\Delta(S, {\cal C})$ the splitting of the fundamental group $S$ of $\Sigma$ dual to the set of curves ${{\cal C}}$ (given by the Van Kampen Lemma). We call the corresponding $S$-tree $T_{{\cal C}}$ the tree dual to ${{\cal C}}$.
\end{defi}

Theorem III.2.6 of \cite{MorganShalen} states that if the fundamental group of a surface $\Sigma$ acts minimally on a simplicial tree $T$, in such a way that boundary elements are elliptic and edge stabilizers are cyclic, then there exists a set ${\cal C}$ of non null-homotopic, two-sided and non boundary-parallel simple closed curves on $\Sigma$ such that $T$ is isomorphic to $T_{{\cal C}}$ (as an $S$-tree). This is proved by building an equivariant map between a universal cover of $\Sigma$ and $T$, which we choose so that the inverse image of midpoints of edges of $T$ give us lifts of non null-homotopic simple closed curves on $\Sigma$. 

We will give in Lemma \ref{CuttingLemma} a slightly different version of this result, which can be proved in essentially the same way. One difference is that we do not assume that the edge stabilizers are cyclic, so we get a surjective equivariant map $T_{\cal C} \to T$ which is not necessarily injective. Also, we restrict ourselves to sets ${\cal C}$ of simple closed curves which are pairwise non parallel, so that we lose simpliciality of the map $T_{\cal C} \to T$. 

\begin{defi} \label{EssentialSet} \emph{(essential set of curves on a surface)} A set ${\cal C}$ of simple closed curves on a surface is essential if its elements are non null-homotopic, two-sided, non boundary-parallel and pairwise non parallel.
We consider such sets up to homotopy.
\end{defi}

\begin{lemma} \label{CuttingLemma} Suppose that the fundamental group $S$ of a surface with boundary $\Sigma$ acts on a simplicial tree $T$, in such a way that boundary subgroups are elliptic. Then there exists an essential set of curves ${\cal C}$ on $\Sigma$, and an equivariant map $t:T_{\cal C} \to T$ between the $S$-tree dual to ${\cal C}$ and $T$.
\end{lemma}

\begin{rmk} The cyclic subgroups of $S$ corresponding to curves in ${\cal C}$ stabilize edges of $T$. The fundamental groups of connected components of the complement of ${\cal C}$ in $\Sigma$ are vertex groups of $\Delta(S, {\cal C})$, thus they are elliptic in $T$.
\end{rmk}

It will be useful to know that an equivariant map is locally minimal in the following sense:
\begin{defi} \emph{(locally minimal map)} \label{MinEquivariant} Let $G$ be a group which acts on (the topological  realizations of) simplicial trees $T$ and $T'$. Let $t$ be a continuous equivariant map $T \to T'$, which sends vertices to vertices. Up to subdividing $T$, we can also assume that $t$ sends an edge to a vertex or an edge. 

The map $t$ is said to be locally minimal if for any vertex $v$ of $T$ all of whose adjacent edges are sent to an edge $e'$ of $T'$, the stabilizer $G_{e'}$ of $e'$ in $T'$ is strictly contained in the stabilizer $G_{v}$ of $v$ in $T$.
\end{defi}

Note that this is strictly weaker than being an immersion, as edges can be collapsed or folded. 

\begin{rmk} \label{CanAssumeLocallyMinimal}If there is a continuous equivariant map $T \to T'$ between two $G$-trees, there is a locally minimal equivariant map $T \to T'$.
\end{rmk}

\subsection{Maps between surfaces}

Let $\Sigma$ and $\Sigma'$ be surfaces with non empty boundary, and let $\phi: \Sigma \to \Sigma'$ be a continuous map which sends $\partial \Sigma$ into $\partial \Sigma'$. We are interested in the corresponding map $\phi_*$ between the fundamental groups $S$ and $S'$ of $\Sigma$ and $\Sigma'$. Note that it is a morphism of surface groups.

We will now give two results which give sufficient geometric conditions on $\phi$ to guarantee respectively the injectivity and the virtual surjectivity of the morphism $\phi_*$. The first of these results is Theorem 3.1 of \cite{Gabai}. An arc in a surface with boundary $\Sigma$ is a path $\alpha: [0,1] \to \Sigma$ whose endpoints lie in the boundary, that is is said to be simple if $\alpha$ is injective, and boundary-parallel if it can be homotoped into the boundary relative to its endpoint. 

\begin{thm} \label{Gabai} Let $\Sigma$ and $\Sigma'$ be connected surfaces with non empty boundary, and denote by $S$ and $S'$ their respective fundamental groups.  Let $\phi: \Sigma \to \Sigma'$ be a continuous map such that $\phi(\partial \Sigma) \subseteq \partial \Sigma'$. If $\phi$ does not send any non boundary-parallel simple arc $\alpha$ in $\Sigma$ to a boundary-parallel arc $\phi \circ \alpha$ in $\Sigma'$, then the corresponding map $\phi_*$ from $S$ to $S'$ is injective.
\end{thm}

Thus injectivity is guaranteed provided non boundary-parallel simple arcs are sent to non boundary-parallel images. The next result guarantees virtual surjectivity of the map $\phi_*$ provided non null-homotopic simple closed curves are sent by $\phi$ to non null-homotopic images. We give
\begin{defi} \emph{(non pinching map)} \label{NonPinching} Let $\Sigma$ be a surface, denote by $S$ its fundamental group. A morphism $S \to G$ is said to be non pinching with respect to $\Sigma$ if its kernel does not contain any element corresponding to a non null-homotopic simple closed curve lying on $\Sigma$.
\end{defi}

We now have
\begin{lemma} \label{FiniteIndex} Let $S$ and $S'$ be fundamental groups of surfaces $\Sigma$ and $\Sigma'$ with non empty boundary. Let $f: S \to S'$ be a morphism of surface groups. If $f$ is non pinching with respect to $\Sigma$, and if $f(S)$ is not contained in a boundary subgroup of $S'$, then $f(S)$ has finite index in $S'$.
\end{lemma}

To prove this result, we will use
\begin{lemma} \label{InfiniteIndex} Let $Q$ be the fundamental group of a surface $\Xi$ with non empty boundary. If $Q_0$ is a finitely generated infinite index subgroup of $Q$, it is of the form $$ Q_0 = C_1*\ldots *C_m*F $$ where $F$ is a (possibly trivial) free group, $m \geq 0$, each of the groups $C_j$ is a boundary subgroup of $Q$, and any boundary element of $Q$ contained in $Q_0$ can be conjugated in one of the groups $C_j$ by an element of $Q_0$.
\end{lemma}

\begin{proof} By Theorem 2.1 in \cite{Scott}, there exists a finite covering $p: \Xi_1 \to \Xi$, and a subsurface $\Xi_0$ of $\Xi_1$, such that $Q_0$ is the image by the injection $p_*$ of the fundamental group of $\Xi_0$. Let $Q_1=\pi_1(\Xi_1)$, and identify $Q_1$ to its isomorphic image in $Q$ by $p_*$. The covering is finite, so $\Xi_1$ is compact, $Q_1$ is of finite index in $Q$, and the boundary elements of $Q_1$ are exactly the boundary elements of $Q$ contained in $Q_1$. Since $Q_0$ is of infinite index in $Q$, it must be of infinite index in $Q_1$.
Thus $\Xi_0$ is a proper subsurface of $\Xi_1$, and at least one of its boundary components is not parallel to a boundary component of $\Xi_1$. In particular, there is a basis of $Q_0$ as a free group which contains a maximal boundary element corresponding to each boundary component of $\Xi_0$ which is also a boundary component of $\Xi_1$. This basis gives the required free factor decomposition of $Q_0$.
\end{proof}

We can now prove Lemma \ref{FiniteIndex}.
\begin{proof} Suppose $f(S)$ has infinite index in $S'$. Then it admits a free product decomposition $C_1*\ldots *C_m*F $ as given by Lemma \ref{InfiniteIndex}, and $m \geq 1$ since boundary elements of $S$ are sent to boundary elements of $S'$. If $f(S)$ is not contained in a boundary subgroup of $S'$, this decomposition contains at least two factors, so the corresponding minimal $f(S)$-tree $T_0$ with trivial edge stabilizers is not reduced to a point. The group $S$ acts via $f$ on $T_0$ and boundary subgroups of $S$ are sent to boundary subgroups of $S'$, thus they lie in conjugates of the factors $C_i$ and are elliptic in $T_0$. By Lemma \ref{CuttingLemma}, we get a set of simple closed curves on $\Sigma$ whose corresponding elements stabilize edges of $T_0$ via $f$, i.e. have trivial image by $f$. This contradicts the fact that $f$ is non pinching.
\end{proof}

In the setting of Lemma \ref{FiniteIndex}, we can also deduce that the complexity of $\Sigma$ must be greater than that of $\Sigma'$ thanks to the following lemma.

\begin{lemma} \label{ComplexitiesAndFiniteIndex} Let $S$ and $S'$ be the fundamental groups of surfaces $\Sigma$ and $\Sigma'$ with non empty boundary. If $f: S \to S'$ is a morphism of surface groups such that $f(S)$ has finite index in $S'$, then the complexity of $\Sigma$ is at least that of $\Sigma'$, that is
$$k(\Sigma) \geq k(\Sigma');$$
and we have equality if and only if $f$ is an isomorphism of surface groups.
\end{lemma}

\begin{proof} A subgroup of finite index in a finitely generated free group of rank $r$ is a free group of rank at least $r$, with equality if and only if the index is $1$. Thus $\rk(S') \leq \rk(f(S))$ with equality if and only if $f$ is surjective. Now $\rk(f(S)) \leq \rk(S)$, and since free groups are Hopfian, we have equality if and only if $f$ is injective.
Thus $\rk(S') \leq \rk(S)$, with equality if and only if $f$ is bijective. If this is the case, $f$ sends non conjugate maximal boundary subgroups of $S$ to non conjugate maximal boundary subgroups of $S'$, so that $\Sigma'$ has at least as many boundary component as $\Sigma$, and $k(\Sigma) \geq k(\Sigma')$.
If the number of boundary components is the same, $f$ induces a bijection between conjugacy classes of maximal boundary subgroups, so the image by $f^{-1}$ of a boundary element is a boundary element, and $f^{-1}$ is also a morphism of surface groups. 
\end{proof}

By combining the previous results we can also get
\begin{lemma} \label{HomeoLemma} Suppose that $f: S_1 \to S_0$ is a non pinching morphism between surface groups corresponding to surfaces with boundary $\Sigma_1$ and $\Sigma_0$. Assume moreover that $f$ sends non conjugate maximal boundary elements to non conjugate maximal boundary elements. Then $f$ is an isomorphism of surface groups.
\end{lemma}

\begin{proof} There exists a continuous map $\phi: \Sigma_1 \to \Sigma_0$ such that $\phi_*=f$. The properties of $f$ imply that we can choose $\phi$ to be injective on $\partial \Sigma_1$, to satisfy $\phi(\partial \Sigma_1) \subseteq \partial \Sigma_0$, and to send non null-homotopic simple closed curve on $\Sigma_1$ to non null-homotopic images.

Suppose $f$ is not injective: by Theorem \ref{Gabai}, there is some non boundary-parallel simple arc $\alpha$ in $\Sigma_1$ such that $\phi(\alpha)$ is boundary-parallel. Since $\phi(\partial \Sigma_1) \subseteq \partial \Sigma_0$, up to restricting to a subpath we may assume that $\alpha$ intersects the boundary in its endpoint only. Since $\phi$ is injective on $\partial \Sigma_1$, the endpoints of $\alpha$ must belong to the same boundary component of $\Sigma_1$, so they are joined by two non homotopic paths $\beta$ and $\beta'$ which lie in this boundary component. Note that $\alpha \beta$ and $\alpha \beta'$ are both non null-homotopic simple closed curves. Since $\phi$ sends $\partial \Sigma_1$ into $\partial \Sigma_0$, we see that the image of either $\alpha \beta$ or $\alpha \beta'$ by $\phi$ is null-homotopic in $\Sigma_0$: this contradicts the properties of $\phi$. We deduce that $f$ is injective.

Now $f$ is non pinching, so by Lemma \ref{FiniteIndex}, the subgroup $f(S_1)$ has finite index in $S_0$. Thus there is a finite cover $p: \tilde{\Sigma}_0 \to \Sigma_0$ of $\Sigma_0$ such that $p_*(\tilde{S}_0) = f(S_1)$, where $\tilde{S}_0 = \pi_1(\tilde{\Sigma}_0)$. The map $p^{-1}_* \circ f$ is an isomorphism between $S_1$ and $\tilde{S}_0$ seen as free groups (without their surface group structure). Moreover, it sends maximal boundary elements of $S_1$ to maximal boundary elements of $\tilde{S}_0$. 

Let $M$ be a subset of $S_1$ containing one maximal boundary element for each boundary component of $\Sigma_1$: it does not extend to a basis of $S_1$ as a free group. Thus the image of $M$ by $p^{-1}_* \circ f$ is a set of maximal boundary elements of $\tilde{S}_0$ which does not extend to a basis of $\tilde{S}_0$. This is only possible if $p^{-1}_* \circ f(M)$ contains an element corresponding to each boundary component of $\tilde{\Sigma}_0$, so we deduce that $\tilde{\Sigma}_0$ has at most as many boundary components as $\Sigma_1$. 

Now $f$ is injective on the boundary of $\Sigma_1$ so that $\Sigma_1$ has at most as many boundary components as $\Sigma_0$; and $\tilde{\Sigma}_0$ is a finite cover of $\Sigma_0$, so $\Sigma_0$ has at most as many boundary components as $\tilde{\Sigma}_0$, with equality if and only if $p$ is a homeomorphism.

All these inequalities must therefore be equalities, so $p_*$ is an isomorphism and $f(S_1)=p_*(\tilde{S}_0)=S_0$, and $\Sigma_0$ has the same number of boundary components as $\Sigma_1$. This implies first that $f$ is surjective, then that it is in fact an isomorphism of surface groups.

\end{proof}

\section{Factor sets} \label{FactorSetsSec}

In this section, we recall the result obtained in \cite{Sel7} of the existence of a factor set for non injective homomorphisms into a torsion-free hyperbolic group $\Gamma$. We then indicate how to get a relative version (the ``restricted'' version of Sela). 

\subsection{Modular groups}

We want to define a subgroup of the group of automorphisms of a group $G$.
\begin{defi} \emph{(Dehn twist)} Suppose that the group $G$ has a decomposition as an amalgamated product $G = A *_C B$, or as an HNN extension $A*_C = \langle A, t \mid tct^{-1} = \phi(c) \textrm{ for all } c \in C \rangle$ for some embedding $\phi: C \to A$. Let $\gamma$ be an element of the centralizer of $C$.

The automorphism of $G$ which restricts to the identity on $A$, and which restricts to conjugation by $\gamma$ on $B$ (in the amalgamated product case) or which sends $t$ to $t\gamma$ (in the HNN case) is called the Dehn twist of $G$ corresponding to $\gamma$.
\end{defi}

We say that a group $G$ is freely indecomposable if it does not admit any non trivial free product decompositions. We will only use modular groups in the case where the group $G$ is torsion-free hyperbolic and freely indecomposable. In this case, the modular group of $G$ is generated by the automorphisms we just defined, together with inner automorphisms.

\begin{defi} \emph{(modular group $\Mod(G)$)} Let $G$ be a freely indecomposable torsion-free hyperbolic group. We define the modular group of $G$, denoted by $\Mod(G)$, to be the subgroup of $\Aut(G)$ generated by the Dehn twists of $G$ and inner automorphisms. 
\end{defi}

Note that for such a $G$, the only non trivial Dehn twists correspond to splittings over cyclic subgroups.

\subsection{Factor sets for morphisms to a torsion-free hyperbolic group}

We have the following result
\begin{prop} \label{MRGammaFI} Let $\Gamma$ be a torsion-free hyperbolic group. Let $G$ be a non cyclic freely indecomposable hyperbolic group. There exists a finite set of proper quotients of $G$ such that for any non injective morphism $f: G \to \Gamma$, there is an element $\sigma$ of $\Mod(G)$ such that $f \circ \sigma$ factors through one of the corresponding quotient maps. 
\end{prop}

Such a set of proper quotients is called a factor set for non injective morphisms $G \to \Gamma$.

This result can be seen as a consequence of \cite[Theorem 1.26]{Sel7}. However, it can also be proved directly by an argument similar to that of \cite[Theorem 1.25]{Sel7}, which states it in the case where $G$ is a $\Gamma$-limit group.

We give in the rest of this subsection an outline of the proof, which follows that of \cite[Theorem 1.25]{Sel7}. It is based on the powerful but technical shortening argument. 

The shortening argument was first used in \cite{RipsSelaHypI}, the reader is also referred to versions of it exposed in \cite{WiltonThesis} for the case where $\Gamma$ is free, and \cite{Bes} or \cite{ThesisPerin} in general. The idea is to consider sequences of morphisms $f_n: G \to \Gamma$, and the sequences of actions of $G$ on the $\delta$-hyperbolic Cayley graph of $\Gamma$ via these morphisms. It can be shown that after proper rescaling, such sequences converge to an action of $G$ on a path-connected $0$-hyperbolic space, namely a real tree \cite{PaulinGromov}. This limiting action is then analyzed using Rips theory (see \cite{RipsSelaHypI}, \cite{BFStableActions} and \cite{GuirardelRTrees}). Rips theory gives a decomposition of certain faithful actions on a real tree $T$ by finding a finite set of subtrees $T_1, \ldots, T_s$ of $T$ with the following properties: if $i \neq j$ or $g$ is non trivial then $T_i$ and $g \cdot T_j$ intersect in at most a point; the translates of the subtrees $T_i$ cover $T$; and the global stabilizer of $T_i$ acts on it in one of a few types of actions which can be described very precisely (see Theorem 10.8 in \cite{RipsSelaHypI} or Theorem 5.1 of \cite{GuirardelRTrees}). If the limit action obtained by the morphisms $h_n$ was faithful, this information can then be used to deduce information about the maps $f_n$, in particular it is possible to shorten the images by $f_n$ of the generators of $G$ by precomposing $f_n$ by an element of $\Mod(G)$.

\begin{defi}\emph{(stable sequence, stable kernel)} Let $G$ be a finitely generated group, and let $(h_n)_{n \in \N}$ be a sequence of morphisms from $G$ to a group $G'$. The sequence $(h_n)_{n \in \N}$ is stable if for any element $g$ of $G$, either all but finitely many of the $h_n(g)$ are trivial, or all but finitely many of the $h_n(g)$ are non trivial. The set of elements $g$ for which the former holds is a normal subgroup of $G$, we call it the stable kernel of the sequence $(h_n)_{n \in \N}$. 
\end{defi}

Note that by a diagonal argument, one can extract a stable subsequence from any sequence of morphisms.
For the rest of this section, let $\Gamma$ be a torsion-free hyperbolic group endowed with a finite generating set $D(\Gamma)$. 

\begin{defi}\emph{($\Gamma$-limit group)} A $\Gamma$-limit group is the quotient of a finitely generated group $G$ by the stable kernel of a stable sequence of morphisms $h_n: G \to \Gamma$. 
\end{defi}

Although $\Gamma$-limit groups are not necessarily finitely presented if $\Gamma$ is not free, Sela shows \cite[Theorem 1.17]{Sel7}: 
\begin{thm} \label{EventuallyFactors} If $\eta: G \to L$ is a $\Gamma$-limit quotient of $G$ corresponding to a stable sequence $h_n$ of morphisms $G \to \Gamma$, all but finitely many of the maps $h_n$ factor through $\eta$.
\end{thm}

Theorem \ref{EventuallyFactors} also implies the following result, which will be of use later.
\begin{thm} \label{DecSeqOfGammaLimit} If $(L_i)_{i \in \N}$ is a sequence of $\Gamma$-limit groups such that there exist surjective maps $\eta_i: L_i \to L_{i+1}$ for all $i$, then all but finitely many of the maps $\eta_i$ are isomorphisms.
\end{thm}

\begin{defi} \emph{(short morphism)} Let $G$ be a group endowed with a finite generating set $D(G)$. A morphism $h: G \to \Gamma$ is said to be short if 
$$ \max_{g \in D(G)} |h(g)|_{D(\Gamma)} \leq  \max_{g \in D(G)} |\gamma h(\sigma(g)) \gamma^{-1}|_{D(\Gamma)}$$
for any element $\sigma$ of $\Mod(G)$ and $\gamma$ of $\Gamma$. Here $|.|_{D(\Gamma)}$ denotes the word length in $\Gamma$ with respect to $D(\Gamma)$.
\end{defi}

\begin{defi} \emph{($\Gamma$ shortening quotient)} 
A $\Gamma$ shortening quotient $Q$ of a finitely generated group $G$ is the quotient of $G$ by the stable kernel of a sequence of non injective short morphisms $h_n: G \to \Gamma$. 

We order $\Gamma$ shortening quotients of a finitely generated group $G$ by the following relation: if $Q_1, Q_2$ are $\Gamma$ shortening quotients of $G$ with corresponding quotient maps $\eta_i: G \to Q_i$, we say $Q_1 \geq Q_2$ if there exists a morphism $\tau: Q_1 \to Q_2$ such that $\eta_2 = \tau \circ \eta_1$.
\end{defi}

Sela shows, using Theorem \ref{EventuallyFactors}, that every $\Gamma$ shortening quotient of $G$ is smaller than a maximal $\Gamma$ shortening quotient, and that there are only finitely many maximal $\Gamma$ shortening quotients $\eta_i: G \to M_i$ (Propositions 1.20 and 1.21 of \cite{Sel7}). 

Now suppose $f: G \to \Gamma$ is a non injective morphism, and let $\sigma$ and $\gamma$ be elements of $\Mod(G)$ and $\Gamma$ respectively, such that $h= \Conj (\gamma) \circ f \circ \sigma$ is short. The sequence $(h_n)_{n \in \N}$ of constant term $h_n=h$ is a sequence of non injective short morphisms, so the quotient of $G$ by its stable kernel (which is just the kernel of $h$) is a $\Gamma$ shortening quotient. Thus it is smaller than one of the $\Gamma$ maximal shortening quotients $M_i$, which means that $f \circ \sigma$ factors through the corresponding quotient map $\eta_i$. Thus to complete the proof of Proposition \ref{MRGammaFI}, there only remains to show that these $\Gamma$ maximal shortening quotients are proper. 

\begin{prop} \label{ShortQuotientsAreProper} If $G$ is a non cyclic and freely indecomposable hyperbolic group, then $\Gamma$ shortening quotients of $G$ are proper quotients.
\end{prop}

Theorem 1.25 of \cite{Sel7} claims that this holds for $G$ a freely indecomposable $\Gamma$ limit group. This is what allows Sela to then build Makanin-Razborov diagrams. The main difference is that to show this more general result, Sela has to deal also with axial components in the limit tree: when $G$ is assumed to be hyperbolic, there are no such components. However, since abelian subgroups of $\Gamma$-limit groups are well behaved, the shortening argument can be extended.  

\begin{proof}[Outline of the proof of Proposition \ref{ShortQuotientsAreProper}] Let $D$ be a finite generating set for $G$. Suppose we have a stable sequence $h_n: G \to \Gamma$ of non injective morphisms whose stable kernel is trivial. This implies in particular that $G$ is torsion-free. Our aim is to show that the $h_n$ cannot be short. 

For each $n$, consider the action of $G$ on the Cayley graph $X$ of $\Gamma$ via the morphism $h_n$. Pick a point $x_n$ in $X$ which minimizes the displacement function $x \mapsto \max_{g \in D} d_X(x, h_n(g)\cdot x)$ of this action. Rescale the distance on $X$ by the minimal displacement $\mu[h_n]=\max_{g \in D} d_X(x_n, h_n(g)\cdot x_n)$. We get a sequence $(X_n,x_n)$ of pointed $G$-spaces. 

Thanks to this rescaling, the sequence of actions $(\lambda_n)_{n \in \N}$ converges in the equivariant Gromov-Hausdorff topology \cite{PaulinGromov} to an action $\lambda$ of $G$ on a pointed metric space. This action is non trivial by choice of basepoints. 

The fact that the $h_n$ are non injective and that their stable kernel is trivial implies that they belong to infinitely many conjugacy classes. Thus, up to extraction of a subsequence, the rescaling constant $\mu[h_n]$ tends to infinity. 

Now if $X$ is $\delta$-hyperbolic, each $X_n$ is a $\delta/\mu[h_n]$-hyperbolic space, and the limit is a path-connected $0$-hyperbolic space, i.e. a real tree.

Using the fact that $G$ is torsion-free hyperbolic, and that the sequence of hyperbolicity constants of the spaces $X_n$ tends to $0$, we can show that the limit action satisfies some nice properties, such as abelianity of arc stabilizers and triviality of tripod stabilizers.

These conditions allow us to analyze the limit tree with Rips theory (see \cite{AcylindricalAccessibility} or \cite{GuirardelRTrees}), and this gives us a decomposition of $G$ as the fundamental group of a graph of groups. Note that there are no Levitt components since $G$ is freely indecomposable, and no axial components since $G$ is hyperbolic. We can thus use the shortening argument developed by Rips and Sela in \cite{RipsSelaHypI}: it shows that for any $n$ large enough, we can find an element $\sigma_n$ of $\Mod(G)$ such that the action $\lambda_n$ twisted by $\sigma_n$ is strictly shorter than $\lambda_n$, i.e. the displacement of the basepoint by $\lambda_n \circ \sigma_n$ is smaller than by $\lambda_n$. By our choice of basepoint, this implies that the morphisms $h_n$ were not short. 
\end{proof}

Note that the non injectivity of the maps $h_n$ is only used to show that the rescaling constant tends to infinity. Suppose now we are given an infinite sequence of pairwise non conjugate short injective maps $i_n$ from $G$ to $\Gamma$. The stable kernel of such a sequence is trivial. We build $(X_n, x_n)$ as above. The non conjugacy of the maps $i_n$ is sufficient to ensure that the rescaling constant tends to infinity, so that by following the argument above, we get a contradiction to the shortness of the maps $i_n$. Thus no such sequence exist, i.e. there is a finite number of conjugacy classes of short embeddings $G \to \Gamma$. This can be formulated as
\begin{thm} \label{ShortForEmbeddings} Let $\Gamma$ be a torsion-free hyperbolic group. Let $G$ be a non cyclic freely indecomposable hyperbolic group. 
There exists a finite set $\{i_1, \ldots, i_k\}$ of embeddings $G \hookrightarrow \Gamma$ such that for any embedding $i: G \hookrightarrow \Gamma$, there is an index $j$ with $1 \leq j \leq k$, an element $\gamma$ of $\Gamma$, and an element $\sigma$ of $\Mod(G)$ such that
$$ i = \Conj(\gamma) \circ i_j \circ \sigma.$$
\end{thm}

\subsection{Relative factor sets}

One of the most important hypotheses in Proposition \ref{ShortQuotientsAreProper} is the fact that $G$ is freely indecomposable: this is required to show that the limit tree has no Levitt components, a condition which is absolutely essential to make the shortening argument work. But in fact, the absence of Levitt components is also guaranteed if $G$ is only freely indecomposable relative to a subgroup $H$, provided $H$ is elliptic in the limit tree. To ensure this, we fix an embedding $H \hookrightarrow \Gamma$, and replace a few definitions and arguments by their relative versions. 

We say that a group $G$ is freely indecomposable with respect to a subgroup $H$ if it does not admit any non trivial free product decomposition of the form $G = G' * G''$ where $H$ is contained in $G'$.
We start by giving
\begin{defi} \emph{(relative modular group $\Mod_H(G)$)} Let $G$ be a torsion-free hyperbolic group, and let $H$ be a subgroup of $G$ with respect to which $G$ is freely indecomposable. The modular group of $G$ relative to $H$ is the subgroup of $\Aut(G)$ generated by those inner automorphisms and Dehn twists which restrict to the identity on $H$. We denote it by $\Mod_H(G)$.
\end{defi}
Note that if $H$ is not abelian, its centralizer in $G$ is trivial so the modular group is in fact generated by Dehn twists. In the relative case, the factor set existence result is given by

\begin{prop} \label{MRGammaRelPartial} Let $G$ be a hyperbolic group which is freely indecomposable with respect to a non abelian subgroup $H$. Let $\Gamma$ be a torsion-free hyperbolic group endowed with a fixed embedding $j:H \hookrightarrow \Gamma$.
There exists a finite set of proper quotients of $G$, and a finite subset $H_0$ of $H$, such that for any non injective morphism $h:G \to \Gamma$ which coincides with $j$ on $H_0$, there is an element $\sigma$ of $\Mod_H(G)$ such that $h \circ \sigma$ factors through one of the corresponding quotient maps. 
\end{prop}

The proof of this proposition is similar to the non relative case, we will thus only outline the differences. For the rest of this section, let $G$ be a hyperbolic group which is freely indecomposable with respect to a non abelian subgroup $H$, fix $D$ a finite generating set for $G$. Let $\Gamma$ be a torsion-free hyperbolic group endowed with a fixed embedding $j:H \hookrightarrow \Gamma$ and with a finite generating set $D(\Gamma)$. We will say that a morphism $G \to \Gamma$ fixes $H$ if it coincides with $j$ on $H$.

The notion of shortness of a morphism $G \to \Gamma$ is now changed to
\begin{defi} \emph{(short morphism relative to $H$)} A morphism $h: G \to \Gamma$ is said to be short relative to $H$ if 
$$ \max_{g \in D} |h(g)|_{D(\Gamma)} \leq  \max_{g \in D} |h(\sigma(g))|_{D(\Gamma)}$$
for any element $\sigma$ of $\Mod_H(G)$.
\end{defi}
The difference with the previous case is that we do not conjugate by an element of $G$: this is because we will only be interested in maps which fix $H$.

\begin{defi} \emph{(fixing $H$ in the limit)} \label{FixHLimit} Denote by $B_G(r)$ the set of elements of $G$ represented by words in $D$ whose length is at most $r$. We say that a sequence of morphisms $h_n: G \rightarrow \Gamma$ fixes $H$ in the limit if for any $r$, for all $n$ large enough, the map $h_n$ coincides on $B_G(r) \cap H$ with the fixed embedding $j: H \hookrightarrow \Gamma$.
\end{defi}

\begin{defi} \emph{($\Gamma$ shortening quotient relative to $H$)} A $\Gamma$ shortening quotient of $G$ relative to $H$ is the quotient of $G$ by the stable kernel of a stable sequence of non injective morphisms $h_n: G \to \Gamma$ which are short relative to $H$ and fix $H$ in the limit. 
\end{defi}

Note that $\Gamma$ shortening quotients relative to $H$ are in particular $\Gamma$-limit groups. Thus, they satisfy the strong descending chain condition given by Proposition \ref{EventuallyFactors}, which is required to prove that every $\Gamma$ shortening quotient relative to $H$ is under a maximal such quotient, and that maximal such quotients are in finite number.

Suppose that no finite set $H_0$ satisfying the conclusion of Proposition \ref{MRGammaRelPartial} exists. Then we can find a stable sequence $(h_r)_{r \in \N}$ of non injective morphisms $h_r: G \to \Gamma$ which fix $H$ on $B_G(r) \cap \Gamma$ and are short relative to $H$, yet do not factor through any of the maximal $\Gamma$ shortening quotients. The quotient of $G$ by the stable kernel of this sequence is a $\Gamma$-shortening quotient $\eta: G \to L$, so it is smaller than one of the maximal $\Gamma$ shortening quotient relative to $H$. Since by Proposition \ref{EventuallyFactors} all but finitely many of the maps $h_r$ factor through $\eta$, we get a contradiction.

As before, there only remains to show
\begin{prop} Let $G$ be a hyperbolic group which is freely indecomposable with respect to a non abelian subgroup $H$. Let $\Gamma$ be a torsion-free hyperbolic group endowed with a fixed embedding $j: H \hookrightarrow \Gamma$. Then $\Gamma$ shortening quotients of $G$ relative to $H$ are proper quotients.
\end{prop}

\begin{proof}[Outline of the proof] Suppose we have a stable sequence $h_n: G \to \Gamma$ of non injective morphisms which fix $H$ in the limit, and whose stable kernel is trivial. We want to see that the $h_n$ are not all short relative to $H$. The sequence $(h_n)_{n \in \N}$ gives a sequence of actions of $G$ on the Cayley graph $X$ of $\Gamma$. As before, we choose basepoints $x_n$, however the choice of basepoints is different: here we take $x_n$ to be simply the vertex corresponding to the identity element of $\Gamma$, and we rescale the metric on $X$ by the displacement of the basepoint which is now $\max_{g \in D} |h_n(g)|_{D(\Gamma)}$. We get in the limit an action $\lambda$ on a pointed metric space $(X_n,x_n)$.

This change in the choice of basepoints matches the change in our definition of shortness of a morphism, so that it is still true that if the $h_n$ are short relative to $H$, the actions $\lambda_n$ are short. However, the non triviality of the limit action is not immediate anymore: we can only show that the basepoint $x$ of the limit metric space is not a global fixed point.

Again the non injectivity of the maps $h_n$ implies that they belong to an infinity of conjugacy classes, so that the rescaling constant tends to infinity. Thus $\lambda$ is an action on a pointed real tree $(T,x)$, and since the morphisms $h_n$ fix $H$ in the limit, $H$ fixes $x$ in the action $\lambda$. 

As in the non relative case, the limit $G$-tree satisfies some nice conditions: in particular its arc stabilizers are abelian. Now if $\lambda$ has a global fixed point $y$, it must be distinct from $x$, but then $H$ stabilizes both $x$ and $y$ so it stabilizes the arc between them. This contradicts the non abelianity of $H$, and we deduce that the limit action is non trivial. 

We analyze the action $\lambda$ with Rips theory. Since $H$ fixes a point in $\lambda$ and $G$ is freely indecomposable with respect to $H$, there are no Levitt components, and since $G$ is hyperbolic, there are no axial components. The shortening argument gives us elements $\sigma_n$ of $\Mod_H(G)$ to shorten all but finitely many of the actions $\lambda_n$. Thus at most finitely many of the morphisms $h_n$ are short relative to $H$.
\end{proof}

If we start with a sequence of pairwise distinct injective maps $i_n$ from $G$ to $\Gamma$ which fix $H$, and are short relative to $H$, the rescaling constants $\max_{g \in D} |h_n(g)|_{n \in \N}$ still tend to infinity. Thus we can apply an argument similar to that used in the proof above, and we get a contradiction: this means that there are only finitely many such maps. We get
\begin{thm} \label{RelShortForEmbeddings} Let $G$ be a hyperbolic group which is freely indecomposable with respect to a non abelian subgroup $H$. Let $\Gamma$ be a torsion-free hyperbolic group endowed with a fixed embedding $H \hookrightarrow \Gamma$.
There exists a finite set $\{i_1, \ldots, i_k\}$ of embeddings $G \hookrightarrow \Gamma$ such that for any embedding $i: G \hookrightarrow \Gamma$ which fixes $H$, there is an index $j$ with $1 \leq j \leq l$, and an element $\sigma$ of $\Mod_H(G)$ such that
$$ i = i_j \circ \sigma.$$
\end{thm}
Note that this is stronger than the result obtained in the non relative case.

\subsection{Relative co-Hopf properties}

From Theorem \ref{RelShortForEmbeddings}, we can deduce a relative co-Hopf property for torsion-free hyperbolic groups:
\begin{cor} \label{CoHopfRel} Let $G$ be a torsion-free hyperbolic group. Let $H$ be a non cyclic subgroup of $G$ relative to which $G$ is freely indecomposable.
If $\phi: G \rightarrow G$ is injective and fixes $H$ then it is an isomorphism.
\end{cor}

Note that this can in fact shown to be true for any subgroup $H$. The proof in the case where $H$ is cyclic is almost the same, there is more work to do in the case where $H$ is trivial (see Theorem 4.4 of \cite{SelaHypII}).

\begin{proof} Suppose $\phi$ is a strict embedding: then the injective morphisms $\phi^n: G \to G$ all fix $H$, and their images are are pairwise distinct since they are strictly embedded one into the other: this contradicts Theorem \ref{RelShortForEmbeddings}.
\end{proof}

Now we can actually get a stronger statement by using the following lemma, suggested by Vincent Guirardel.
\begin{lemma} \label{FreelyIndecFG} If a finitely generated group $G$ is freely indecomposable relative to a subgroup $H$, then $H$ has a finitely generated subgroup relative to which $G$ is freely indecomposable.
\end{lemma}

\begin{proof} Suppose $G'$ is a subgroup of $G$. Denote by $T(G')$ the set of all simplicial $G$-trees $\tau$ with trivial edge stabilizers in which $G'$ fixes a vertex $v_{\tau}$. Define 
$$A(G') = \bigcap _{\tau \in T(G')} \Stab(v_{\tau})$$

To each $\tau$ in $T(G')$, we associate the corresponding free product decomposition of $G$.  Since $G$ is finitely generated, the number of factors of such a decomposition is bounded: let $m_G(G')$ be the maximal number of factors that such a decomposition can have. A decomposition with $m_G(G')$ factors is clearly of the form
$$A * B_1 * \ldots * B_r $$
where $B_1, \ldots , B_r$ are freely indecomposable (possibly cyclic), and $A$ contains $G'$ and is freely indecomposable with respect to $G'$. Such a decomposition corresponds to a tree $\tau$ in $T(G')$ so $A(G') \leq A$. But in any tree $\tau$ of $T(G')$, $A$ fixes the vertex $v_{\tau}$, so $A=A(G')$. 

If $G'\leq G''$, we have $T(G') \supseteq T(G'')$, so that $A(G') \leq A(G'')$ and
$m_G(G')\geq m_G(G'')$, and if we have equality, a maximal decomposition with respect to $G''$ is also a maximal decomposition with respect to $G'$ so that $A(G')=A(G'')$. 

We can now prove the lemma. Let $\{h_1, h_2,  \ldots\}$ be a generating set for $H$, and let $H_k= \langle h_1, \ldots h_k \rangle$ of $H$. The sequence $(m_G(H_k))_{k > 0}$ is non increasing, so it must stabilize, thus the sequence $A(H_k)$ stabilizes after some index $k_0$. In particular $H_k \leq A(H_k) \leq A(H_{k_0})$ for all $k$, so $H \leq A(H_{k_0})$. But $A(H_{k_0})$ is a free factor of $G$: since we assumed $G$ freely indecomposable with respect to $H$, we must have $A(H_{k_0})=G$, and $G$ is freely indecomposable with respect to $H_{k_0}$.
\end{proof}

We get a partial relative co-Hopf property for hyperbolic groups.
\begin{prop} \label{CoHopfRelPartial} Let $G$ be a torsion-free hyperbolic group. Let $H$ be a non cyclic subgroup of $G$, with respect to which $G$ is freely indecomposable. There exists a finite subset $H_0$ of $H$ such that if $\phi: G \rightarrow G$ is an injective morphism which fixes $H_0$, then it is an isomorphism.
\end{prop}

\begin{proof} Just take $H_0$ to be a generating set for the subgroup $H'$ given by Lemma \ref{FreelyIndecFG}. If $\phi$ fixes $H_0$, it fixes $H'$ relative to which $G$ is freely indecomposable. Thus we can apply Corollary \ref{CoHopfRel} to $G$ with the subgroup $H'$, to deduce that $\phi$ is an isomorphism.
\end{proof}

\section{Elementary embeddings in hyperbolic groups} \label{MainResultSec}
\subsection{Hyperbolic towers}

We will use the notations and results of Bass-Serre theory exposed in \cite{SerreTrees}. 
\begin{defi} \emph{(Bass-Serre presentation)} Let $G$ be a group acting on a simplicial tree $T$ without inversions, denote by $\Gamma$ the corresponding quotient graph of groups and by $p$ the quotient map $T \to \Gamma$. A Bass-Serre presentation for $\Gamma$ is a pair $(T^0, \Gamma^0)$ consisting of
\begin{itemize}
\item a subtree $T^0$ of $T$ which contains exactly one edge of $p^{-1}(e)$ for each edge $e$ of $\Gamma$;
\item a subtree $\Gamma^0$ of $T^0$ which is mapped injectively by $p$ onto a maximal subtree of $\Gamma$;
\end{itemize}
\end{defi}

\begin{defi} \emph{(graph of groups with surfaces)} \label{GOGWithSurfaces} A graph of groups with surfaces is a graph of groups $\Lambda$ together with a set $V_{S}$ of its vertices such that any vertex $v$ in $V_S$ satisfies:
\begin{itemize} 
\item there exists a compact connected surface with boundary $\Sigma$, such that the vertex group $G_v$ is the fundamental group $S$ of $\Sigma$; 
\item for each edge $e$ adjacent to $v$, the injection $i_e: G_e \hookrightarrow G_v$ maps $G_e$ onto a maximal boundary subgroup of $S$;
\item this induces a bijection between the set of edges adjacent to $v$ and the set of conjugacy classes in $S$ of maximal boundary subgroups of $S$.
\end{itemize}
The vertices of $V_S$ are called surface type vertices. A vertex $v$ of the tree $T_{\Lambda}$ corresponding to $\Lambda$ whose projection $p(v)$ to $\Lambda$ is of surface type is also said to be of surface type. The surfaces corresponding to surface type vertices of $\Lambda$ are called the surfaces of $\Lambda$.
\end{defi}

We define extended hyperbolic towers. 
\begin{defi} \emph{((extended) hyperbolic floors)} \label{HypFloor} Consider a triple $(G, G', r)$ where $G$ is a group, $G'$ is a subgroup of $G$, and $r$ is a retraction from $G$ onto $G'$. 

We say that $(G, G', r)$ is an extended hyperbolic floor if there exists a non trivial decomposition $\Gamma$ of $G$ as a graph of groups with surfaces, and a Bass-Serre presentation $(\Gamma^0, T^0)$ of $\Gamma$ such that:
\begin{itemize}
\item the surfaces of $\Gamma$ which are not once punctured tori have Euler characteristic at most $-2$; 
\item $G'$ is the free product of the stabilizers of the non surface type vertices of $\Gamma^0$;
\item every edge of $\Gamma$ joins a surface type vertex to a non surface type vertex (bipartism);
\item either the retraction $r$ sends surface type vertex groups of $\Gamma$ to non abelian images; or $G'$ is cyclic and there exists a retraction $r': G * \Z \to G' * \Z$ which sends surface type vertex groups of $\Gamma$ to non abelian images.
\end{itemize}
If the first alternative holds in this last condition, we say that $(G,G',r)$ is a hyperbolic floor.
\end{defi}

\begin{defi} \emph{((extended) hyperbolic tower)} \label{HypTower}
Let $G$ be a non cyclic group, let $H$ be a subgroup of $G$.
We say that $G$ is an extended hyperbolic tower based on $H$ if there exists a finite
sequence $G=G^0 \geq G^1 \geq \ldots \geq G^m \geq H$ of subgroups of $G$ where $m \geq 0$ and:
\begin{itemize}
\item for each $k$ in $[0, m-1]$, there exists a  retraction $r_k:G^{k} \rightarrow G^{k+1}$
such that the triple $(G^k, G^{k+1}, r_k)$ is an (extended) hyperbolic floor, and $H$ is contained in one of the non surface type vertex group of the corresponding hyperbolic floor decomposition;
\item $G^m = H * F * S_1 * \ldots * S_p$ where $F$ is a (possibly trivial) free group, $p \geq 0$, and each $S_i$
is the fundamental group of a closed surface without boundary of Euler characteristic at most $-2$.
\end{itemize}
Note that for $0 \leq k < m-1$, $(G^k, G^{k+1}, r_k)$ can only be a (non extended) hyperbolic floor. Moreover if $(G^{m-1}, G^m, r_{m-1})$ is not a hyperbolic floor, then $G^{m}$ must be cyclic, so $H$ is cyclic or trivial. In particular every extended hyperbolic tower over a non trivial non cyclic group is in fact a hyperbolic tower. 
\end{defi}

In the published version of this paper \cite{PerinElementaryAnn}, we did not define extended hyperbolic floors and towers, but only hyperbolic floors and towers. However the proof of Proposition 5.11 of \cite{PerinElementaryAnn} overlooked some special cases, and for the statement of the proposition to be true, hyperbolic floors should be replaced by \textbf{extended} hyperbolic floors (see Proposition \ref{Retraction} in this paper).

Note that the definition of hyperbolic towers coincides exactly with that given in Sela's work, where they are called "non-elementary hyperbolic $\omega$-residually free towers" (see the paragraph before Proposition 6 in \cite{Sel6}, and Definition 6.1 of \cite{Sel1} for the definition of $\omega$-residually free towers). 

In fact, non-elementary hyperbolic $\omega$-residually free towers should also be replaced by {\bf extended} hyperbolic towers for Proposition 6 of \cite{Sel6} to hold : a group $G$ that is elementary equivalent to a non abelian free group will only admit the structure of an {\bf extended} hyperbolic tower over the trivial subgroup. This is because Proposition \ref{Retraction} is an intermediate step in the proof of Proposition 6 of \cite{Sel6}, and that there Sela also overlooked the possibility of extended hyperbolic floors. According to Sela, the converse of this modified result also holds, namely any extended hyperbolic tower over the trivial subgroup is elementary equivalent to the free group $\F_2$ \cite{SelaPrivate}.

\begin{rmk} \label{TransitivityOfHypTowers} If $G_1$ and $G_2$ are (extended) hyperbolic towers over subgroups $H_1$ and $H_2$, then $G_1*G_2$ is an (extended) hyperbolic tower over $H_1 * H_2$. If $G$ is an (extended) hyperbolic tower over a subgroup $G'$, and $G'$ is an (extended) hyperbolic tower over a subgroup $H$, then $G$ is an (extended) hyperbolic tower over $H$. If $H * G_1$ and $H * G_2$ are (extended) hyperbolic towers over $H$, then $H* G_1*G_2$ is a hyperbolic tower over $H$.
\end{rmk}

Recall that our main result, Theorem \ref{MainResult}, says that if $G$ is a torsion-free hyperbolic group, and $H$ is an elementary subgroup of $G$, then $G$ is a hyperbolic tower based on $H$.

To prove Theorem \ref{MainResult}, we need to construct successive
retractions from subgroups of $G$ to proper subgroups until we get to $H$. The strategy will be to build by the mean of first-order sentences some maps that we will call preretractions, which preserve some characteristics of the cyclic (relative) JSJ decomposition of these subgroups of $G$ (with respect to $H$). Then we will show in Section \ref{RetractionSec} that the existence of a preretraction implies the existence of an extended hyperbolic floor.

\subsection{JSJ decompositions}

A JSJ decomposition $\Lambda$ of a group $G$ is a decomposition as a graph of groups which encodes all possible splittings of the group $G$ over a given class ${\cal E}$ of subgroups. The standard reference for the case where $G$ is finitely presented and one-ended, and ${\cal E}$ is the class of finite and cyclic subgroups of $G$ is \cite{RipsSelaJSJ}. This has been generalized in \cite{DunwoodySageev} and \cite{FujiwaraPapasoglu} to the case where ${\cal E}$ is the class of slender subgroups of $G$. For a unifying approach, see \cite{GuirardelLevittJSJI, GuirardelLevittJSJII}. In the case where $G$ is one-ended hyperbolic, \cite{Bowditch} gives a canonical construction. A JSJ decomposition of a group $G$ relative to a subgroup $H$ is a graph of groups decomposition in which $H$ is elliptic, and which encodes all possible splittings of $G$ in which $H$ is elliptic and edge groups lie in ${\cal E}$.

In the sequel, we will use JSJ decompositions in the case where $G$ is torsion-free hyperbolic and freely indecomposable (respectively freely indecomposable with respect to a subgroup $H$), and ${\cal E}$ is the class of cyclic groups. We call such a decomposition a cyclic JSJ decomposition of $G$ (respectively a relative cyclic JSJ decomposition with respect to $H$). In this case, a (relative) cyclic JSJ decomposition admits a natural structure of graph of groups with surfaces (see \cite{Bowditch} or \cite[Theorem 7.7]{GuirardelLevittJSJI}, and \cite[Theorem 1.5]{GuirardelLevittJSJII} for the relative case). Moreover, in the relative case, the subgroup $H$ lies in a non surface type vertex group. It will be convenient to consider as surface type vertices only those whose corresponding surface has characteristic at most $-2$, or are punctured tori (i.e. surfaces admitting a pseudo-Anosov). In particular, the JSJ of the closed surface of characteristic $-1$ will be considered to consist of a single non surface type vertex.

An important property of such a cyclic (relative) JSJ decomposition $\Lambda$ is that its vertex groups are ``preserved'' under modular automorphisms, as given by 
\begin{lemma} \label{ModAndJSJ} Let $G$ be a torsion-free hyperbolic group which is freely indecomposable (with respect to a subgroup $H$). Let $\Lambda$ be a (relative) cyclic JSJ decomposition (with respect to $H$). An element of $\Mod(G)$ (respectively of $\Mod_H(G)$) restricts to conjugation on each non surface type vertex group of $\Lambda$, and sends surface type vertex groups isomorphically on conjugates of themselves.
\end{lemma}

This lemma is a consequence of the universal property of cyclic (relative) JSJ decompositions: recall that the modular group $\Mod(G)$ is generated by automorphisms of $G$ which preserve some cyclic splitting of $G$, and that a JSJ decomposition in some sense contains all such splittings.

Following \cite{GuirardelLevittJSJI}, a group $G$ might admit several JSJ decompositions, but it is often possible to choose one which admits nice properties by the tree of cylinders construction \cite{GuirardelLevittTreeOfCylinders}. In the case where $G$ is one-ended torsion-free hyperbolic, for example, this construction gives a JSJ decomposition which is $2$-acylindrical \cite[Theorem 2]{GuirardelLevittTreeOfCylinders}. We will call it \textbf{the} cyclic (relative) JSJ decomposition of $G$ (with respect to $H$). 

The properties of the JSJ decomposition that we will use (in addition to the fact that it satisfies Lemma \ref{ModAndJSJ}) are summarized in the following remark.
\begin{rmk} \label{JSJIsJSJLike} Let $\Lambda$ be the cyclic (relative) JSJ decomposition (with respect to a subgroup $H$) of a torsion-free hyperbolic group $G$ which is freely indecomposable (relative to $H$). Then
\begin{enumerate}
\item[(i)] \label{CyclicEdges} the edge groups of $\Lambda$ are cyclic;
\item[(ii)] \label{NoNeighb} an edge of $\Lambda$ is adjacent to at most one surface type vertex, and to at most one vertex with cyclic vertex group;
\item[(iii)] \label{Acylindricity} (strong $2$-acylindricity) if a non trivial element of $A$ stabilizes two distinct edges of the tree $T_{\Lambda}$ corresponding to $\Lambda$, they are adjacent and their common endvertex has cyclic stabilizer;
\item[(iv)] the surfaces of $\Lambda$ are punctured tori or have characteristic at most $-2$.
\end{enumerate}
\end{rmk}

Many of the results we will need about cyclic (relative) JSJ decompositions only use the properties given by Remark \ref{JSJIsJSJLike}. We thus give
\begin{defi} \emph{(JSJ-like decomposition)} \label{JSJLike} Let $\Lambda$ be a graph of groups with surfaces, with fundamental group $A$. We say that $\Lambda$ is a JSJ-like decomposition of $A$ if it satisfies the properties $(i)$-$(iv)$ given in Remark \ref{JSJIsJSJLike}.
\end{defi}

Thus in particular a (relative) cyclic JSJ decomposition is a JSJ-like decomposition. 

Note that distinct vertices of the tree corresponding to a JSJ-like decomposition have distinct stabilizers. Also, since surface groups are restricted so that boundary subgroups are malnormal, a JSJ-like decomposition is $1$-acylindrical next to surface type vertices, that is, no non trivial element stabilizes distinct edges adjacent to a same surface type vertex.

\subsection{Preretractions}

Preretractions are morphisms that preserve some of the structure of a JSJ-like decomposition. We need to define them as maps $A \to G$ where $A$ is a subgroup of $G$.

\begin{defi} \emph{(preretraction)} \label{Preretraction} Let $G$ be a group, let $A$ be a subgroup of $G$, and let $\Lambda$ be a JSJ-like decomposition of $A$. A morphism $A \to G$ is a preretraction with respect to $\Lambda$ if its restriction to each non surface type vertex group $A_v$ of $\Lambda$ is a conjugation by some element $g_v$ of $G$,
and surface type vertex groups have non abelian images.
\end{defi} 

\begin{rmk} By definition of a JSJ-like decomposition, the restriction of a preretraction to an edge group of $\Lambda$ is also a conjugation by an element of $G$. 
\end{rmk}

We will now give two results which are central in our proof of Theorem \ref{MainResult}.

\begin{prop} \label{Retraction} \label{RETRACTION} Let $A$ be a torsion-free hyperbolic group. Let $\Lambda$ be a cyclic JSJ-like decomposition of $A$ which does not consist of a single surface type vertex.

Assume that there exists a non injective preretraction $A \rightarrow A$ with respect to $\Lambda$. Then there exists a subgroup $A'$ of $A$ and a retraction $r$ from $A$ to $A'$ such that $(A,A',r)$ is an extended hyperbolic floor. Moreover given a rigid vertex group $R_0$ of $\Lambda$ we can choose $A'$ to contain $R_0$.
\end{prop}

The second proposition gives sufficient conditions to get a preretraction $A \to A$ from a preretraction $A \to G$, so that we can then apply Proposition \ref{RETRACTION}. It will be needed for the induction steps in the proof of \ref{MainResult}.

\begin{prop} \label{RetractionPlus} \label{RETRACTIONPLUS} Let $G$ be a torsion-free hyperbolic group. Let $A$ be a non cyclic retract of $G$ which admits a cyclic JSJ-like decomposition $\Lambda$. Suppose $G'$ is a subgroup of $G$ containing $A$ such that either $G'$ is a free factor of $G$, or $G'$ is a retract of $G$ by a retraction $r:G \to G'$ which makes $(G, G', r)$ a hyperbolic floor.

If there exists a non injective preretraction $A \to G$ with respect to $\Lambda$, then there exists a non injective preretraction $A \to G'$ with respect to $\Lambda$. 
\end{prop}

The proof of these two propositions is postponed until the last section of this paper. Both results are intermediate steps in the proof of Proposition 6 of \cite{Sel6}, though they are not explicitly stated there. For now, we will assume these two results hold, and use them to prove Theorem \ref{MainResult}.

\subsection{Using first-order to build preretractions}

Suppose that $H$ is a group with an elementary embedding in a torsion-free hyperbolic group $G$. To show that $G$ admits a structure of hyperbolic tower over $H$, we will try to build non injective preretractions, as Theorem \ref{Retraction} can then be used to produce hyperbolic floor structures.

This is what the two following results will enable us to do. They form the heart of the proof of Theorem \ref{MainResult}.
\begin{prop} \label{PreretractionForA} Suppose that $G$ is a non cyclic torsion-free hyperbolic group, and let $H$ be a subgroup whose embedding in $G$ is elementary. 
Suppose $A$ is a subgroup of $G$ which is hyperbolic, properly contains $H$, and is freely indecomposable relative to $H$. Let $\Lambda$ be the cyclic JSJ decomposition of $A$ relative to $H$. Then there exists a non injective preretraction $A \to G$ with respect to $\Lambda$. 
\end{prop}

Note that this implies in particular that $\Lambda$ is non trivial. 

\begin{prop} \label{PreretractionForB} Suppose that $G$ is a torsion-free hyperbolic group, and that $H$ is a subgroup elementarily embedded in $G$ which is also a retract of $G$. Let $B$ be a freely indecomposable hyperbolic subgroup of $G$ which is neither cyclic nor a closed surface group of Euler characteristic at most $-2$. Let $\Lambda$ be the cyclic JSJ decomposition of $B$. Suppose that no non trivial element of $B$ is conjugate in $G$ to an element of $H$. Then there exists a non injective preretraction $B \to G$ with respect to $\Lambda$.
\end{prop}

An important point is that these two results do not require that any of the embeddings $H \hookrightarrow A$,  $A \hookrightarrow G$, or $B \hookrightarrow G$ be elementary. We only need the embedding of $H$ in $G$ to be elementary.

Let us now prove Propositions \ref{PreretractionForA} and \ref{PreretractionForB}. The idea of both proofs is to express (a consequence of) the existence of a factor set by a first-order logic formula satisfied by $H$ whose interpretation on $G$ then gives us a non injective preretraction. 

We will use the following definition
\begin{defi} \emph{($\Lambda$-related morphisms)} Let $A$ be a group which admits a JSJ-like decomposition $\Lambda$, and let $f$ be a morphism from $A$ to a group $G$. We say that a morphism $f':A \to G$ is $\Lambda$-related to $f$ if
\begin{itemize}
	\item for each non surface type vertex group $R$ of $\Lambda$, there exists an element $u_R$ of $G$ such that $f'$ restricted to $R$ is $\Conj(u_R) \circ f$;
  \item each surface type vertex group of $\Lambda$ which has non abelian image by $f$ also has non abelian image by $f'$.
\end{itemize}
\end{defi}

\begin{rmk} \label{RelToiIsPreretraction}
Suppose that $A$ is a subgroup of $G$, that it admits a JSJ-like decomposition $\Lambda$, and let $f: A \to G$ be a morphism. Then $f$ is $\Lambda$-related to the embedding $A \hookrightarrow G$ if and only if it is a preretraction with respect to $\Lambda$.
\end{rmk}

\begin{rmk} \label{PrecByModIsRelated} Let $A$ be a torsion-free hyperbolic group which is freely indecomposable (with respect to a subgroup $H$), and let $\Lambda$ be a cyclic (relative) JSJ decomposition (with respect to $H$). Let $f$ be a morphism $A \to G$.

If $\sigma$ is an element of $\Mod(G)$ (respectively of $\Mod_H(G)$), then $f'=f \circ \sigma$ is $\Lambda$-related to $f$ by Lemma \ref{ModAndJSJ}.
\end{rmk}

The following lemma shows that $\Lambda$--relatedness can be expressed in first-order logic.  

\begin{lemma} Let $A$ be a group generated by a finite tuple ${\mathbf a}$. Suppose that $A$ admits a JSJ-like decomposition $\Lambda$. 
There exists a formula $\Rel(\mathbf{x}, \mathbf{y})$ such that for any pair of morphisms $f$ and $f'$ from $A$ to $G$, the formula $\Rel(f(\mathbf{a}), f'(\mathbf{a}))$ is satisfied by $G$ if and only if $f'$ is $\Lambda$-related to $f$.
\end{lemma}

\begin{proof} We introduce some notation. Denote by $R_1, \dots R_r$ the non surface type vertex groups of $\Lambda$, and by $S_1, \ldots S_s$ its surface type vertex groups. For $1 \leq p \leq r$, choose a finite generating set ${\mbox{\boldmath$\rho$}}_p$ for $R_p$, and for $1 \leq q \leq s$, choose a finite generating set ${\mbox{\boldmath$\sigma$}}_q$ for $S_q$. We take the convention to denote tuples by bold font. 

The elements of ${\mbox{\boldmath$\rho$}}_p$ and ${\mbox{\boldmath$\sigma$}}_q$ can be represented by words in the elements ${\mathbf a}$, we denote these by  ${\mbox{\boldmath$\rho$}}_p = \bar{{\mbox{\boldmath$\rho$}}}_p({\mathbf a})$ and ${\mbox{\boldmath$\sigma$}}_q = \bar{{\mbox{\boldmath$\sigma$}}}_q({\mathbf a})$ respectively. 

Now, if $w$ is an element of $A$ which can be represented by a word $\bar{w}({\mathbf a})$, its image by the morphism $f: A \to G$ extending ${\mathbf a} \mapsto {\mathbf g}$ is represented by $\bar{w}({\mathbf g})$.

The maps $f$ and $f'$ extending ${\mathbf a} \mapsto {\mathbf g}$ and ${\mathbf a} \mapsto {\mathbf g}'$ respectively satisfy the relatedness condition on the rigid type vertex groups of $\Lambda$ if and only if
$$ \exists u_1 \dots \exists u_r \; \bigwedge^{r}_{p=1} \left \{  \bar{{\mbox{\boldmath$\rho$}}}_p({\mathbf g}') = u_p \bar{{\mbox{\boldmath$\rho$}}}_p({\mathbf g}) u^{-1}_p \right \}. $$

To express the abelianity of a subgroup generated by a tuple ${\mathbf z}=(z^1, \ldots, z^{q})$, we can use the formula $\Ab({\mathbf z}): \bigwedge_{i,j} \left \{ [z^i,z^j]=1 \right \} $. Thus the non abelianity condition about the image by $f$ and $f'$ of surface type vertex groups of $\Lambda$ can also be expressed by 
$$ \bigwedge^{s}_{q=1} \; \left\{ \neg \Ab(\bar{{\mbox{\boldmath$\sigma$}}}_q({\mathbf g})) \Rightarrow  \neg \Ab(\bar{{\mbox{\boldmath$\sigma$}}}_q({\mathbf g}')) \right \}. $$

\end{proof}

We can now prove the two key propositions.

\begin{proof}[Proof of Proposition \ref{PreretractionForA}] By Proposition \ref{MRGammaRelPartial}, there exists a finite subset $H_0$ of $H$, and a finite family of proper quotients $\eta_j: A \to L_j$ for $j \in [1, m]$, such that any non injective morphism $\theta: A \to A$ which fixes $H_0$ factors through one of the quotients $\eta_j$ after precomposition by an element of $\Mod_H(A)$. Proposition \ref{CoHopfRelPartial} on the other hand shows that an injective morphism $\theta: A \to A$ which fixes a big enough finite subset of $H$ is also surjective. 

Now a morphism $\theta: A \to H$ can be seen as a non surjective morphism $A \to A$ since we assumed $H \not = A$. Thus, up to enlarging $H_0$, any morphism $\theta: A \to H$ which fixes $H_0$ is non injective, so there exists an element $\tau$ of $\Mod_H(A)$ such that $\theta'=\theta \circ \tau$ factors through one of the quotients $\eta_j$. 

Let $\Lambda$ be the cyclic JSJ decomposition of $A$ relative to $H$. By Remark \ref{PrecByModIsRelated}, for any morphism $\theta: A \to H$ which fixes $H_0$, there exists a morphism $\theta': A \to H$ which is $\Lambda$-related to $\theta$ and which factors through one of the quotients $\eta_j$.

Finally, for each $l$ in $[ 1, m ]$, we fix an element $\nu_j$ in the kernel of $\eta_j: A \rightarrow Q_l$. 
\begin{statement} \label{Stat1} For any morphism $\theta: A \to H$ which fixes $H_0$, there exists a morphism $\theta': A \to H$ which is $\Lambda$-related to $\theta$, and an index $j$ in $[1, m]$ such that $\theta'(\nu_j)=1$.
\end{statement}

We claim that Statement \ref{Stat1} can be expressed by a first-order formula satisfied by $H$. 

The group $A$ is hyperbolic, we choose a finite presentation $\langle {\mathbf a} \mid  {\bar{\mathbf{\Sigma}}_A}({\mathbf a})\rangle $. Morphisms $A \to H$ are in one to one correspondence with solutions to the system of equations $\bar{{\mathbf \Sigma}}_A({\mathbf x})=1$ in $H$. The elements $\nu_l$ can be represented by words $\bar{\nu}_l({\mathbf a})$. Let $H_0 =\{ h_1, \ldots, h_u\}$, each element $h_i$ can be represented by a word $ \bar{h}_i({\mathbf a})$.

Finally consider the formula $\psi(h_1, \ldots, h_u)$ over ${\cal L}_H$ given by
\begin{eqnarray*} \forall {\mathbf x} \, \left[ \bar{\Sigma}_A({\mathbf x}) = 1 \, \wedge \, \bigwedge^u_{i=1} h_i =\bar{h}_i({\mathbf x}) \right ] \Rightarrow \exists {\mathbf y} \, \left[ \bar{{\mathbf \Sigma}}_A( {\mathbf y})=1 \right] \, \wedge \,
\Rel({\mathbf x}, {\mathbf y})\; \wedge \left [ \bigvee_{l=1}^{m} \bar{\nu}_l({\mathbf y}) =1 \right ].\\
\end{eqnarray*}

The interpretation of the first-order formula $\psi(h_1, \ldots, h_u)$ on $H$ is exactly Statement \ref{Stat1}, so it is true on $H$. The formula $\psi(h_1, \ldots, h_u)$ is therefore satisfied by $G$. Let us look at its interpretation on $G$.

If we take the ``tautological solution'' ${\mathbf a}$ to the equation $\bar{\Sigma}_A({\mathbf x})=1$, it satisfies the formula in the first square brackets: indeed, $\bar{\Sigma}_A({\mathbf a}) = 1$, and for each $i$, we have $h_i = \bar{h}_i({\mathbf a})$ by definition of $\bar{h}_i$. Thus, by the second part of the formula, we get a tuple ${\mathbf y}$ such that ${\mathbf a} \mapsto {\mathbf y}$ extends to a morphism $f: A \to G$, which is $\Lambda$-related to the morphism ${\mathbf a} \mapsto {\mathbf a}$, and which sends one of the elements $\nu_l$ to $1$. By Remark \ref{RelToiIsPreretraction}, $f: A \to G$ is a preretraction, and since the elements $\nu_l$ are non trivial, it is not injective.
\end{proof}

We now show the second key result.

\begin{proof}[Proof of Proposition \ref{PreretractionForB}] We choose a presentation $\langle {\mathbf b} \mid \bar{{\mathbf \Sigma}}_B({\mathbf b}) \rangle$ for $B$. Let $\Lambda$ be the cyclic JSJ decomposition of $B$. 

Since $H$ is a retract of $G$, it is a quasiconvex, thus it is itself hyperbolic. By Proposition \ref{MRGammaFI}, there exist proper quotients $\eta_1: B \to L_1, \ldots, \eta_m: B \to L_m$ of $B$ such that any non injective morphism $B \to H$ factors through one of the maps $\eta_j$ after precomposition by an element of $\Mod(B)$. Again we choose non trivial elements $\nu_1, \dots, \nu_m$ of $B$ such that $\nu_j$ is in the kernel of $\eta_j$.

If we proceed to the same weakenings as in the proof of Proposition \ref{PreretractionForA}, we see that for any non injective morphism $\theta$ from $B$ to $H$, there is another morphism $\theta'$ which is $\Lambda$-related to $\theta$, and which sends one of the $\nu_j$ to $1$. 

We now need to find a sufficient condition for non injectivity of a map $B \rightarrow H$ that is expressible in first-order. Proposition \ref{ShortForEmbeddings}, applied to $B$ and $H$, tells us that there exist a finite set $i_1, \dots i_t$ of embeddings of $B$ in $H$ such that for any embedding $i: B \hookrightarrow H$, there exists an element $\sigma$ of $\Mod(B)$, an integer $k$ in $[1 , t]$ and an element $h$ of $H$ such that
$$ \Conj(h) \circ i \circ \sigma = i_k$$ 
By Remark \ref{PrecByModIsRelated}, the map on the left hand side is $\Lambda$-related to $i$, so that if $\theta: B \to H$ is an embedding, it is $\Lambda$-related to one of the maps $i_k$. 

Let $\theta: B \rightarrow H$. Consider the following statement 
\begin{stheta}
Any morphism $\theta'': B \to H$ which is $\Lambda$-related to $\theta$ satisfies $\theta''({\mathbf b}) \neq  i_k({\mathbf b})$ for all $k$ in $[1 , t]$. 
\end{stheta}
By the previous argument, $S(\theta)$ is a sufficient condition for $\theta$ not to be an embedding. We finally get that the following statement holds. 
\begin{statement}\label{FinalStat} If $\theta: B \to H$ is a morphism for which $S(\theta)$ holds, then there exists a morphism $\theta': B \to H$ which is $\Lambda$-related to $\theta$, and an integer $l$ in $[1, m]$ such that $\theta'(\nu_l)=1$.
\end{statement}

We claim that this can be expressed by a first-order formula. Let us first express $S(\theta)$ by a first-order formula:
$$ \psi({\mathbf x}, i_1({\mathbf b}), \ldots, i_1({\mathbf b})): \; \left [ \bar{{\mathbf \Sigma}}_B({\mathbf x})=1\right ] \wedge  \forall {\mathbf z} \left \{ \left[ \bar{{\mathbf \Sigma}}_B( {\mathbf z})=1 \wedge \Rel({\mathbf z}, {\mathbf x}) \right] \Rightarrow \left[ \bigwedge^t_{k=1} {\mathbf z} \neq i_k({\mathbf b}) \right] \right \}.$$

Let ${\mathbf x}$ be a tuple in $H$. The formula $ \psi({\mathbf x}, i_1({\mathbf b}), \ldots, i_1({\mathbf b}))$ is satisfied by $H$ if and only if the morphism $\theta: B \to H$ defined by ${\mathbf b} \mapsto {\mathbf x}$ satisfies $S(\theta)$. So if $ \psi({\mathbf x}, i_1({\mathbf b}), \ldots, i_1({\mathbf b}))$ is satisfied by $H$, the morphism $\theta$ defined by ${\mathbf b} \mapsto {\mathbf x}$ is not injective.

We can now write the first-order sentence $\xi(i_1({\mathbf b}), \ldots, i_1({\mathbf b}))$: 
\begin{eqnarray*}
\forall {\mathbf x} \; \;   \psi({\mathbf x}, i_1({\mathbf b}), \ldots, i_1({\mathbf b})) \;\Rightarrow \;\exists {\mathbf y} \, [\bar{{\mathbf \Sigma}}_B(y)=1] \wedge \, \Rel({\mathbf x},{\mathbf y})
\; \wedge \; \left[ \bigvee_{j=1}^{l} \bar{\nu}_j({\mathbf y}) =1 \right ].
\end{eqnarray*}

The first-order formula $\xi(i_1({\mathbf b}), \ldots, i_1({\mathbf b}))$ on $H$ expresses Statement \ref{FinalStat}, so it is satisfied by $H$. As $H$ is elementarily embedded in $G$, it is also satisfied by $G$. As in the proof of \ref{PreretractionForA}, we can apply it to the tautological solution ${\mathbf b}$ of $\bar{{\mathbf \Sigma}}_B({\mathbf x})=1$. To see that $G \models \psi({\mathbf b}, i_1({\mathbf b}), \ldots, i_1({\mathbf b}))$, note first that since we assumed that $B$ is not the fundamental group of a closed surface of characteristic at most $-2$, the JSJ decomposition of $B$ admits at least one non surface type vertex group $R$. A morphism $\mu: B \to G$ which is $\Lambda$-related to the embedding ${\mathbf b} \mapsto {\mathbf b}$ restricts to conjugation on $R$, thus $\mu(R)$ cannot lie in $H$ since by hypothesis, no element of $B$ can be conjugated into $H$ by an element of $G$. This implies in particular that for all $k$, the tuple $\mu({\mathbf b})$ is distinct from the tuple $i_k({\mathbf b})$, so that $G \models \psi ({\mathbf b})$.

The second part of the sentence $\xi(i_1({\mathbf b}), \ldots, i_1({\mathbf b}))$ thus gives a morphism $B \to G$ which is $\Lambda$-related to the embedding $B \hookrightarrow G$ and kills one of the elements $\nu_i$: it is a non injective preretraction. 
\end{proof}

\subsection{Proof of the main result}

Assuming we have Proposition \ref{RETRACTION} and Proposition \ref{RETRACTIONPLUS}, we can now prove Theorem \ref{MainResult}.

\begin{proof}[Proof of Theorem \ref{MainResult}] Note first that if $G$ is infinite cyclic, its only elementarily embedded subgroup is itself, and the theorem is trivial. Assume thus that $G$ is a non cyclic torsion-free hyperbolic group, and let $H$ be an elementary subgroup of $G$. Note that $H$ is not abelian since $G$ is not.

We will first show that $G$ admits a structure of hyperbolic tower over a group $G'$ of the form $G'=H* B'_1* \ldots *B'_r$. Once this is done, we will show that for each $i$, the group $B'_i$ has a structure of extended hyperbolic tower over $\{1\}$, and this will give the result by Remark \ref{TransitivityOfHypTowers}.

Set $G^0=G$. We define by induction a finite sequence $G=G^0 > G^1 > \ldots > G^N$ of subgroups of $G$, such that $H$ is a free factor of $G^N$, and $G^m$ has a structure of hyperbolic floor over $G^{m+1}$ for each $m$ up to $N-1$.

Assume $G^m$ is defined, and write the Grushko decomposition of $G^m$ relative to $H$ as 
$$G^m = A^m * B_1^m * \ldots * B^m_{p_m}$$
where $A^m$ is the factor containing $H$. If $A^m=H$ we are done, so assume $A^m \neq H$.

Note that $A^m$ is freely indecomposable relative to $H$, and that, as a retract, it is a quasiconvex subgroup of $G$, and thus it is hyperbolic. Denote by $\Lambda$ the cyclic JSJ decomposition of $A^m$ relative to $H$.

All the hypotheses of Proposition \ref{PreretractionForA} for $A=A^m$ are satisfied, so we can apply it to get a non injective preretraction $A^m \to G$ with respect to $\Lambda$. Note that we do not use the fact that $A^m$ is elementarily embedded in $G$: according to Sela, the converse of Theorem \ref{MainResult} holds, so that this is true \cite{SelaPrivate}. However, we do not need this for the induction step (or indeed at any point in the proof), we only use the fact that the embedding of $H$ in $G$ is elementary.

We now apply Proposition \ref{RetractionPlus} successively to the floors of the hyperbolic tower formed by $G$ over $G^m$, and to the free product $G^m = A^m * (B_1^m * \ldots * B^m_{p_m})$, to get a non injective preretraction $A^m \to A^m$ with respect to $\Lambda$. Thus by Proposition \ref{Retraction}, we get a retraction $r: A^m \rightarrow A_0^m$ on a proper subgroup of $A^m$ such that $(A^m, A_0^m, r)$ is an extended hyperbolic floor, and the rigid group of $\Lambda$ which contains $H$ is in $A_0^m$. Since $H$ and thus $A^m_0$ is not abelian, this is in fact a hyperbolic floor.
Now define $G^{m+1}$ by 
$$G^{m+1}= A^m_0 * B^m_1 * \ldots * B^m_{p_m}.$$
Since $A^m$ has a structure of hyperbolic floor over $A^m_0$, the group $G^m$ has a structure of hyperbolic floor over $G^{m+1}$ as required.

For each $m$, the group $G^{m+1}$ is a strict retract of $G^m$, and since the groups $G^m$ are all subgroups of $G$, they are $G$-limit groups. Thus by Proposition \ref{DecSeqOfGammaLimit} the sequence is finite. At the end of this process, we get a group $G^N$ in which $H$ is a free factor, and such that $G$ has a structure of hyperbolic tower over $G^N$.

If all the other factors of the Grushko decomposition of $G^N$ relative to $H$ are surface groups of characteristic at most $-2$ or free groups, we are done. So assume that there is a factor $B$ which is neither free nor a closed surface group of characteristic at most $-2$. 
Note that as a retract of $G$, the group $G^N$ is hyperbolic, so as a free factor of $G^N$, the group $B$ is itself hyperbolic. We will now show that $B$ has a structure of extended hyperbolic tower over $\{1\}$. 

Both $H$ and $B$ are free factors of $G^N$, thus any two of their conjugates in $G^N$ intersect trivially. But since $G^N$ is a retract of $G$, any two conjugates of $H$ and $B$ in $G$ must also intersect trivially. Hence the conditions of Proposition \ref{PreretractionForB} are satisfied by $B$: by applying it, we get a non injective preretraction $B \to G$. We apply Proposition \ref{RetractionPlus} iteratively to get a non injective preretraction $B \to B$: by Proposition \ref{Retraction} there is a retraction $r$ from $B$ to a proper subgroup $B'$, such that $(B, B' , r)$ is an extended hyperbolic floor. 

Note that since $B'$ is a retract of $G$, the number of factors in its Grushko decomposition is bounded above by the rank of $G$. If any of the factors of the Grushko decomposition of $B'$ are neither free nor surface, we can repeat the process above. This terminates, as before, because the groups involved are $G$-limit groups and because the number of factors in the Grushko decomposition of our groups is bounded. We finally get that $B$ is an extended hyperbolic tower over $\{1\}$. 

Thus $G$ is an extended hyperbolic tower over $H$. As we noted in Definition \ref{HypTower}, since $H$ is not abelian, $G$ is in fact a hyperbolic tower over $H$.
\end{proof}

\subsection{The case of free groups}

In the special case where our hyperbolic group is free, we can use Theorem \ref{MainResult} to show Theorem \ref{CaseFreeGroups}. To do so, we prove first

\begin{lemma} \label{FreeHaveNoFloors} Let $F$  be a free group which admits a structure of hyperbolic tower over a subgroup $H$. Then $H$ is a free factor of $F$.
\end{lemma}

\begin{proof} If the tower has at least one floor, there exists a subgroup $F'$ of $F$, and a retraction $r: F \to F'$ so that $H \leq F'$, and $(F, F', r)$ is a hyperbolic floor built by amalgamating some surface groups with boundary to free factors of $F'$. Pick a maximal set $\{ \gamma_1, \ldots \gamma_r \}$ of pairwise non conjugate maximal boundary elements of these surface groups. We know, from the standard presentation of a surface group with boundary, that the product of the elements $\gamma_i$ is equal in $F$ to a product of commutators and squares. 

By Lemma 4.1 in \cite{BestvinaFeighnOuterLimits}, since both $F$ and $F'$ are free groups,  there is a decomposition of $F'$ as $Z * F''$ where $Z$ is an infinite cyclic group generated by one of these maximal boundary elements (say $\gamma_1$), and all the other maximal boundary elements $\gamma_i$ lie in conjugates of $F''$. Now let $\alpha: F' \to Z/2Z $ be the map which kills $F''$ and the squares in $Z$. The image of $\gamma_1$ by $\alpha \circ r$ is the generator of $Z/2Z$, and for $i \not =1$, the image of $\gamma_i$ is trivial. However, squares and commutators of $F$ have trivial image by $\alpha \circ r$, so we get a contradiction. This shows that the only structure of hyperbolic tower that a free group can have over one of its subgroup is a trivial one, where the subgroup is a free factor of the free group.
\end{proof}

\begin{proof}[Proof of Theorem \ref{CaseFreeGroups}] We only need to show that an elementary subgroup of a free group is a non abelian free factor, since the converse is given by Theorem 4 in \cite{Sel6}.

By Theorem \ref{MainResult}, if $H$ is an elementary subgroup of $F$, then $F$ has a structure of hyperbolic tower over $H$. By Lemma \ref{FreeHaveNoFloors}, $H$ is a free factor in $F$.
\end{proof}

\subsection{The case of surface groups}

All the surfaces we consider are compact and connected.

\begin{defi} \emph{(minor subsurface)} Let $\Sigma$ be a closed hyperbolic surface. Let $\Sigma_0$ be a proper connected subsurface of $\Sigma$ which is not a disk. Denote by $\Sigma^c_0$ the closure of $\Sigma \setminus \Sigma_0$. We say that $\Sigma_0$ is a minor subsurface of $\Sigma$ if 
\begin{itemize}
	\item $\Sigma^c_0$ is connected and is either a punctured torus, or has characteristic at most $-2$;
  \item $|\chi(\Sigma_0)| \leq |\chi(\Sigma)|/2$, with equality if and only if $\Sigma_0$ and $\Sigma^c_0$ are homeomorphic; 
	\item if $\Sigma_0$ is not orientable, neither is $\Sigma^c_0$.
\end{itemize}
\end{defi}

\begin{rmk} Note that if $\Sigma$ is the connected sum of three or four projective planes, it does not admit any minor subsurface.
\end{rmk}

We will show
\begin{thm} \label{SurfaceCaseBis} Let $S$ be the fundamental group of a closed hyperbolic surface $\Sigma$. If $H$ is a proper elementary subgroup of $S$, then it is a non abelian free factor of the fundamental group of a minor subsurface of $\Sigma$.
\end{thm}

Note that Theorem \ref{SurfaceCaseIntro} that we gave in the introduction is a direct consequence of Theorem \ref{SurfaceCaseBis}. If $\Sigma$ is not the connected sum of four projective planes, Theorem \ref{SurfaceCaseBis} follows immediately from Theorem \ref{MainResult}, which tells us that $S$ has a structure of hyperbolic tower over $H$, and from one direction of the following equivalence result:
\begin{prop} \label{HypFloorMinorSubsurface} Let $S$ be the fundamental group of a closed hyperbolic surface $\Sigma$ which is not the connected sum of four projective planes, and let $H$ be a non trivial proper subgroup of $S$. Then $S$ admits a structure of hyperbolic tower over $H$ if and only if $H$ is a free factor of the fundamental group of a minor subsurface of $\Sigma$.
\end{prop}

If $\Sigma$ is the connected sum of four projective planes, the only structures of hyperbolic floors it admits over a proper subgroup are over the trivial group, or hyperbolic floor structures over cyclic subgroups. Thus $S$ has no proper elementary subgroups. Since $\Sigma$ has no proper minor subsurfaces, the conclusions of Theorem \ref{SurfaceCaseBis} hold in this case as well. 

There remains to prove Proposition \ref{HypFloorMinorSubsurface}. Suppose $S$ has a structure of hyperbolic tower over a non trivial proper subgroup $H$. Note that $S$ is freely indecomposable, so if $H \neq S$, the tower has at least one hyperbolic floor $(S, R, r)$. The subgroup $R$ contains $H$, and has itself a structure of hyperbolic tower over $H$. Denote by $\Gamma$ the graph of groups decomposition of $S$ and by $R_0, \ldots, R_k$ the free factors of the decomposition of $R$ associated to the hyperbolic floor structure $(S, R, r)$. We may assume $H \leq R_0$. Choose representatives $S_1, \ldots, S_l$ of the conjugacy classes of surface type vertex groups of $\Gamma$. 

By  Theorem III.2.6 of \cite{MorganShalen}, the splitting $\Gamma$ of $S$ is geometric, namely it is dual to a set of non null-homotopic simple closed curves on the surface $\Sigma$. Thus each of the groups $S_i$, $R_j$ is the fundamental group of a connected subsurface of $\Sigma$ that we denote by $\Sigma_i$, $\Xi_j$ respectively. 

Let us first prove Proposition \ref{HypFloorMinorSubsurface} in a particular case
\begin{lemma} \label{SurfaceCaseNonPinching} Suppose that for some $i$, the map $r|_{S_i}$ is non pinching with respect to $\Sigma_i$ (recall Definition \ref{NonPinching}). Then $\Xi_0$ is a minor subsurface of $\Sigma$.
\end{lemma}

\begin{proof} First we show that $r(S_i)$ lies in a conjugate of one of the subgroups $R_j$ in $R$. 
Indeed, consider the $R$-tree $T_*$ with trivial edge stabilizers corresponding to the free factor decomposition $R= R_0 * \ldots * R_k$. The group $S_i$ acts on this tree via $r$, and in this action, its boundary elements are elliptic: by Lemma \ref{CuttingLemma}, we get a set of simple closed curves ${\cal C}$ on $\Sigma_i$ whose corresponding element stabilize edges of $T_*$, and thus have trivial image by $r$. Since $r|_{S_i}$ is non pinching, ${\cal C}$ is empty and $r(S_i)$ is elliptic in $T_*$. This proves the claim. Up to replacing $S_i$ by a conjugate, we may assume $r|_{S_i}: S_i \to R_j$.

Note also that $r$ sends non conjugate maximal boundary elements of $S_i$ to non conjugate maximal boundary elements of $R_j$ since it is a retraction. Lemma \ref{HomeoLemma} then implies that $r|_{S_i}$ is an isomorphism of surface groups between $S_i$ and $R_j$. 

This implies that $\Sigma_i$ and $\Xi_j$ have the same number of boundary components: by connectedness of $\Sigma$, this implies that the complement of $\Xi_j$ in $\Sigma$ is exactly $\Sigma_i$: we thus have $\Xi_j=\Xi_0$ and $\Sigma_i = \Sigma_1$. Note that by the restrictions on surfaces in the definition of a hyperbolic tower, $\Sigma_1$ is either a punctured torus, or it has characteristic at most $-2$.

Now we have $\chi(\Xi_0) = \chi(\Sigma_1)$ and $\chi(\Sigma) =\chi(\Xi_0) + \chi(\Sigma_1)$, so that $\chi(\Xi_0)= \chi(\Sigma)/2$. 
\end{proof}

If none of the maps $r|_{S_i}$ are non pinching, the idea to prove Proposition \ref{HypFloorMinorSubsurface} is to write $r|_{S_1}$ as $r' \circ \rho$, where $r'$ is non pinching with respect to some suitable surface groups contained in $\rho(S_1)$. 

\begin{defi} \label{PinchingMapSurf} \emph{(pinching map $\rho_{\cal C}$)} Let $S_1$ be the fundamental group of a surface $\Sigma_1$. Let ${\cal C}$ be an essential set of curves on $\Sigma_1$. Denote by $N({\cal C})$ the subgroup of $S_1$ normally generated by elements corresponding to the curves of ${\cal C}$. The pinching map $\rho_{\cal C}$ is the quotient map $S_1 \to S_1/N({\cal C})$.
\end{defi}
The pinching map $\rho_{\cal C}$ is injective on each boundary subgroup. If $\Delta(S_1, {\cal C})$ is the graph of groups decomposition of $S_1$ dual to ${\cal C}$, its edge groups have trivial image by $\rho_{\cal C}$, and its vertex group are quotiented by some of their boundary subgroups. 

\begin{defi}\emph{(graph of groups $\Gamma(S_1, {\cal C})$)} \label{PinchingDecSurf} Let $S_1$ be the fundamental group of a surface $\Sigma_1$. Let ${\cal C}$ be an essential set of curves on $\Sigma_1$. We denote $\Gamma(S_1, {\cal C})$ the graph of group decomposition of $\rho_{\cal C}(S_1)$ obtained by replacing, in the decomposition $\Delta(S_1, {\cal C})$ of $S_1$ dual to ${\cal C}$, each vertex and edge group by its image by $\rho_{\cal C}$. 
\end{defi}

\begin{defi} \label{Interior} \emph{(interior and exterior surfaces of $\Gamma(S_1, {\cal C})$)} A vertex of $\Gamma(S_1, {\cal C})$ is said to be interior if its fundamental group does not contain the image by $\rho_{\cal C}$ of a boundary subgroup of $S_1$. Otherwise it is said to be exterior. By extension, a vertex of the tree $T_{\Gamma}$ corresponding to $\Gamma(S_1, {\cal C})$ is said to be interior (respectively exterior) if its image by the quotient map $T_{\Gamma} \to \Gamma(S_1, {\cal C})$ is interior (respectively exterior).
\end{defi}

\begin{figure}[!ht]
\begin{center}
\input{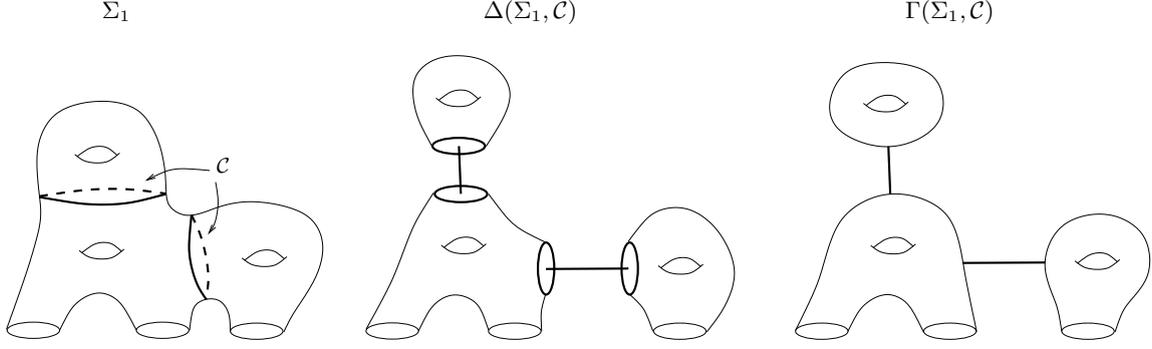}
\caption{The construction of the graph of groups $\Gamma(S_1,{\cal C})$.}
\label{Gamma}
\end{center}
\end{figure}

Note that a vertex group $S^{\Delta}_1$ of $\Delta(S_1, {\cal C})$ is the fundamental group of a subsurface $\Sigma^{\Delta}_1$ of $\Sigma_1$. The image of $S^{\Delta}_1$ by $\rho_{\cal C}$ is the fundamental group of the surface $\Sigma^{\Gamma}_1$ obtained by gluing discs to the boundary components of $\Sigma^{\Delta}_1$ corresponding to curves of ${\cal C}$ (see Figure~\ref{Gamma}). If $\Sigma^{\Gamma}_1$ corresponds to an interior vertex of $\Gamma(S_1, {\cal C})$, all the boundary components of $\Sigma^{\Delta}_1$ correspond to curves of ${\cal C}$, so the image of $S^{\Delta}_1$ by $\rho_{\cal C}$ is the fundamental group of a closed surface. On the other hand, exterior vertex groups are fundamental groups of surfaces with boundary.

\begin{rmk} \label{PinchingComplexity} If $\Sigma^{\Gamma}_1$ is a surface corresponding to a vertex of $\Gamma(S_1, {\cal C})$, then $\chi(\Sigma^{\Gamma}_1)$ is at least $\chi(\Sigma_1)$, with equality if and only if ${\cal C}$ is empty.
\end{rmk}

We can apply this construction to find a factorization of $f$.
\begin{defi} \emph{(essential set of curves pinched by $f$)} \label{KilledByFSurf} Let $S_1$ be the fundamental group of a surface $\Sigma_1$. Let $f: S_1 \to G$ be a group morphism whose restriction to each boundary subgroup is injective.

Let $\gamma$ be the homotopy class of a simple closed curve on $\Sigma$. If the elements of $S_1$ corresponding to $\gamma$ lie in the kernel of $f$, we say $\gamma$ is pinched by $f$. Let ${\cal C}$ be an essential set of curves. If all its elements are pinched by $f$, we say it is an essential set of curves pinched by $f$.
\end{defi}

\begin{rmk} \label{FactorsAsNonPinching} Suppose $f: S_1 \to G$ is injective on boundary subgroups of $S_1$. If ${\cal C}$ is a maximal essential set of curves pinched by $f$ on $\Sigma_1$, the map $f$ factors as $f' \circ \rho_{\cal C}$, and $f'$ is non pinching with respect to the surfaces corresponding to vertices of $\Gamma(S_1, {\cal C})$.
\end{rmk}

We can now prove Proposition \ref{HypFloorMinorSubsurface}.
\begin{proof}[Proof of Proposition \ref{HypFloorMinorSubsurface}] Suppose $H$ is a free factor of a minor subsurface $\Sigma_0$ of $\Sigma$, and denote by $S_0$ the fundamental group of $\Sigma_0$. If $\Sigma_0$ is a cylinder or a Mobius band, we can find a subsurface $\Sigma_1$ containing $\Sigma_0$ such that $\Sigma$ is either the double of $\Sigma_1$, or the double of $\Sigma_1$ connected sum with a projective plane. We can then easily define a retraction $r$ of $S$ on the fundamental group $S_1$ of $\Sigma_1$, it gives us the required hyperbolic floor $(S, S_1, r)$. 

Thus we may assume that $S_0$ is not abelian. Denote by $\Sigma_1$ the closure of $\Sigma - \Sigma_0$. By definition, $\Sigma_1$ is either a punctured torus, or has characteristic at most $-2$. Since $\Sigma_1$ has the same orientability as $\Sigma_0$, the same number of boundary components, and smaller or equal characteristic, we can find an essential curve $\gamma$ on $\Sigma_1$ such that the graph of groups $\Gamma(S_1, \{\gamma \})$ has a unique exterior surface $\Sigma_1^{\Gamma}$, which is homeomorphic to $\Sigma_0$ via some homeomorphism $\phi$ restricting to the identity on $\partial \Sigma_1$.

Consider the map defined to be the identity on $S_0$, and $\phi_*$ on the vertex group $S^{\Gamma}_1$ corresponding to $\Sigma^{\Gamma}_1$: it can be extended to a retraction $r: \rho_{\cal C}(S) \to S_0$.
Then $r=r' \circ \rho_{\cal C}$ is a well-defined retraction $S \to S_0$ which makes $(S, S_0, r)$ a hyperbolic floor. Thus $S$ is a hyperbolic tower over $H$.

Conversely, suppose that $S$ is a hyperbolic tower over $H$. Since $S$ is freely indecomposable, $S$ admits a structure of hyperbolic floor $(S, R, r)$ with corresponding graph of groups decomposition $\Lambda$. We know that $H$ is contained in a non surface type vertex stabilizer $R_0$ of $T_{\Lambda}$, and that $R_0$ is a hyperbolic tower based on $H$. Let $S_1$ be a surface type vertex stabilizer of $T_{\Lambda}$ adjacent to $R_0$.

By Theorem III.2.6 of \cite{MorganShalen}, $\Lambda$ is dual to a set of simple closed curves on $\Sigma$. Denote by $\Sigma_0$ and $\Sigma_1$ the subsurfaces of $\Sigma$ corresponding to $R_0$ and $S_1$ respectively, so that both $R_0$ and $S_1$ are endowed with a structure of surface group with boundary. In particular, $R_0$ is free. It is enough to show that $\Sigma_0$ is a minor subsurface of $\Sigma$, since by Lemma \ref{FreeHaveNoFloors} $H$ is a free factor of $R_0$.

Note that the intersection of $R_0$ and $S_1$ is an infinite cyclic subgroup $Z$.
The morphism $r_1 = r|_{S_1}: S_1 \to R$ restricts to the identity on $Z$. Let ${\cal C}$ be a maximal essential set of curves on $\Sigma_1$ pinched by $r_1$. The map $r_1$ factors as $r'_1 \circ \rho_{\cal C}$. Let $S^{\Gamma}_1$ be the exterior surface type vertex group of $\Gamma(S_1, {\cal C})$ containing $Z$, denote by $\Sigma^{\Gamma}_1$ the corresponding surface. Then  $r'_1|_{S^{\Gamma}_1}$ is non pinching with respect to $\Sigma^{\Gamma}_1$, has image in $R= R_0* \ldots* R_m$, and sends boundary elements of $S$ to conjugates of elements of the subgroups $R_j$: the image of $S^{\Gamma}_1$ by $r'_1$ is contained in a conjugate of a subgroup $R_j$, and since it contains $Z$, it is contained in $R_0$. Now $r'_1$ sends non conjugate maximal boundary elements of $S^{\Gamma}_1$ to non conjugate maximal boundary elements of $R_0$. We can thus apply Lemma \ref{HomeoLemma} to conclude that $r'_1|_{S^{\Gamma}_1}$ is an isomorphism of surface groups between $S^{\Gamma}_1$ and $R_0$. 

In particular, $\Sigma_0$ and $\Sigma^{\Gamma}_1$ have the same number of boundary components. This implies that the complement $\Sigma^c_0$ of $\Sigma_0$ in $\Sigma$ is connected.  

Since $\Sigma_0$ and $\Sigma^{\Gamma}_1$ are homeomorphic, we have $\chi(\Sigma_0)=\chi(\Sigma^{\Gamma}_1)$. On the other hand, by Remark \ref{PinchingComplexity}, we have $\chi(\Sigma^{\Gamma}_1) \geq \chi(\Sigma_1)$, with equality if and only if ${\cal C}$ is empty, in which case we have $\Sigma^{\Gamma}_1=\Sigma_1$ and $\Sigma_0$ and $\Sigma_1$ are homeomorphic. Now $\chi(\Sigma_1) \geq \chi(\Sigma^c_0)$, and if we have equality, $\Sigma_1=\Sigma^c_0$. On the other hand we have $\chi(\Sigma) =  \chi(\Sigma_0) + \chi(\Sigma^c_0)$, so $|\chi(\Sigma_0)| \leq |\chi(\Sigma)|/2$, and if we have equality, $\Sigma_0$ and $\Sigma^c_0$ are homeomorphic. Finally, if $\Sigma_0$ is not orientable, neither is $\Sigma^{\Gamma}_1$ and thus neither are $\Sigma_1$ and $\Sigma^c_0$. Thus $\Sigma_0$ is a minor subsurface of $\Sigma$.
\end{proof}

\section{A property of JSJ-like decompositions} \label{IsoSec}
In this section, we show that if a preretraction $G \to G$ relative to a cyclic JSJ-like decomposition of $G$ satisfies some strong injectivity conditions on the vertex groups, it must be an isomorphism. This will be used in the last section to prove Propositions \ref{Retraction} and \ref{RetractionPlus}. 

\begin{prop} \label{IsoOnVertexIsIsoOnGroup} Let $G$ be a torsion-free hyperbolic group, and let $\Lambda$ be a JSJ-like decomposition of $G$. Let $\theta: G \rightarrow G$ be a preretraction with respect to $\Lambda$ which sends surface type vertex groups of $\Lambda$ isomorphically to conjugates of themselves. Then $\theta$ is an isomorphism.
\end{prop}

Call $Z$ type vertices the vertices of $\Lambda$ which have infinite cyclic vertex group, and call rigid type vertices the vertices which are neither of Z type, nor of surface type. We will also say that a vertex in the tree $T_{\Lambda}$ is of type $Z$ or rigid according to the type of its image by the quotient map $T_{\Lambda} \to \Lambda $. 

\begin{proof} First note that if $G$ is cyclic, the only JSJ-like decomposition it admits is the trivial one, for which the result is immediate. We may thus assume that $G$ is not cyclic. 

Denote by $T$ the Bass-Serre tree $T_{\Lambda}$ corresponding to $\Lambda$. To prove the proposition, we will construct a bijective simplicial map $j: T \rightarrow T$, such that $j$ is equivariant with respect to $\theta$.

The stabilisers of an edge $e$ of $T$ and of a vertex $v$ of $T$ in the standard action of $G$ on $T$ are denoted by $G_e$ and $G_v$ respectively. 

\paragraph{1. Construction of the map $j$ on vertices.} By hypothesis, for each vertex $v$ of $T$, there is an element $g_v$ of $G$ such that $\theta(G_v) = g_v G_v g_v^{-1}$. We set the image of $v$ by $j$ to be $g_v \cdot v$. Its stabiliser is exactly $\theta(G_v)$, and since distinct vertices of the tree corresponding to a JSJ-like decomposition have distinct stabilisers, this defines $j(v)$ uniquely.
The image of $g \cdot v$ by $j$ is the unique vertex whose stabiliser is $\theta(g)\theta(G_v)\theta(g^{-1})$, namely $\theta(g) \cdot j(v)$, and the map $v \mapsto j(v)$ is equivariant with respect to $\theta$. Note that $j(v)$ is in the orbit of $v$, and thus is of the same type.  Note also that $G_{j(v)} = \theta(G_v) \simeq G_v$.

\paragraph{2. The map $v \mapsto j(v)$ can be extended to a simplicial map $j: T \to T$.} We need to check that adjacent vertices are sent on adjacent vertices. Suppose $v,w$ adjacent: without loss of generality $G_v$ is not a surface type vertex group. The image $C = g_v (G_v \cap G_w) g_v^{-1}$ of $G_v \cap G_w$ by $\theta$ is non trivial and fixes both $j(v)$ and $j(w)$, thus $j(v)$ and $j(w)$ are at a distance at most $2$ by $2$-acylindricity. We will first show that it cannot be $2$, then that it cannot be $0$.

\begin{itemize}
\item Assume the distance is $2$. 
The vertex $u$ between $j(v)$ and $j(w)$ is a $Z$ type vertex, sot $j(v)$,$j(w)$,$v$ and $w$ are not $Z$ type vertex. Note that $C$ fixes the vertex $g_v \cdot w$, which is at a distance $1$ from $j(v)$, thus it is distinct from $j(v)$ and from $j(w)$. Its stabiliser is not cyclic, thus it is distinct from $u$. This is a contradiction. 

\item Assume now $j(v)=j(w)$, so $v$ and $w$ are in the same orbit (in particular they must be of rigid type, since they are adjacent). Let $a \in G$ be such that $w = a \cdot v$. We have $G_w=aG_va^{-1}$, and $a \notin G_v$.
We see that $\theta(a) \in \theta(G_v)$, since $\theta(a)$ stabilizes $j(v)$, so there exists $a' \in G_v$ such that $\theta(a')=\theta(a)$. 

Let $C_1:=G_v \cap G_w$, and let $C_2 \leq G_v$ be such that $C_1= aC_2a^{-1}$. Note that $C_1$ is maximal abelian in $G$ since it is the stabiliser of an edge which connects two rigid vertices. 
Now $\theta(C_2) = \theta(a^{-1}) \theta(C_1) \theta(a)$ so that $\theta(C_2)=\theta(a'^{-1} C_1 a')$. By injectivity of $\theta$ on $G_v$, $C_2=a'^{-1} C_1 a'$. Thus we have $a'a^{-1} \in C_1 \leq G_v$. But $a' \in G_v$ so we deduce $a \in G_v$. This is a contradiction.   
\end{itemize}
Thus we can extend $v \mapsto j(v)$ to a simplicial map $j: T \to T$.
 
\paragraph{3. Injectivity of j.} It is enough to show that there are no foldings, i.e. that no two edges adjacent to a same vertex are sent to the same edge by $j$. Suppose that two vertices $w, w'$ of $T$ are adjacent to a vertex $v$, and that the edges $e$ with enpoints $v,w$ and $e'$ with endpoints $v,w'$ have the same image by $j$. Let $g_e$, $g_{e'}$ be generators of the stabilisers $G_e$ of $e$, and $G_{e'}$ of $e'$ respectively.

Both $\theta(g_e)$ and $\theta(g_{e'})$ fix $j(e)$, so they commute. As $\theta$ is injective on $G_v$, the elements $g_e$ and $g_{e'}$ of $G_v$ also commute. Thus they have a common power which fixes both $e$ and $e'$: by strong $2$-acylindricity, $v$ is a $Z$ type vertex. This implies that $w$, $w'$, and $j(w)$ are not type $Z$ vertices.

There exists an element $\gamma$ of $G$ such that $w'= \gamma \cdot w$, and $\gamma$ does not lie in $G_w$. Now  $\theta(\gamma)$ stabilises $j(w)$ thus there is an element $a$ of $G_w$ such that $\theta(a)=\theta(\gamma)$. 

Let $g$ be an element of $G_v$ which stabilise both $e$ and $e'$: then $g$ is both in $G_w$ and in $\gamma G_w \gamma^{-1}$. Let $g'$ be an element of $G_w$ be such that $g = \gamma g' \gamma^{-1}$. We have 
\begin{eqnarray*}
\theta(g) &=& \theta(\gamma) \theta(g') \theta(\gamma^{-1}) \\
            &=& \theta(a) \theta(g') \theta(a^{-1}) = \theta(a g' a^{-1}).
\end{eqnarray*}
Since $\theta$ is injective on $G_w$, we deduce that $g = a g' a^{-1}$ so $g' =\gamma^{-1} g \gamma = a^{-1} g a$. This shows $[\gamma a^{-1},g] = 1$, so $\gamma a^{-1}$ preserves the set $\Fix(g)$ of fixed point of $g$. But $\Fix(g)$ has diameter $2$ and is centred on $v$, so $\gamma a^{-1}$ fixes $v$, and $\gamma a^{-1}$ lies in $G_v$. Now $a$ was chosen so that $\theta(\gamma)=\theta(a)$, so $\theta(\gamma a^{-1})=1$. By injectivity of $\theta$ on $G_v$, we get $\gamma=a$. This is a contradiction since $\gamma$ is not in $G_w$, but $a$ is.

\paragraph{4. Injectivity of $\theta$.} The injectivity of $j$ implies the injectivity of $\theta$.

\paragraph{5. Surjectivity of j.} We prove this by showing that if a vertex $v$ is in the image of $j$, all the edges adjacent to $v$ are also in the image.
Let $g^v$ be an element of $G$ such that $j(g^v \cdot v)=v$. Pick $e_1, \ldots , e_r$ some representatives of the orbits of edges adjacent to $v$. The image $e'_k$ of $g^v \cdot e_k$ by $j$ must be adjacent to $v$. 

We claim that if $e_k$ and $e_l$ lie in different orbits, so do $e'_k$ and $e'_l$. Indeed, if there exists $\alpha$ in $G_v$ such that $\alpha \cdot e'_k=e'_l$, $\alpha$ must fix $v$. As $v$ is in the image of $j$, its stabiliser is in the image of $\theta$ so there exists an element $a$ of $G$ such that $\theta(a)=\alpha$, and by equivariance of $j$ we get $j( ag^v \cdot e_k) = j(g^v \cdot e_l)$. By injectivity of $j$ this means $e_k$ and $e_l$ are in the same orbit: this proves the claim. Thus the edges $e'_k$ form a system of representative of the orbits of edges adjacent to $v$.  

Now let $e$ be an edge adjacent to $v$: there is an element $\beta$ of $G$ such that $\beta \cdot e'_k = e$ for some $k$, and $\beta$ must fix $v$. We know $G_v$ is in the image of $\theta$ so $\beta = \theta(b)$ for some $b$. Thus $j(b \cdot (g^v \cdot e_k))=\theta(b) \cdot j(g^v \cdot e_k) = \beta \cdot e'_k= e$, so $e$ is in the image of $j$. Hence all the vertices which neighbour $v$ are in the image of $j$. This local surjectivity condition implies global surjectivity of $j$.

\paragraph{6. Surjectivity of $\theta$.} Let $g$ be an element of $G$, and let $v$ be a vertex of $T$ with non cyclic stabiliser. By surjectivity of $j$ there exists $w$ such that $j(w)=v$, and $w'$ such that $j(w')= g \cdot v$, so $G_{w'}=hG_wh^{-1}$ for some $h$. We see that 
$$g G_v g^{-1} = G_{g \cdot v}= \theta(G_{w'}) = \theta(h) \theta(G_w) \theta(h^{-1})=\theta(h) G_v \theta(h^{-1}).$$ We get $G_v= g^{-1} \theta(h)G_v \theta(h)^{-1}g$. Thus $G_v$ stabilises both $v$ and $g^{-1} \theta(h) \cdot v$. Since $G_v$ is not cyclic, $v=g^{-1} \theta(h) \cdot v$ so $g^{-1}\theta(h)$ is in $G_v$. Since we know that $G_v$ is in the image of $\theta$, we get that $g$ is in the image of $\theta$.
\end{proof}

\section{From preretractions to hyperbolic floors} \label{RetractionSec}
The aim of this section is to prove Propositions \ref{Retraction} and \ref{RetractionPlus}. In both cases, we are given a preretraction $f: A \to G$ from a group $A$ which admits a JSJ-like decomposition $\Lambda$, and we want to modify it to get an extended hyperbolic floor if $G=A$, or a preretraction to a proper retract $G'$ of $G$ if $A \neq G$. Let us give an outline of the proofs.

In section \ref{PinchingDecSubsec}, we give a decomposition $\Lambda_{\cal C}$ for the group $\rho_{\cal C}(A)$ obtained by quotienting $A$ by the elements representing curves of a set ${\cal C}$ of simple closed curves on the surfaces of $\Lambda$. If ${\cal C}$ corresponds to the curves pinched by the preretraction $f: A \to G$, then $f$ will factor as $f=f_{\cal C} \circ \rho_{\cal C}$. Moreover, we obtain a free product decomposition of $\rho_{\cal C}(A)$ as $A_1 * \ldots * A_l* R_{\cal C}$, and some decompositions $\Lambda_i$ of the subgroups $A_i$ as graph of groups with surfaces with respect to which $f_{\cal C}$ is non pinching. The aim of this construction is to enable us to work with non pinching maps: indeed, the property of non pinching maps described by Lemma \ref{FiniteIndex} will be crucial.  We also classify surfaces of $\Lambda$ in three types (A), (B), and (C) according to the images of the corresponding vertex groups under the map $\rho_{\cal C}$.

The idea is that we want to work with the non pinching maps $f_{\cal C}|_{A_i}$ to build the retraction or the preretraction we are looking for. However we need to worry also about preserving non abelianity of the images of surface type vertex groups: for this, we will use the factor $R_{\cal C}$ and choose suitable maps $R_{\cal C} \to G$. In section \ref{NewPreretractionsSubsec}, given a morphism $h_A: A_{\cal C}  \to G$, we give a way to choose $h_R: A \to G$ so that $(h_A * h_R) \circ \rho_{\cal C}$ sends surface type vertex groups of type (C) to non abelian images. 

In section \ref{PreretractionsAndPinchingSubsec}, we describe what happens to surface type vertex groups of type (A) under $f_{\cal C}$. This already enables us to prove Proposition \ref{RetractionPlus} in section \ref{ProofPlusSubsec}. 

In section \ref{EllipticRefSubsec}, we show that if a surface type vertex group $S$ of $\Lambda$ intersects one of the subgroups $f_{\cal C}(A_i)$ in more than a boundary subgroup, then in fact there must be a surface type vertex group $S'$ of $\Lambda$ such that $f(S') \cap S$ has finite index in $S$. Here the fact that $f_{\cal C}$ is non pinching on the surfaces of $\Lambda_i$ is crucial.

We want to generalize this to the case where $S$ intersects $f_{\cal C}(A_1 * \ldots * A_l)$ in more than a boundary subgroup, however, this might not hold since $S$ could intersect the images by $f_{\cal C}$ of distinct factors $A_i$ in distinct boundary subgroups. We define a special class of preretractions called minimal preretractions, and show that this does hold for them. Thus, up to replacing $f$ by a minimal preretraction, we may assume that if $S$ is in the set $S(f)$ of surface type vertex groups which intersects $f_{\cal C}(A_1 * \ldots * A_l)$ in more than a boundary subgroup, there is a surface type vertex group $S'$ of $\Lambda$ such that $f(S') \cap S$ has finite index in $S$.

It will be useful to replace $f$ by a power of itself, however, a power of a preretraction is not necessarily a preretraction. In section \ref{StablePreretractionsSubsec}, we are interested in preretractions for which this holds, which we call stable. We prove two lemmas which indicates conditions on $\Lambda$ and on the initial preretraction $f$ under which one can build stable preretractions.

Section \ref{AnnoyingJSJ} is devoted to dealing with the cases which are not covered by these lemmas. We show that in such cases, we can find directly a structure of extended hyperbolic floor for $A$.

We finish the proof of Proposition \ref{Retraction} in section \ref{ProofRetractionSubsec}. By the previous sections we can assume that we have a non injective preretraction $f$ which is both stable and minimal. We deduce from this that up to replacing $f$ by a power of itself, it sends surface type vertex groups which are in $S(f)$ isomorphically onto conjugates of themselves. We then consider a subgroup $G'$ of $G$ which is spanned in some way by these vertex groups and by the non surface type vertex groups, and we show using in particular Proposition \ref{IsoOnVertexIsIsoOnGroup} that up to a last modification, $f$ restricts to the identity on $G'$. Moreover, we can ensure that the surface type vertex group which are not in $S(f)$ intersect $f(A)$ (and not just $f_{\cal C}(A_1 * \ldots * A_l)$) at most in a boundary subgroup, so that their images are in $G'$. These surfaces will be precisely the surfaces of the hyperbolic floor decomposition, and $f$ is the corresponding retraction.

\subsection{Pinching decomposition} \label{PinchingDecSubsec}

Recall from Definition \ref{EssentialSet} that an essential set of curves on a surface $\Sigma$ is a set of disjoint, two-sided non boundary parallel simple closed curves on $\Sigma$, considered up to homotopy. By extension, if $\Lambda$ is a graph of groups with surfaces, we will call a union of essential sets of curves on the surfaces of $\Lambda$ an essential set of curves on $\Gamma$. Similarly, we extend Definition \ref{PinchingMapSurf} by
\begin{defi} \emph{(pinching map $\rho_{\cal C}$)} Let $A$ be a group which admits a decomposition as a graph of groups with surfaces $\Lambda$, and let ${\cal C}$ be an essential set of curves on $\Lambda$. Let $N({\cal C})$ be the subgroup normally generated in $A$ by the elements of the surface type vertex groups of $\Lambda$ which represent curves of ${\cal C}$. The quotient map $\rho_{\cal C}: A \to A/ N({\cal C})$ is called the pinching map corresponding to ${\cal C}$. 
\end{defi}

The group $\rho_{\cal C}(A)$ admits a decomposition $\rho_{\cal C}(\Lambda)$ obtained from $\Lambda$ by replacing each vertex and edge group by its image by $\rho_{\cal C}$. 

Let $\Sigma$ be a surface of $\Lambda$ and let $S$ be one of the corresponding surface type vertex groups. Denote by ${\cal C}_{\Sigma}$ the set of curves of ${\cal C}$ which lie on $\Sigma$. Recall that we defined the pinching decomposition $\Gamma(S, {\cal C}_{\Sigma})$ of $\rho_{\cal C}(S)$ in Definition \ref{PinchingDecSurf}. Denote by $T_{\rho(S)}$ the $\rho_{\cal C}(S)$-tree associated to $\Gamma(S, {\cal C}_{\Sigma})$: its edge stabilizers are trivial. Recall that the image by $\rho_{\cal C}$ of a boundary element of $S$ fixes a vertex in $T_{\rho(S)}$, and that such a vertex (as well as its image by the quotient map $T_{\rho(S)} \to \Gamma(S, {\cal C}_{\Sigma})$) is called an exterior vertex (Definition \ref{Interior}). 

\begin{defi} \emph{(graph of groups $\Lambda_{\cal C}$)}  We refine  $\rho_{\cal C}(\Lambda)$ at the each surface type vertex with corresponding vertex group $S$ by the graph of groups $\Gamma(S, {\cal C}_{\Sigma})$. The graph of groups thus obtained is denoted by $\Lambda_{\cal C}$. It admits a structure of graph of groups with surfaces, where the surface type vertices correspond to the exterior vertices of the subgraphs of groups $\Gamma(S, {\cal C}_{\Sigma})$. The $\rho_{\cal C}(A)$-tree corresponding to $\Lambda_{\cal C}$ is denoted by $T_{\cal C}$.
\end{defi}

For each surface type vertex group $S$ of $\Lambda$, the $\rho_{\cal C}(S)$-tree $T_{\rho(S)}$ embeds simplicially and equivariantly in $T_{\cal C}$. We will identify it to its image under this embedding.

\begin{defi}\emph{(graphs of groups $\Lambda_1, \ldots, \Lambda_r$)} We denote  by $\Lambda_1, \ldots, \Lambda_r$ the maximal connected subgraph of groups of $\Lambda_{\cal C}$ whose edge groups are non trivial.
\end{defi}

\begin{figure}[!ht]
\begin{center}
\input{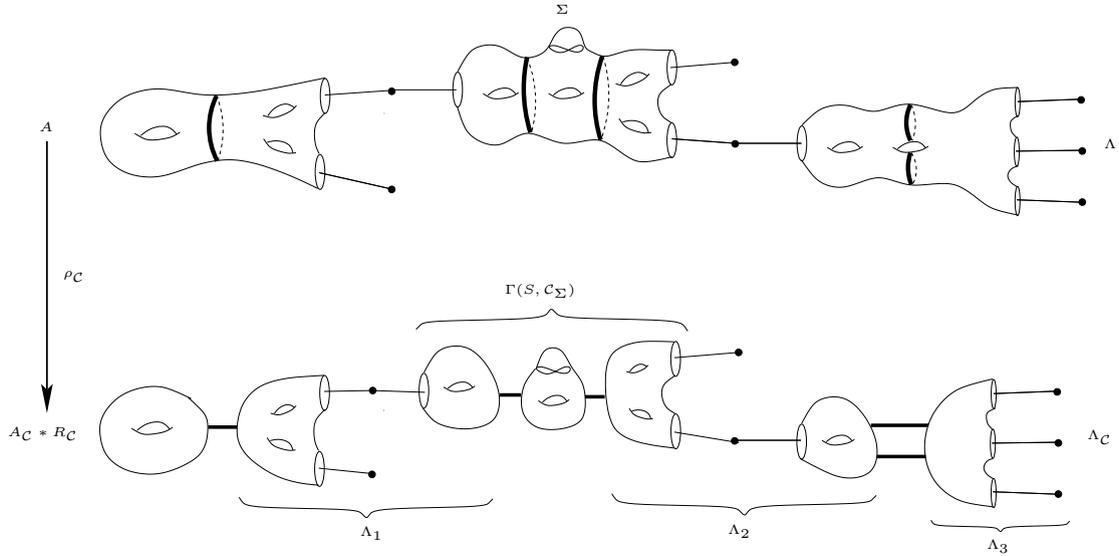}
\caption{Building the graph of groups $\Lambda_{\cal C}$ from $\Lambda$.}
\label{ConstructionFig}
\end{center}
\end{figure}

We now build the graph of groups $\bar{\Lambda}_{\cal C}$ by collapsing each subgraph $\Lambda_i$ to a point. We denote by $\bar{T}_{\cal C}$ the corresponding $\rho_{\cal C}(A)$-tree: it is obtained from $T_{\cal C}$ by collapsing all the non trivially stabilized edges, so that each subtree $T_{\rho(S)}$ of $T_{\cal C}$ embeds in $\bar{T}_{\cal C}$ via the collapse map. The graph of groups $\bar{\Lambda}_{\cal C}$ has trivial edge groups. 

\begin{defi} \emph{(pinching decomposition associated to ${\cal C}$)} Given an essential set of curves on the surfaces of $\Lambda$, a pinching decomposition associated to ${\cal C}$ is a Bass-Serre presentation $(\bar{T}^0_{\cal C}, \bar{\Lambda}^0_{\cal C})$ of $\bar{\Lambda}_{\cal C}$ such that for any subtree $T_{\rho(S)}$, the tree $\bar{T}^0_{\cal C}$ intersects at most one translate of $T_{\rho(S)}$ in more than a point.
\end{defi}

The choice of a pinching decomposition gives a free product decomposition of $\rho_{\cal C}(A)$ as:
$$A_1 * \ldots A_r * S_1 * \ldots * S_q* U_1* \ldots* U_t$$
where the subgroups $S_j$ are the stabilizers of vertices of $\bar{T}^0_{\cal C}$ which are interior vertices of one the subtrees $T_{\rho(S)}$; the subgroup $A_i$ is the stabilizer of the vertex of $\bar{T}^0_{\cal C}$ which is the image of $\Lambda_i$ by the collapse map $\Lambda_{\cal C} \to \bar{\Lambda}_{\cal C}$; and the subgroups $U_k$ are infinite cyclic and correspond to edges of $\bar{T}^0_{\cal C}$ which are not edges of $\bar{\Lambda}^0_{\cal C}$. 

The graph of groups $\Lambda_i$ can be seen as a decomposition of $A_i$ in a canonical way. In what follows, we will thus identify vertex and edge groups of $\Lambda_i$ naturally with subgroups of $A_i$.

\begin{rmk}The choice of an essential set of curves ${\cal C}$ on the surfaces of $\Lambda$ uniquely defines $\rho_{\cal C}$, the graph of groups $\Lambda_{\cal C}$ and its subgraphs of groups $\Lambda_i$, and the graph of groups $\bar{\Lambda}_{\cal C}$. The choice of a pinching decomposition determines the subgroups $A_i$, $S_j$, and $U_k$ of $\rho_{\cal C}(A)$. We will moreover denote by $A_{\cal C}$ and $R_{\cal C}$ the free factors $A_1 * \ldots *A_r$ and $S_1 * \ldots *S_q* F$ of $\rho_{\cal C}(A)$ respectively.
\end{rmk}

We now classify surfaces of $\Lambda$ in three different types: for each type we will later have a different way of ensuring non abelianity of the image of the corresponding surface groups under the various morphisms we will build.
\begin{lemma} \label{NiceSurface} Let ${\cal C}$ be an essential set of curves  on the surfaces of $\Lambda$, and choose a pinching decomposition associated to ${\cal C}$. 

Let $\Sigma$ be a surface of $\Lambda$, and let $S$ be a surface type vertex group corresponding to $\Sigma$.
One of the following is satisfied:
\begin{itemize}
\item[(A)] $\Gamma(S, {\cal C}_{\Sigma})$ is a tree of group which admits a single infinite vertex group;
\item[(B)] $S$ has two boundary subgroups $Z_1$ and $Z_2$ all of whose conjugates in $A$ intersect trivially;
\item[(C)] $S$ has a conjugate $S'$ such that $\rho_{\cal C}(S')$ intersects one of the subgroups $A_i$ in an edge group $Z$ of $\Lambda_i$, and contains also a conjugate $gZg^{-1}$ of $Z$ by an element $g$ of one of the factors $Q=S_j$ or $Q=U_k$ of $R_{\cal C}$. Moreover, there exists an epimorphism $\pi:Q \to \Z$ such that $g \notin \Ker \; \pi$.  
\end{itemize}
\end{lemma}

According to this, we will say that a surface $\Sigma$ of $\Lambda$ (and any of its corresponding vertex groups) is of type (A), (B), or (C). 

\begin{proof} Suppose $\Sigma$ is neither of type (A), nor of type (B). We see $T_{\rho(S)}$ as a subtree of $T_{\cal C}$, recall it embeds in $\bar{T}_{\cal C}$ under the collapse map. Up to replacing $S$ by a conjugate we may assume that no translate of the embedded image of $T_{\rho(S)}$ in $\bar{T}_{\cal C}$ intersects $\bar{T}^0_{\cal C}$ in more than a point. 

Pick an exterior vertex $v_e$ of $T_{\rho(S)}$ whose image $\bar{v}_e$ in $\bar{T}_{\cal C}$ lies in $\bar{\Lambda}^0_{\cal C}$.

Let $z$ be a vertex of $T_{\cal C}$ adjacent to $v_e$, such that the corresponding edge group $Z$ is the image by $\rho_{\cal C}$ of a boundary subgroup of $S$. Note that $\bar{z} = \bar{v}_e$, so $Z$ lies in $A_i$ for some $i$, and is an edge group of $\Lambda_i$. 

If $\Gamma(S, {\cal C}_{\Sigma})$ has interior vertices with infinite stabilizer, there is a vertex $v$ of $\bar{T}^0_{\cal C}\cap T_{\rho(S)}$ stabilized by one of the subgroups $S_j$ with $S_j$ infinite. This means $S_j$ is a subgroup of $\rho_{\cal C}(S)$. Infinite surface groups surject non trivially onto $\Z$: we pick an element $g$ of $S_j$ outside of the kernel of such a surjection. The conjugate $gZg^{-1}$ also lies in $\rho_{\cal C}(S)$. 

If $\Gamma(S, {\cal C}_{\Sigma})$ is not a tree of groups, there is an edge $e$ of $\bar{T}^0_{\cal C}\cap T_{\rho(S)}$ which is not in $\bar{\Lambda}^0_{\cal C}$. Let $g$ be the corresponding Bass-Serre element: it generates one of the factors $U_k$, and $g \cdot v_e$ lies in $T_{\rho(S)}$. Thus $gZg^{-1}$ also lies in $\rho_{\cal C}(S)$. 

If $\Gamma(S, {\cal C}_{\Sigma})$ is a tree of groups and all its interior vertices have finite groups, it must have at least two exterior vertices since we are not in case (A). By our choice of Bass-Serre presentation, $T_{\rho(S)} \cap \bar{T}^0_{\cal C}$ contains a vertex in each orbit of vertices of $T_{\rho(S)}$ under the action of $\rho_{\cal C}(S)$. Thus there exists an exterior vertex $w_e$ of $T_{\rho(S)}$ distinct from $v_e$ whose image $\bar{w}_e$ in $\bar{T}_{\cal C}$ lies in $\bar{T}^0_{\cal C}$. 

Now since $\Sigma$ is not of type (B), all the boundary subgroups of $S$ are conjugates in $A$, so by strong $2$-acylindricity of $\Lambda$, all the edges adjacent to the vertex corresponding to $\Sigma$ in $\Lambda$ are adjacent to a common vertex with cyclic vertex group. This implies that the exterior vertex groups of $\Gamma(S, {\cal C}_{\Sigma})$ are all adjacent to a single vertex with cyclic vertex group. Thus $w_e$ is adjacent to a vertex $g \cdot z$ for some element $g$ of $\rho_{\cal C}(A)$, and the corresponding edge group is non trivial, so $g \cdot z$ and $w_e$ have the same image under the collapse map $T_{\cal C} \to \bar{T}_{\cal C}$. We get $\bar{w}_e = g \cdot \bar{v}_e$. Since $\bar{v}_e$ and $\bar{w}_e$ lie in $\bar{T}^0_{\cal C}$, the element $g$ corresponds to an edge of $\bar{T}^0_{\cal C}$ which is not in $\bar{\Lambda}^0_{\cal C}$, so it generates one of the factors $U_k$. 
\end{proof}

\subsection{Non abelianity for surfaces of type (C)} \label{NewPreretractionsSubsec}

\begin{setting} \label{setting} Let $G$ be a torsion-free hyperbolic group which admits a decomposition as a graph of group with surfaces $\Gamma$, and let $A$ be a subgroup of $G$ endowed with a JSJ-like decomposition $\Lambda$. Assume non surface type vertex groups of $\Lambda$ (and thus edge groups of $\Lambda$) are elliptic in $\Gamma$. 
\end{setting}

Note in particular that if $A$, $\Lambda$ and $G$ satisfy the hypotheses of Proposition \ref{RetractionPlus}, and if $\Gamma$ is the graph of groups associated to the free product $G = G'*G''$, or the decomposition corresponding to the hyperbolic floor $(G, G', r)$ respectively, then we are in the setting above. Similarly, if $A$ and $\Lambda$ satisfy the hypotheses of Proposition \ref{Retraction}, then we are in the setting above with $G=A$ and $\Gamma = \Lambda$.

The following lemma shows that from a morphism $h_A: A_{\cal C} \to G$ which satisfies some weak conditions, we can build a morphism $A \to A$ which sends surface type vertex groups of type (C) to non abelian images. 
\begin{lemma} \label{TypeC}  Suppose $G$, $\Gamma$, $A$, and $\Lambda$ are as given in Setting \ref{setting}. Let ${\cal C}$ be an essential set of curves on the surfaces of $\Lambda$. Choose a pinching decomposition associated to ${\cal C}$. 

Suppose $h_A: A_{\cal C} \to G$ is a morphism  which is injective on the edge group of the graphs of groups $\Lambda_i$. Let $G_0$ be a non abelian subgroup of $G$. Then there exists a morphism $h_R: R_{\cal C} \to G_0$, such that any surface type vertex group of $\Lambda$ of type (C) has non abelian image by $(h_A * h_R) \circ \rho_{\cal C}: A \to G$. 
\end{lemma}

\begin{proof} If $\Sigma$ is a surface of $\Lambda$ of type (C), let $S$ be the corresponding vertex group such that $\rho_{\cal C}(S)$ contains both an edge group $Z$ of one of the subgraphs of groups $\Lambda_i$, and a conjugate $gZg^{-1}$ of $Z$ by an element $g$ of a factor $Q$ of $R_{\cal C}$.  

By hypothesis, $h_A(Z)$ is non trivial. It is enough to choose $h_R$ so that $h_R(g)$ is not contained in the maximal cyclic subgroup containing $h_A(Z)$: indeed, then $h_A(Z)$ and $h_R(g) h_A(Z) h_R(g)^{-1}$ will intersect trivially by malnormality of maximal cyclic subgroups in torsion-free hyperbolic groups.

Note that for each factor $Q$ of $R_{\cal C}$, the morphism $h_R|_Q$ is subjected to at most one such restriction since the vertex or the edge corresponding to $Q$ lies in only one of the subgraphs of groups $\Gamma(S, {\cal C}_{\Sigma})$. Since $G_0$ is non abelian, we can find an element lying outside of the cyclic group $h_A(Z)$. We define $h_R|_Q$ to be the map surjecting $Q$ onto the cyclic subgroup generated by such an element.
\end{proof}

We will need in fact a more precise version of this lemma in the particular case where $G=A$ and $\Lambda$ is complex enough.
\begin{lemma} \label{TypeCPlus} Suppose $A$ is a torsion-free hyperbolic group endowed with a JSJ-like decomposition $\Lambda$. Suppose moreover that if $\Lambda$ has only one non surface type vertex group, this vertex group is non cyclic. Let ${\cal C}$ be an essential set of curves on the surfaces of $\Lambda$. Choose a pinching decomposition associated to ${\cal C}$. 

Suppose $h_A: A_{\cal C} \to A$ is a morphism which sends edge groups of the graphs of groups $\Lambda_i$ injectively to edge groups of $\Lambda$. Let $G_0$ be a non abelian subgroup of $G$.

Then there exists a map $h_R: R_{\cal C} \to G_0$, such that for any surface type vertex group $S$ of $\Lambda$ of type (C), the image of $S$ by $h = (h_A * h_R) \circ \rho_{\cal C}$ contains an edge group $Z$ of $\Lambda$, and a conjugate $uZu^{-1}$ of $Z$ by an element $u$ of a non surface type vertex group $V$ such that one of the following holds
\begin{itemize}
	\item any two conjugates of $Z$ and $V$ intersect trivially; 
	\item $Z$ intersects $V$ non trivially and $\langle u \rangle \cap Z = \{1\}$. 
\end{itemize}
\end{lemma}

\begin{proof} Let $S$ be a surface type vertex group such that $\rho_{\cal C}(S)$ contains both an edge group $Z'$ of one of the subgraphs of groups $\Lambda_i$, and a conjugate $gZ'g^{-1}$ of $Z'$ by an element $g$ of a factor $Q$ of $R_{\cal C}$. 

Let $Z = h_A(Z')$: it is contained in $h(S)$, and by hypothesis, it is an edge group of $\Lambda$. 

Suppose first that any non surface type vertex group containing a non trivial subgroup of $Z$ is cyclic. Since $\Lambda$ is JSJ-like, there is at most one such vertex group. By hypothesis, $\Lambda$ has at least one other non surface type vertex group $V$, so that any two conjugates of $Z$ and $V$ intersect trivially. We define $h_R|_Q$ to be the map surjecting $Q$ onto an infinite cyclic subgroup of $V$. Then the first alternative is satisfied. 

Otherwise, let $V$ be a non abelian non surface type vertex group of $\Lambda$ such that $Z \cap V$ is not trivial. We define $h_R|_Q$ to be the map surjecting $Q$ onto the infinite cyclic group generated by an element of $V$ which does not centralize $Z$. Then the second alternative is satisfied. 
\end{proof}

\subsection{Non abelianity for surfaces of type (A)} \label{PreretractionsAndPinchingSubsec}

\begin{defi} \emph{(essential set of curves pinched by $f$, morphisms $f_{\cal C}$, $f_A$ and $f_R$)} Let $f: A \to G$ be a morphism. An essential set of curves ${\cal C}$ on the surfaces of $\Lambda$ is said to be an essential set of curves pinched by $f$ if $f$ factors through $\rho_{\cal C}$. We denote by $f_{\cal C}$ the morphism $\rho_{\cal C}(A) \to G$ such that $f = f_{\cal C} \circ \rho_{\cal C}$. 

For a given pinching decomposition associated to ${\cal C}$, we denote by $f_A$ and $f_R$ the restrictions of $f_{\cal C}$ to ${A_{\cal C}}$ and $R_{\cal C}$ respectively.
\end{defi}

Note that $f_A$ sends edge groups of $\Lambda_i$ isomorphically to edge groups of $\Lambda$, and non surface type vertex groups of $\Lambda_i$ isomorphically to non surface type vertex groups of $\Lambda$.

\begin{rmk} If ${\cal C}$ is a \textbf{maximal} essential set of curves pinched by $f$, then $f_{\cal C}$ is non pinching on the surfaces of $\Lambda_{\cal C}$.
\end{rmk}

The following lemma shows what happens to surface type vertex groups of the graphs of groups $\Lambda_i$ which have non abelian image under a morphism $A_{\cal C} \to G$. In particular, it will be helpful to guarantee non abelianity of surface groups of type (A) under the preretractions we will build (see Remark \ref{TypeAUnderPreretractions}).

\begin{lemma} \label{FSHatNonAbelian} Suppose $G$, $\Gamma$, $A$, and $\Lambda$ are as given in Setting \ref{setting}. Let $f$ be a morphism $A \to G$, and let ${\cal C}$ be a maximal essential set of curves pinched by $f$. Let $\hat{S}$ be a surface type vertex group of $\Lambda_{\cal C}$.

If $f_{\cal C}(\hat{S})$ is non abelian, then it contains either a finite index subgroup of a surface type vertex group of $\Gamma$, or a non abelian subgroup of a non surface type vertex group of $\Gamma$.
\end{lemma}

\begin{proof} Denote by $\hat{\Sigma}$ the surface associated to $\hat{S}$. Its boundary subgroups are sent to edge groups of $\Lambda$ by $f_{\cal C}$, so they are elliptic in  $\Gamma$ and we can apply Lemma \ref{CuttingLemma} to find an essential set of curves ${\cal C}^+$ on $\hat{\Sigma}$ such that vertex groups of the dual graph of group decomposition $\Delta(\hat{S}, {\cal C}^+)$ are sent to vertex groups of $\Gamma$. Since $f_{\cal C}$ is non pinching on $\hat{\Sigma}$, the edge groups of $\Delta(\hat{S}, {\cal C}^+)$ are sent injectively to edge groups of $\Gamma$ by $f_{\cal C}$. 

We will now show that one of the vertex groups of $\Delta(\hat{S}, {\cal C}^+)$ must have non abelian image by $f_{\cal C}$. If not, we get from $\Delta(\hat{S}, {\cal C}^+)$ a graph of group decomposition for $f_{\cal C}(\hat{S})$ all of whose vertex and edge groups are infinite cyclic, that is, $f_{\cal C}(\hat{S})$ is a generalized Baumslag-Solitar group. In such groups, the centralizer of an element is the whole group \cite{Forester}. Since $f_{\cal C}(\hat{S})$ lies in the torsion-free hyperbolic group $G$, it must be cyclic. This contradicts its non abelianity.

Let thus $\hat{S}_0$ be a vertex group of $\Delta(\hat{S}, {\cal C}^+)$ such that $f_{\cal C}(\hat{S}_0)$ is not abelian. Either $f_{\cal C}(\hat{S}_0)$ is a non abelian subgroup of a non surface type vertex group, or it is a non abelian subgroup of a surface type vertex group $S_1$ of $\Lambda$: by Lemma \ref{FiniteIndex} it has finite index in $S_1$.
\end{proof}

\begin{rmk} \label{TypeAUnderPreretractions} Suppose $G$, $\Gamma$, $A$, and $\Lambda$ are as given in Setting \ref{setting}. Let $f$ be a preretraction $A \to G$, and let ${\cal C}$ be a maximal essential set of curves pinched by $f$. Let $S$ be a surface type vertex group of type (A), and let $\hat{S}$ be a surface type vertex of $\Lambda_{\cal C}$ which lies in $\rho_{\cal C}(S)$. 

Then $f_{\cal C}(\hat{S}) = f(S)$ so it is not abelian, and by Lemma \ref{FSHatNonAbelian}, it contains either a finite index subgroup of a surface type vertex group of $\Gamma$, or a non abelian subgroup of a non surface type vertex group of $\Gamma$.
\end{rmk}

\subsection{Proof of Proposition \ref{RetractionPlus}}\label{ProofPlusSubsec}

\begin{proof}[Proof of Proposition \ref{RetractionPlus}] Let ${\cal C}$ be a maximal essential set of simple closed curves pinched by the non injective preretraction $f: A \to G$. Choose a pinching decomposition associated to ${\cal C}$. 

Let $\Gamma$ be the graph of groups decomposition associated to the free product decomposition $G = G' * G''$ or to the hyperbolic floor $(G, G', r)$ respectively; in the first case denote by $r$ the retraction of $G \to G'$ which restricts to the trivial map on $G''$. 

The morphism $h_A = r \circ f_A$ is injective on edge groups of $\Lambda$, so by Lemma \ref{TypeC} we can find a morphism $h_R: R_{\cal C} \to G'$ such that under the morphism $h = (h_A * h_R) \circ \rho_{\cal C}$, surface type vertex groups of type (C) have non abelian images.

The morphism $h$ restricts to a conjugation on non surface type vertex groups of $\Lambda$. In particular, if $S$ is a surface type vertex group of type (B), $h(S)$ contains conjugates in $G$ of edge groups $Z_1$ and $Z_2$ of $\Lambda$ whose conjugates in $A$ intersect trivially. But $A$ is a retract of $G$, so conjugates of $Z_1$ and $Z_2$ in $G$ also intersect trivially, and the subgroup they generate is not abelian. 

Let now $\Sigma$ be a surface of $\Lambda$ of type (A). By Remark \ref{TypeAUnderPreretractions}, there is a corresponding surface type vertex group $S$, such that $f(S)$ either contains a finite index subgroup of a surface type vertex group $S_1$ of $\Gamma$, or a non abelian subgroup of a non surface type vertex group $V$ of $\Gamma$. If we are in the first case, $h(S)$ contains a finite index subgroup of $r(S_1)$ which is non abelian, so it is itself non abelian. In the second case, $r$ restricts to conjugation on $V$, so $h(S)$ is not abelian either. 

Thus $h$ is a preretraction. If ${\cal C}$ is non empty, $h$ is non injective since it factors through $\rho_{\cal C}$. If ${\cal C}$ is empty, $f_A = f$ is non injective, hence so is $h$.
\end{proof}

\subsection{Elliptic refinements}\label{EllipticRefSubsec}

Let $A$ be a torsion-free hyperbolic group which admits a JSJ-like decomposition $\Lambda$. Let $f: A \to A$ be a preretraction. Let ${\cal C}$ be a maximal essential set of curves pinched by $f$, and choose a pinching decomposition associated to ${\cal C}$.

We consider the minimal subtree of $A_i$ in its action via $f_A$ on the tree $T_{\Lambda}$ corresponding to $\Lambda$. If a surface type vertex $v$ with stabilizer $S$ lies outside of this minimal subtree, in particular the intersection of $S$ with $f_A(A_i)$ is contained in a boundary subgroup. The aim of this section is to examine the case where $v$ lies in this minimal subtree.

Non surface type vertex groups of $\Lambda_i$ are sent to non surface type vertex groups of $\Lambda$ by $f_A$, so they are elliptic in the action of $A_i$ on $T_{\Lambda}$ via $f_A$. Boundary subgroups of a surface type vertex group $S$ of $\Lambda_i$ are elliptic, so by Lemma \ref{CuttingLemma}, there exists an essential set of curves ${\cal C}^+_{\Sigma}$ on $\Sigma$ such that vertex groups (respectively edge groups) of the graph of groups $\Delta(S, {\cal C}^+_{\Sigma})$ dual to ${\cal C}^+_{\Sigma}$ are sent to vertex groups (respectively edge groups) of $\Lambda$ by $f_A$.

\begin{defi} \emph{(elliptic refinement $\Lambda^+_i$)} We refine the graph of groups $\Lambda_i$ at each surface type vertex with corresponding group $S$ by the graph of groups $\Delta(S, {\cal C}^+_{\Sigma})$. The graph of groups thus obtained is called an elliptic refinement for $\Lambda_i$ with respect to $f_A$, and is denoted by $\Lambda^+_i$. It admits a structure of graph of groups with surfaces where the surface type vertices are the vertices of $\Delta(S, {\cal C}^+_{\Sigma})$ for all $S$. We denote the corresponding tree by $T^+_i$.
\end{defi}

\begin{rmk} The map $f_A$ sends edge groups of $\Lambda^+_i$ injectively to edge groups of $\Lambda$, and non surface type vertex groups of $\Lambda^+_i$ isomorphically to non surface type vertex groups of $\Lambda$. Moreover, it is non pinching on the surfaces of $\Lambda^+_i$. 
\end{rmk}

All the vertex groups of $\Lambda^+_i$ are elliptic in the action on $T_{\Lambda}$ via $f_A$, hence by Lemma \ref{CanAssumeLocallyMinimal}, there exists a locally minimal $f_A$-equivariant map $\phi_i:T^+_i \to T_{\Lambda}$.

\begin{lemma} \label{OneLambda} Let $v_S$ be a surface type vertex of $T_{\Lambda}$ with stabilizer $S$. If $v_S$ lies in $\phi_i(\hat{T}^+_i)$, then there is a surface type vertex group $S_0$ of $\Lambda^+_i$ such that $f_A(S_0)$ has finite index in $S$.
\end{lemma}

\begin{proof} Recall that a JSJ-like decomposition is $1$-acylindrical next to surface type vertices, that is, if an element stabilizes two distinct edges adjacent to $v_S$, it must be trivial. Let $v$ be a point of $T^+_i$ such that $\phi_i(v) = v_S$. 

Suppose first that $v$ is a surface type vertex of $T^+_i$. Then $f_A$ is a morphism of surface groups $S_v \to S$ between the surface group $S_v$ corresponding to $v$ and $S$. By local minimality of $\phi_i$, the image $f_A(S_v)$ is not contained in a boundary subgroup of $S$. Since $f_A$ is non pinching on the surface type vertex groups of $\Lambda^+_i$,  by Lemma \ref{FiniteIndex} the subgroup $f_A(S_v)$ must have finite index in $S$. 

Since $f_A$ sends non surface type vertex groups of $\Lambda^+_i$ isomorphically to non surface type vertex groups of $\Lambda$, if $v$ is a non surface type vertex of $T^+_i$, it must have cyclic stabilizer $Z$. Thus $v$ has an open neighborood which is stabilized by a non trivial element. Similary, if $v$ is a point of the interior of an edge of $T^+_i$, it has a non trivially stabilized open neighborood. 

By $1$-acylindricity next to $v_S$, the image by $\phi_i$ of this open neighborood of $v$ is contained in an edge of $T$. But this contradicts local minimality of $\phi_i$. 
\end{proof}

We can deduce from this
\begin{lemma} \label{ProperSubtree} If $f$ is not injective, the subtree $\phi_i(T^+_i)$ is a proper subtree of $T$.
\end{lemma}

\begin{proof} Suppose not. By Lemma \ref{OneLambda}, for any surface type vertex group $S$ of $\Lambda$, there is a surface type vertex group $S_0$ of $\Lambda^+_i$ such that $f_A(S_0)$ has finite index in $S$. Thus by Lemma \ref{ComplexitiesAndFiniteIndex}, the complexity of the surface $\Sigma$ corresponding to $S$ is at most that of the surface $\Sigma_0$ corresponding to $S_0$, and if we have equality the map $f_A$ sends $S_0$ isomorphically onto $S$. Now $\Sigma_0$ is a surface of $\Delta(S_1, {\cal C}^+_{\Sigma_1})$ for some vertex group $S_1$  of $\Lambda_{\cal C}$ with corresponding surface $\Sigma_1$, and in turn $\Sigma_1$ is a surface of $\Gamma(S_2, {\cal C}_{\Sigma_2})$  for some vertex group $S_2$  of $\Lambda$ with corresponding surface $\Sigma_2$. We have $k(\Sigma_0) \leq k(\Sigma_1)$, with equality if and only if ${\cal C}^+_{\Sigma_1}$ is empty, and $k(\Sigma_1) \leq k(\Sigma_2)$, with equality if and only if ${\cal C}_{\Sigma_2}$ is empty. 

Thus we have $k(\Sigma) \leq k(\Sigma_2)$, and if we have equality $f$ sends $S_2$ isomorphically onto $S$. The map $j: \Sigma \mapsto \Sigma_2$ is a map from the set of surfaces of $\Lambda$ to itself which increases complexity, and if $k(\Sigma_2) = k(\Sigma)$ no other surface is mapped by $j$ to $\Sigma_2$ since then $f(S_2)= S$.

The set of surfaces of $\Lambda$ being finite, we conclude that we must have $k(\Sigma)=k(\Sigma_2)$ for all the surfaces $\Sigma$ of $\Lambda$, so that $f$ sends surface groups of $\Lambda$ isomorphically onto surface groups of $\Lambda$. Some power $f^k$ of $f$ sends surface type vertex groups isomorphically onto themselves: by Proposition \ref{IsoOnVertexIsIsoOnGroup}, $f^k$ is an isomorphism, which contradicts non injectivity of $f$.
\end{proof}

\subsection{Minimal preretractions} \label{MinimalPreretractionsSubsec}

Let $A$ be a torsion-free hyperbolic group which admits a JSJ-like decomposition $\Lambda$. Let $f:A \to A$ be a preretraction with respect to $\Lambda$. Let ${\cal C}$ be an essential set of curves pinched by $f$. 

\begin{defi} \emph{(set ${\cal S}_{\cal C}(f)$)} We let ${\cal S}_{\cal C}(f)$ be the set of surfaces of $\Lambda$ for which there exists a corresponding surface type vertex group $S$ such that $f_A(A_{\cal C}) \cap S$ is not contained in a boundary subgroup of $S$. 
\end{defi}

\begin{defi} \emph{(${\cal C}$-minimal preretractions)} We say that $f$ is ${\cal C}$-minimal if there is no preretraction $h: A \to A$ which factors through $\rho_{\cal C}$ such that ${\cal S}_{\cal C}(f)$ strictly contains ${\cal S}_{\cal C}(h)$.
\end{defi}

An essential property of ${\cal C}$-minimal maps is given by
\begin{prop} \label{NoWeirdSurfaces} Let $f: A \to A$ be a preretraction. Let ${\cal C}$ be a maximal essential set of curves pinched by $f$. Assume $f$ is non injective and ${\cal C}$-minimal.

If $\Sigma$ lies in ${\cal S}_{\cal C}(f)$, then there is a surface type vertex group $S_0$ of one of the graphs of groups $\Lambda^+_i$ such that $f_A(S_0)$ has finite index in a surface type vertex group corresponding to $\Sigma$. 
\end{prop} 

We build elliptic refinements $\Lambda^+_i$ with respect to $f_A$ for the graphs of groups $\Lambda_i$ as in the previous section. Recall that we have a locally minimal equivariant map $\phi_i: T^+_i \to T_{\Lambda}$, where $T^+_i$ is the tree corresponding to $\Lambda^+_i$, and that (by Lemma \ref{OneLambda}) if a surface type vertex $v_S$ with vertex group $S$ lies in one of the trees $\phi_i(\hat{T}^+_i)$, then there is a surface type vertex group $S_0$ of $\Lambda^+_i$ such that $f_A(S_0)$ has finite index in $S$.

If $v_S$ does not lie in $\phi_i(T^+_i)$, its intersection with $f_A(A_i)$ is at most a boundary subgroup, so if there is only one factor $A_i$, Proposition \ref{NoWeirdSurfaces} is proved. The problem when there are at least two factors $A_i$ is that $S$ could intersect several subgroups $f_A(A_i)$ in distinct boundary subgroups. By $2$-acylindricity of $\Lambda$, this would require in particular that the trees $\phi_i(T^+_i)$ be close in $T$. The idea is thus to find a preretraction $f'$ which restricts to $\Conj(g_i) \circ f_{A}$ on each $A_i$, where the elements $g_i$ are chosen so that the trees $g_i \cdot \phi_i(T^+_i)$ are far away from each other. This will imply that the surface corresponding to $v_S$ is not in ${\cal S}_{\cal C}(f')$, but this will contradict ${\cal C}$-minimality of $f$.

We first prove a lemma which shows that we can find a preretraction from a morphism $A_{\cal C} \to A$ which coincides up to conjugation with $f_A$ on each subgroup $A_i$. 
\begin{lemma} \label{ConjugatingAway} Let $A$ be a torsion-free hyperbolic group endowed with a JSJ-like decomposition $\Lambda$. Let $f: A \to A$ be a non injective preretraction with respect to $\Lambda$. Let ${\cal C}$ be a maximal essential set of curves pinched by $f$ on the surfaces of $\Lambda$, and choose a pinching decomposition associated to ${\cal C}$. 

Suppose $h_A: A_{\cal C} \to A$ restricts to $\Conj(g_i) \circ f_A$ on $A_i$. Then there exists a morphism $h_R: R_{\cal C} \to A$ such that $h = (h_A * h_R) \circ \rho_{\cal C}$ is a non injective preretraction. Moreover if $h_A(A_{\cal C})$ is not abelian, we can assume that $h(A)$ is contained in $h_A(A_{\cal C})$.
\end{lemma} 

\begin{proof} Note that $h_A$ sends edge groups of $\Lambda_i$ injectively into edge groups of $\Lambda$, since $f_A$ does. We can thus choose $h_R$ according to Lemma \ref{TypeC}, and $h$ sends surface type vertex groups of type (C) to non abelian images.
Clearly $h$ restricts to conjugation on each non surface type vertex of $\Lambda$. In particular, this implies the non abelianity of the image of surface type vertex groups of type (B).
Finally, by Remark \ref{TypeAUnderPreretractions}, for any surface of type (A), there is a corresponding surface type vertex group $S$ and a subgroup $\hat{S}$ of $\rho_{\cal C}(S)$ which lies in $A_i$ for some $i$ and such that $f_A(\hat{S})$ is not abelian. Hence $h_A(\hat{S})$ is not abelian, so neither is $h(S)$. 
\end{proof}

We will also need the two following lemmas about actions on trees. Recall that a $G$-tree is said to be irreducible if none of its ends are fixed by $G$.

\begin{lemma} \label{DistantSubtrees} Let $G$ be a finitely generated group, and let $T$ be a minimal irreducible $G$-tree. If $\tau$ and $\tau'$ are proper subtrees of $T$, for any integer $D$ there is a translate of $\tau'$ by an element of $G$ which lies at a distance at least $D$ of $\tau$.
\end{lemma}

\begin{proof}By Lemma 4.3 in \cite{PaulinGromovTop}, the hypotheses allow us, for any two distinct vertices $v$ and $w$ of $T$, to find an element of $G$ which is hyperbolic in the action of $G$ on $T$ and whose axis contains the path between $v$ and $w$.

Suppose first that the smallest tree $\tau_0$ containing $\tau \cup \tau'$ is a proper subtree of $T$. Let $K$ be a connected component of the complement of $\tau_0$ in $T$, and let $u$ be the vertex of $T$ such that $\overline{K} \cap \tau_0 = \{u\}$. By minimality and irreducibility of the action, $K$ is not a line, so we can find points $v$ and $w$ in such a component such that the tripod formed by $v$, $w$, and $u$ is non trivial. We pick a hyperbolic element $g$ whose axis contains the path between $v$ and $w$. The projection of $\tau$ and $\tau'$ on the axis of $g$ is reduced to a point. Thus $g^{D} \cdot \tau'$ is at distance greater than $D$ of $\tau$.

If on the other hand $\tau_0 = T$, we pick vertices $v, w$ of the tree which are in $\tau'$ but not in $\tau$, and in $\tau$ but not in $\tau'$ respectively. Now $\tau$ lies in the connected component of $T-\{v\}$ containing $w$ and $\tau'$ lies in the connected component of $T-\{w\}$ containing $v$. Thus the intersection $\tau \cap \tau'$ lies in the connected component of $T-\{v, w\}$ containing the arc between $v$ and $w$. Pick a hyperbolic element whose axis contains the path between $v$ and $w$. By applying a suitable power of this element we can translate $\tau'$ away from $\tau$.
\end{proof}

\begin{lemma} \label{PingPong} Let $G$ be a finitely generated group,  and let $\tau$ be a $k$-acylindrical minimal $G$-tree. Suppose $G_1$ and $G_2$ are subgroups of $G$ which generate $G$, and whose minimal subtrees $T_1$ and $T_2$ in $\tau$ are joined by an arc $D$ of length at least $2k +3$. Then for any vertex $v$ of $\tau$
\begin{itemize}
\item either $\Stab_G(v)$ stabilizes an edge adjacent to $v$;
\item or $v$ lies in a translate of $T_i$ by an element of $G$, and in this case $\Stab_G(v)$ stabilizes this translate.
\end{itemize}
Moreover, any vertex of $D$ which is at distance greater than $k+1$ of both $T_1$ and $T_2$ has valence $2$ in $T$, and we have $G = G_1*G_2$.
\end{lemma}

\begin{proof} The tree $\tau$ is the union of translates of $T_1$, $T_2$ and $D$ by elements of $G$. Let $\hat{T_i}$ for $i=1,2$ be the set of points whose distance to $T_i$ is at most $k+1$: note that $\hat{T}_1$ and $\hat{T}_2$ are disjoint. Denote by $\hat{D}$ the subsegment of $D$ which joins $\hat{T}_1$ and $\hat{T}_2$. Let $B_i$ be the complement in $\tau-\hat{T}_i$ of the connected component containing the interior of $\hat{D}$ for $i=1,2$.

By $k$-acylindricity, an element of $G_1$ sends points of $\hat{D}$, of $\hat{T}_2$ and of $B_2$ into $B_1$, and an element of $G_2$ sends points of $\hat{D}$, of $\hat{T}_1$ and of $B_1$ into $B_2$.

If $v \in \hat{D}$, its image by a non trivial element of $G$ lies in $B_1 \cup B_2$. This implies that $v$ is not contained in any translates of $T_1$ or $T_2$, so that it has valence $2$ in $T$, but also that $\Stab_G(v)$ is trivial. Moreover, if $u$ is a non trivial word in $G_1$ and $G_2$, then $u \cdot v$ is distinct from $v$, so $u$ represents a non trivial element of $G$. Thus $G=G_1*G_2$.

If $v$ lies in $\hat{T}_1$ and $g \notin G_1$, then $g \cdot v$ lies in $B_1 \cup B_2$. This implies that the stabilizer of $v$ is contained in $G_1$. If $v$ lies in $\hat{T}_1 - T_1$, the stabilizer of $v$ also stabilizes the path between $v$ and $T_1$, so it stabilizes an edge adjacent to $v$. We get a similar result if $v$ lies in $\hat{T}_2$. If $v$ lies in a translate $g \cdot \hat{D}$ of $\hat{D}$, or in a translate $g \cdot \hat{T_i}$ of $\hat{T_i}$, we apply the results above to $g^{-1} \cdot v$.
\end{proof}

We now prove Proposition \ref{NoWeirdSurfaces}. We denote by $H^g$ the conjugate of a subgroup $H$ by an element $g$.
\begin{proof}[Proof of Proposition \ref{NoWeirdSurfaces}] We will show by induction on $k$, for $1 \leq k \leq l$, that we can find $g_1, \ldots, g_l$ in $A$ such that if $\tau_k$ is the minimal subtree in $T_{\Lambda}$ of the subgroup $G_k$ generated by $f_A(A_1)^{g_1}, \ldots , f_A(A_k)^{g_k}$ we have:
\begin{itemize}
\item $G_k = f_A(A_1)^{g_1}* \ldots * f_A(A_k)^{g_k}$;
\item $\tau_k$ is a proper subtree of $T_{\Lambda}$;
\item if $v$ is a vertex of $\tau_k$, either $v$ lies in a translate of $\phi_i(T^+_i)$, or the intersection of the stabilizer of $v$ with $G_k$ stabilizes an edge.
\end{itemize}

For $k=1$ we let $g_1=1$. Note that the minimal subtree $\tau_1$ of $f_A(A_1)$ is exactly $\phi_1(T^+_1)$, which is a proper subtree of $T$ by Lemma \ref{ProperSubtree}. The other properties are immediate.

Suppose we have found $g_1, \ldots, g_{k-1}$ satisfying these properties. The subtrees $\tau_{k-1}$ and $T_k= \phi_k(T^+_k)$ are both proper subtrees of $T_{\Lambda}$ by induction hypothesis and Lemma \ref{ProperSubtree} respectively. By applying Lemma \ref{DistantSubtrees}, there is an element $g_k$ of $A$ such that $\tau_{k-1}$ and $g_k \cdot T_k$ are at a distance at least $20$ from each other. 

Lemma \ref{PingPong} shows that the subgroup $G_k$ generated by $G_{k-1}$ and $f_A(A_k)^{g_k}$ is in fact their free product, and that the third property holds as well. Moreover, any vertex of $D$ far from $\tau_{k-1}$ and $T_k$ has valence $2$ in $\tau_k$. Since $D$ is long, by strong $2$-acylindricity of $T_{\Lambda}$ there exist such a vertex of whose valence in $T_{\Lambda}$ is greater than $2$. Hence $\tau_k$ is a proper subtree of $T_{\Lambda}$. 

Now let $F_A: A_{\cal C} \to A$ be defined by $\Conj(g_i) \circ f_{A}$ on each $A_i$.  By Lemma \ref{ConjugatingAway}, we can choose $F_R: R_{\cal C} \to A$ such that $F = (F_A * F_R) \circ \rho_{\cal C}$ is still a preretraction. 

Let $S$ be a surface type vertex group of $\Lambda$ whose intersection with $F_A(A_{\cal C})$ is not contained in a boundary subgroup, so that the corresponding surface $\Sigma$ lies in ${\cal S}_{\cal C}(F)$. The vertex $v_S$ it stabilizes in $T_{\Lambda}$ must lie in the minimal subtree $\tau_l$ of $F_A(A_{\cal C})$, so it lies in a translate of $\phi_i(T^+_i)$ for some $i$. By Lemma \ref{OneLambda}, up to replacing $S$ by a conjugate, one of the surface type vertex group of $\Lambda^+_i$ is sent by $F_A$ to a finite index subgroup of $S$. Since $f_A$ and $F_A$ differ by a conjugation on $A_i$, this is still true for $f_A$, and in particular $\Sigma$ lies in ${\cal S}_{\cal C}(f)$. 

Thus ${\cal S}_{\cal C}(F) \subseteq {\cal S}_{\cal C}(f)$, but by ${\cal C}$-minimality of $f$, we must have equality. This proves the result.
\end{proof}

\subsection{Stable preretractions} \label{StablePreretractionsSubsec}

Let $A$ be a torsion-free hyperbolic group which admits a JSJ-like decomposition $\Lambda$. 

\begin{defi} \emph{(stable preretraction)} Let $f$ be a preretraction $A \to A$ with respect to $\Lambda$. We say that $f$ is stable if $f^k$ is a preretraction for all $k>0$.
\end{defi}

Even if we know that there exists a non injective preretraction $f: A \to A$, we cannot always guarantee the existence of a stable non injective preretraction. However, this will be possible under some restrictions on $\Lambda$, or if $f$ satisfies some extra conditions. 

\begin{lemma} \label{StabilizingAPreretraction} Suppose that if $\Lambda$ has only one non surface type vertex group $V$, then $V$ is non abelian. Let $f: A \to A$ be a non injective preretraction, and let ${\cal C}$ be a maximal essential set of curves pinched by $f$. Choose a pinching decomposition associated to ${\cal C}$.

There exists a morphism $h_R: R_{\cal C} \to A$ such that $h = (f_A * h_R) \circ \rho_{\cal C}$ is a stable non injective preretraction and $h(A) \leq f_A(A_{\cal C})$. 
\end{lemma}

\begin{proof} Note that the subgroup $f_A(A_{\cal C})$ is not abelian: indeed, it contains conjugates of all the non surface type vertex groups of $\Lambda$, but by assumption $\Lambda$ admits either a non abelian non surface type vertex group, or two non conjugate non surface type vertex groups. 

We choose a morphism $h_R: R_{\cal C} \to f_A(A_{\cal C})$ according to Lemma \ref{TypeCPlus}. In particular, if $S$ is a surface type vertex group of type (C), its image by $h$ is not abelian.

The morphism $h$ restricts to conjugation on non surface type vertex groups of $\Lambda$, and thus surface type vertex groups of type (B) have non abelian image by $h$. If $\Sigma$ is a surface of $\Lambda$ of type (A), $h(S)$ is not abelian by Remark \ref{TypeAUnderPreretractions}. If ${\cal C}$ is not empty, $h$ is non injective since it factors through $\rho_{\cal C}$, and if ${\cal C}$ is empty, $h = f_A = f$ is also non injective. Thus $h$ is a non injective preretraction.

Suppose $h^{k-1}$ is a non injective preretraction. The morphism $H=h^k$ factors as $H_{\cal C} \circ \rho_{\cal C}$, and the restriction $H_A$ of $H_{\cal C}$ to $A_{\cal C}$ is exactly $h^{k-1} \circ h_A$. Clearly $H$ restricts to conjugation on non surface type vertex groups and edge groups of $\Lambda$, so in particular surface type vertex groups of type (B) have non abelian image by $H$. 

Let $S$ be a surface type vertex group of type (A): by Remark \ref{TypeAUnderPreretractions}, its image $h(S)$ contains either a subgroup of finite index of a surface type vertex group $S_1$ of $\Lambda$, or a non abelian subgroup of a non surface type vertex group $V$ of $\Lambda$. In the first case, note that $h^{k-1}(S_1)$ is not abelian by induction hypothesis, so $H(S)$ is non abelian. In the second case, $h^{k-1}$ restricts to a conjugation on $V$ so $H(S)$ is non abelian.

We chose $h_R$ according to Lemma \ref{TypeCPlus}, so if $S$ is of type (C), the group $h(S)$ contains an edge group $Z$ of $\Lambda$ and a conjugate $uZu^{-1}$ of $Z$ by an element $u$ of a non surface type vertex group $V$ such that one of the following holds
\begin{itemize}
	\item any two conjugates of $Z$ and $V$ intersect trivially; 
	\item $Z$ intersects $V$ non trivially, and $\langle u \rangle \cap Z$ is trivial. 
\end{itemize}
In the first case, since $h^{k-1}$ sends $Z$ and $V$ to conjugate of themselves, $H(S)$ contains both $h^{k-1}(Z)$ and a conjugate of $h^{k-1}(Z)$ by an element $h^{k-1}(u)$ which lies in a conjugate of $V$, and as such does not centralize $h^{k-1}(Z)$. In the second case, $h^{k-1}$ restricts to a conjugation on $V$, so $h^{k-1}(Z)$ and $h^{k-1}(\langle u \rangle)$ intersect trivially. In both cases $H(S)$ is not abelian.
\end{proof}

If the conditions on $\Lambda$ are not satisfied, we can still get a stable non injective preretraction provided there exists a non injective preretraction which satisfies some good properties.
\begin{lemma} \label{StabilizingWithFSHatNonAbelian} Let $f: A \to A$ be a non injective preretraction, and let ${\cal C}$ be a maximal essential set of curves pinched by $f$. Choose a pinching decomposition associated to ${\cal C}$.

Suppose that for any surface type vertex group $S$ of $\Lambda$, at least one of the exterior vertex group $\hat{S}$ of $\Gamma(S, {\cal C}_{\Sigma})$ has non abelian image by $f_{\cal C}$. 

Then the morphism $h = (f_A * 1) \circ \rho_{\cal C}$ is a stable non injective preretraction.  
\end{lemma}

Note that we also get $h(A) \leq f_A(A_{\cal C})$.

\begin{proof} Clearly $h$ restricts to conjugation on non surface type vertex groups of $\Lambda$, and for any surface type vertex group $S$, we have $\hat{S} \leq \rho_{\cal C}(S)$, so $f(S)$ contains $f_A({\hat S})$ and is not abelian. This shows $h$ is a preretraction. Since $f$ is not injective, neither is $h$.

Suppose $h^{k-1}$ is a non injective preretraction.  By Lemma \ref{FSHatNonAbelian}, if $S$ is a surface type vertex group of $\Lambda$, then $f_A(\hat{S})$ contains either a non abelian subgroup of a non surface type vertex group, or a finite index subgroup of a surface type vertex group. Since $h^{k-1}$ is a preretraction, in both cases $H(S)$ contains a non abelian subgroup.
\end{proof}

\begin{rmk} \label{StableAndMinimal} Note that being ${\cal C}$-minimal only depends on $f_A$. Thus if $f$ is ${\cal C}$-minimal, and we are in the setting of one of the two lemmas above, we may assume without loss of generality that $f$ is stable and that $f(A) \leq f_A(A_{\cal C})$.
\end{rmk}

\subsection{A special case} \label{AnnoyingJSJ}

In Section \ref{ProofRetractionSubsec}, we will prove Proposition \ref{Retraction} with the extra assumption that there exists a stable and ${\cal C}$-minimal non injective preretraction. However, there are cases which are covered neither by Lemma \ref{StabilizingAPreretraction} nor by Lemma \ref{StabilizingWithFSHatNonAbelian}, so that we cannot guarantee that such a preretraction exists. We will now prove directly the existence of an extended hyperbolic floor in these cases.

\begin{prop} \label{IfNoStable} Let $A$ be a torsion-free hyperbolic group which admits a JSJ-like decomposition $\Lambda$ which does not consist of a single surface type vertex. Let $f:A \to A$ be a non injective preretraction with respect to $\Lambda$.

Suppose that $\Lambda$ and $f$ satisfy neither the hypotheses of Lemma \ref{StabilizingAPreretraction}, nor those of Lemma \ref{StabilizingWithFSHatNonAbelian}. Then there exists a retraction $r: A \to A'$ to a proper subgroup of $A$ such that $(A, A', r)$ is an extended hyperbolic floor. If moreover $\Lambda$ has at least two surface type vertices, $(A,A',r)$ can be assumed to be a hyperbolic floor.
\end{prop}

Assume that the hypotheses of Lemma \ref{StabilizingAPreretraction} do not hold: $\Lambda$ has a unique non surface type vertex $w$, and the corresponding vertex group $Z = \langle z \rangle$ is maximal cyclic. We assume without loss of generality that $f$ is the identity on $Z$.

Let ${\cal C}$ be a maximal essential set of curves pinched by $f$. Choose a pinching decomposition associated to ${\cal C}$. Let $v$ be a surface type vertex of $\Lambda$. If the corresponding surface $\Sigma$ is orientable, we fix an orientation. We choose maximal boundary elements $b_1, \ldots, b_m$ of $S$ corresponding to boundary curves (with positive orientation if $\Sigma$ is orientable). For each edge $e$ adjacent to $v$, let $k_i$ be such that the edge group $G_e$ has a generator which maps to $b_i$ and $z^{k_i}$ under the embeddings $G_e \hookrightarrow G_v$ and $G_e \hookrightarrow G_w=Z$ respectively. We call the integers $k_i$ the boundary exponents of $\Sigma$.

\begin{lemma} \label{SumBoundaryExponents}   Let $A$ be a torsion-free hyperbolic group which admits a JSJ-like decomposition $\Lambda$ which does not consist of a single surface type vertex. Let $f:A \to A$ be a non injective preretraction with respect to $\Lambda$. Suppose that $\Lambda$ and $f$ satisfy neither the hypotheses of Lemma \ref{StabilizingAPreretraction}, nor those of Lemma \ref{StabilizingWithFSHatNonAbelian}.

Then there exists a surface $\Sigma$ of $\Lambda$ with boundary exponents $k_1, \ldots, k_m$ such that the sum $k_1+ \ldots + k_m$ is even. If $\Sigma$ is in fact orientable, we have in fact $k_1+ \ldots + k_m=0$. 

Moreover, if $\Sigma$ is $(i)$ a $4$-punctured sphere, $(ii)$ a $3$-punctured projective plane, $(iii)$ a $2$-punctured Klein bottle, we must have (up to reordering) in case $(i)$ $k_1 = - k_2$ and $k_3= -k_4$; in case $(ii)$ $k_1=- k_2$ and $k_3$ even; in case $(iii)$ $k_1 = -k_2$, or $k_1$ and $k_2$ both even.
\end{lemma}

\begin{proof} Since the hypotheses of Lemma \ref{StabilizingWithFSHatNonAbelian} are not satisfied, there exists a surface type vertex group $S$ of $\Lambda$ with corresponding surface $\Sigma$, such that any exterior surface type vertex group $\hat{S}$ of $\Gamma(S, {\cal C}_{\Sigma})$ has abelian image by $f_{\cal C}$. Note that $f_{\cal C}$ is non pinching on the surface $\hat{\Sigma}$ corresponding to $\hat{S}$. 

The only surface groups which admit a map to an abelian group which is non pinching and injective on boundary subgroups correspond to punctured spheres, punctured projective planes, and punctured Klein bottles. Indeed, any other surface has a simple closed curve represented by a product of commutators. 

The product of the elements $\rho_{\cal C}(b_i)$ which lie in a common exterior surface type vertex group $\hat{S}$ of $\Gamma(S, {\cal C}_{\Sigma})$ is thus either trivial, or a square, or a product of squares. In the cyclic group $f_{\cal C}(\hat{S})$, the product of the corresponding elements $f(b_i)$ is either trivial, or a square. Since $f(b_i)$ is conjugate to $z^{k_i}$, the sum of these boundary exponents is even, and equal to $0$ if $\Sigma$ is orientable. 

Note now that since all the exterior surface type vertex group have abelian image by $f$, and $S$ doesn't, this means that ${\cal C}$ is not empty. If $\Sigma$ is a $4$-punctured sphere, the only possibility is a simple closed curve separating $\Sigma$ into two $3$-punctured spheres. If $\Sigma$ is a $3$-punctured projective plane, ${\cal C}$ consists of a single curve which separates $\Sigma$ into a $2$-punctured projective plane and a $3$-punctured projective plane. If $\Sigma$ is a $2$-punctured Klein bottle there are two possibilities for ${\cal C}$: it contains a single curve which separates $\Sigma$ into either two $2$-punctured projective plane, or a punctured Klein bottle and a $3$-punctured sphere. The restrictions on boundary exponent follows.
\end{proof}

To prove Proposition \ref{IfNoStable}, we deal separately with the case where such a surface $\Sigma$ has a single boundary component.

\begin{lemma} \label{OneBoundaryComp} Suppose $\Lambda$  admits  a surface $\Sigma$ which satisfies the conclusion of Lemma \ref{SumBoundaryExponents}.

If $\Sigma$ has only one boundary component, then $A$ admits a structure of {\bf extended} hyperbolic floor, which can be assumed to be a hyperbolic floor if $\chi(\Sigma) < -2$ or if $\Lambda$ has surfaces other than $\Sigma$.
\end{lemma}

\begin{proof} Note that $k_1 \neq 0$, so $\Sigma$ is not orientable and $k_1$ is even, say $k_1=2k$. If $S$ is the vertex group corresponding to $\Sigma$, the subgroup $S' = \langle z^k, S \rangle_A$ is the fundamental group of the closed non orientable surface $\Sigma'$ obtained by gluing a Mobius band $M$ along the boundary component of $\Sigma$. We have $S' =\langle d_1, \ldots, d_p \mid d_1^2\ldots d_p^2=1 \rangle$.

If $p > 5$ there is a retraction $S' \to \langle d_1, d_2 \rangle_{S'}$ given by $r(d_3)=d^{-1}_2$, $r(d_4)=d^{-1}_1$ and $r(d_j) =1$ for all $j>4$. It extends to a retraction which gives $(A, r(A), r)$ a structure of hyperbolic floor (with a single surface which is the once punctured connected sum of $p-2$ projective planes).

If $p=4$, let $s$ be a non trivial element in one of the other surface groups of $\Lambda$ (if there are some), or the generator of an infinite cyclic group $\Z$ otherwise. We define a retraction $r$ on $A$ (respectively on $A * \Z$) by setting $r(d_1) = d_1$, $r(d_2)=d^{-1}_1$, and $r(d_3) = r(d^{-1}_4) = s$ and extending it to $A$ (respectively to $A* \Z$) by the identity on the other surface type vertex groups of $\Lambda$ (respectively on $\Z$). Thus $A$ has a structure of extended hyperbolic floor over a proper subgroup, which is a hyperbolic floor if $\Lambda$ has surfaces other than $\Sigma$.
\end{proof}

We can finally prove Proposition \ref{IfNoStable}.
\begin{proof}[Proof of Proposition \ref{IfNoStable}] Let $\Sigma$ be the surface of $\Lambda$ obtained in Lemma \ref{SumBoundaryExponents}. By Lemma \ref{OneBoundaryComp}, we may assume $\Sigma$ has at least two boundary components. 

Suppose first that $\Sigma$ is an $m$-punctured surface of genus at least $1$ ($m \geq 2$), with the exception of a twice punctured torus. We can find (see Figure \ref{FoldableFig} for the orientable case) a set $C$ which is the union of one or two disjoint simple closed curves on $\Sigma$ such that $\Sigma$ can be written as the union of two subsurfaces $\Sigma_1$ and $\Sigma_2$ with
\begin{itemize}
  \item $\Sigma_1 \cap \Sigma_2 = C$;
	\item $\Sigma_1$ is orientable and contains exactly one boundary component of $\Sigma$;
	\item $\Sigma_2$ is homeomorphic to an $(m-2)$-punctured copy of $\Sigma_1$ (if $\Sigma$ is orientable), or to the $(m-2)$-punctured connected sum of a projective plane with $\Sigma_1$ (if $\Sigma$ is non orientable).
\end{itemize}

\begin{figure}[!ht]
\begin{center}
\input{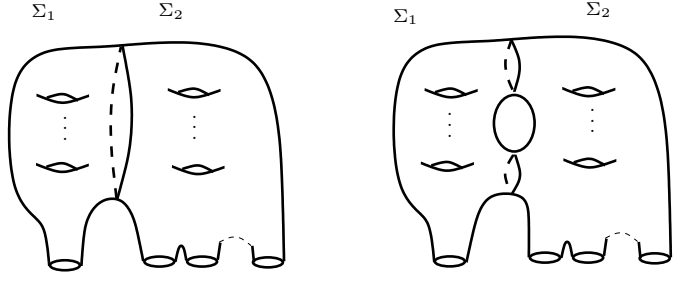}
\caption{The even genus case is represented on the left, the odd genus case on the right.}
\label{FoldableFig}
\end{center}
\end{figure}

In the even genus case, this means we have a presentation for the fundamental group $S$ of $\Sigma$ with generators 
$\{ x_i, y_i, x'_i, y'_i \}_{1 \leq i \leq r}  \cup \{b_1, \ldots, b_m\} \cup \{d\}$, and relations
$$ b^{-1}_1 \; \Pi^r_{i=1} [x_i, y_i] = b_2 \ldots  b_m d^2  \; \Pi^r_{i=1} [ x'_i, y'_i ]$$
and $d=1$ if $\Sigma$ is orientable, and such that the fundamental group $S_1$ of $\Sigma_1$ is generated by $\{ x_i, y_i \}_{1 \leq i \leq r} \cup \{b_1\}$ and  the fundamental group $S_2$ of $\Sigma_2$ by $\{ x'_i, y'_i \}_{1 \leq i \leq r}  \cup \{b_2, \ldots, b_m\} \cup \{d\}$

We then define the retraction $r$ to be the identity on $S_1$, and defined on $S_2$ by $r(x'_i) = x_i$, $r(y'_i) = y_i$, and $r(b_i)= z^{k_i}$ for $i\geq 2$, and if $\Sigma$ is non orientable, we set $r(d)= z^{\delta}$ where $\delta$ is such that $k_1 + k_2 + \ldots + k_m + 2\delta = 0$. This extends to $A$, and $(A, r(A), r)$ is a hyperbolic floor. We can proceed similarly in the odd genus case.

We consider now the case of a twice punctured torus. The corresponding vertex group admits a presentation as $\langle x, y, b_1, b_2 \mid [x,y] = b_1 b_2 \rangle,$ and we must have $b_1 = z^k$ and $b_2 = tz^{-k}t^{-1}$ for $t$ a Bass-Serre element corresponding to the extra edge joining the surface vertex to the rigid vertex. The retraction which sends $x$ to $z^k$, $y$ to $t$, and is the identity on $Z$ gives a structure of hyperbolic floor for $A$ over the (free) subgroup generated by $z$ and $t$. 

If $\Sigma$ has genus $0$, it is an $m$-punctured $(i)$ sphere, $(ii)$ projective plane or $(iii)$ Klein bottle. If $m$ is at least $(i)$ $5$, $(ii)$ $4$, and $(iii)$ $3$ respectively, then $A$ admits a structure of hyperbolic floor with one surface group corresponding to the complement of a thrice punctured sphere in $\Sigma$.

Suppose now that $\Sigma$ is a $(i)$ $4$-punctured sphere, $(ii)$ $3$-punctured projective plane or $(iii)$ $2$-punctured Klein bottle. By the moreover part of Lemma \ref{SumBoundaryExponents}, we can choose
\begin{itemize}\addtolength{\itemsep}{-0.5\baselineskip}
 \item[(i)] $t, u, v$ elements of $A$ such that $z^{k_1} tz^{-k_1} t^{-1} uz^{k_3}u^{-1} vz^{-k_3}v^{-1} = 1$;
 \item[(ii)] $t, u, a$ elements of $A$ such that $z^{k_1} tz^{-k_1} t^{-1} uz^{2 \delta}u^{-1} a^2 = 1$;
 \item[(iii)] $t, a, b$ elements of $A$ such that $z^{k_1} tz^{k_2} t^{-1} a^2 b^2 = 1$ with $k_1 = -k_2$ or $k_1 = 2 \delta$ and $k_2 = 2 \epsilon$.
\end{itemize}
We now define
\begin{itemize}\addtolength{\itemsep}{-0.5\baselineskip}
 \item[(i)] $r(t)=1$ and $r(u) = r(v) = s$;
 \item[(ii)] $r(t)=1$,  $r(u)=s$ and $r(a)=s z^{-\delta} s^{-1}$;
 \item[(iii)] $r(t) = 1$, $r(a) = s$ and $r(b)= s^{-1}$ if $k_1 = -k_2$; and $r(t)=s$, $r(a) = sz^{-\epsilon}s^{-1}$ and $r(b) = z^{-\delta}$ otherwise;
\end{itemize}
where $s$ is a non trivial element a surface group of $\Lambda$ other than $S$ (if there are some), and the generator of an infinite cyclic group $\Z$ if not. We extend $r$ to $A$ (respectively to $A* \Z$) by the identity on the other surface type vertex groups of $A$ (as well as on $\Z$). This gives $A$ a structure of (extended) hyperbolic floor over $r(A)$ (respectively over $\langle z \rangle$). 

\end{proof}

\subsection{Proof of Proposition \ref{Retraction}} \label{ProofRetractionSubsec}

Let $A$ be a torsion-free hyperbolic group which admits a JSJ-like decomposition $\Lambda$. 
Let $f: A \to A$ be a non injective preretraction with respect to $\Lambda$. We want to show that there exists an extended hyperbolic floor $(A, A', r)$.

Let ${\cal C}$ be a maximal set of essential curves pinched by $f$. Up to replacing $f$ by another non injective preretraction, we may assume that ${\cal C}$ is not contained properly in any essential set of curves pinched by a preretraction. We choose a pinching decomposition associated with ${\cal C}$.

Recall that ${\cal S}_{\cal C}(f)$ is the set of surfaces one of whose corresponding vertex groups intersects $f_A(A_{\cal C})$ in more than a boundary subgroup, and that $f$ is ${\cal C}$-minimal if for no preretraction $h$ factoring through $\rho_{\cal C}$ does ${\cal S}_{\cal C}(f)$ properly contain ${\cal S}_{\cal C}(h)$.

Without loss of generality, we may assume that $f$ is ${\cal C}$-minimal. By Remark \ref{StableAndMinimal} and Proposition \ref{IfNoStable}, we may also suppose that $f$ is also stable and that $f(A) \leq f_A(A_{\cal C})$. 

\begin{lemma} We have ${\cal S}_{\cal C}(f^k) = {\cal S}_{\cal C}(f)$ for all $k > 0$. 
\end{lemma}

\begin{proof} Since $f$ is stable, $f^k$ is also a preretraction. The set ${\cal C}$ is an essential set of curves pinched by $f^k$. Now $(f^k)_A = f^{k-1} \circ f_A$ so $(f^k)_A(A_{\cal C})$ is contained in $f(A)$, which is contained in $f_A(A_{\cal C})$ by assumption on $f$. In particular ${\cal S}_{\cal C}(f^{k}) \subseteq {\cal S}_{\cal C}(f)$. By ${\cal C}$-minimality of $f$, this inclusion is an equality. 
\end{proof}

\begin{rmk} \label{f2minimal} Note that this implies that $f^k$ is also ${\cal C}$-minimal. Moreover, $f^k$ is also stable, and $f^k(A) \leq (f^k)_A(A_{\cal C})$.
\end{rmk}

Let ${\cal C}^+$ be a maximal essential set of curves on the surfaces of $\Lambda$ whose corresponding elements are sent to edge groups of $\Lambda$ by $f$. Note that the elements corresponding to curves of ${\cal C}^+$ are also sent to edge groups of $\Lambda$ by any power of $f$. Up to replacing $f$ by $f^k$ for some $k$, we can assume ${\cal C}^+$ is a maximal essential set of curves whose corresponding elements are sent to edge groups of $\Lambda$ by $f$. By Remark \ref{f2minimal}, we see that none of the properties of $f$ are affected by this modification.

With the set ${\cal C}^+$, we can build elliptic refinements $\Lambda^+_i$ of the graphs of groups $\Lambda_i$ with respect to $f$ as in Section \ref{EllipticRefSubsec}.  

\begin{lemma} \label{IsoOnCoveredSurfaces} Some power of $f$ sends surface type vertex groups whose corresponding surface is in ${\cal S}_{\cal C}(f)$ isomorphically onto conjugates of themselves.
\end{lemma}

\begin{proof} Let $\Sigma$ be a surface in ${\cal S}_{\cal C}(f)={\cal S}_{\cal C}(f^2)$. Let $S$ be a surface type vertex group corresponding to $\Sigma$ which intersects $(f^2)_A(A_{\cal C})$ in more than a boundary subgroup. Note that $(f^2)_A = f \circ f_A$. 

By hypothesis on $f$, the set ${\cal C}$ is in fact a maximal set of essential curves pinched by $f^2$. Since $f^2$ is ${\cal C}$-minimal, by Proposition \ref{NoWeirdSurfaces}, there is a surface type vertex group $S_0$ in one of the graphs of groups $\Lambda^+_i$ such that $(f^2)_A(S_0)$ has finite index in $S$. Now $f_A(S_0)$ is elliptic in $\Lambda$: if it lies in a non surface type vertex group $V$ of $\Lambda$, then $(f^2)_A(S_0) = f(f_A(S_0))$ lies in a conjugate of $V$, but $V \cap S$ is at most cyclic which contradicts $(f^2)_A(S_0)$ having finite index in $S$. Thus $f_A(S_0)$ lies in a surface type vertex group $S_1$ of $\Lambda$. By Lemma \ref{FiniteIndex}, since $f_A$ is non pinching on the surfaces of $\Lambda^+_i$, $f_A(S_0)$ has finite index in $S_1$. In particular, the surface $\Sigma_1$ corresponding to $S_1$ is in ${\cal S}_{\cal C}(f)$. 

Now $f(S_1)$ has a subgroup of finite index, namely $f(f_A(S_0))$, which is elliptic in $\Lambda$. Thus it is itself elliptic in $\Lambda$. Since $f(f_A(S_0))$ is not abelian and is contained in $S$, $f(S_1)$ must be contained in $S$. Hence $f(S_1)$ has finite index in $S$. By Lemma \ref{ComplexitiesAndFiniteIndex}, this implies that the complexity $k(\Sigma_1)$ of $\Sigma_1$ is greater than that of $\Sigma$, and if we have equality $f|_{S_1} : S_1 \to S$ is an isomorphism of surfaces. 

Consider the application $\Sigma \mapsto \Sigma_1$ from ${\cal S}_{\cal C}(f)$ to itself. It is injective since $f(S_1) \leq S$, and it increases complexity: since ${\cal S}_{\cal C}(f)$ is finite, it is a bijection and we have $k(\Sigma_1) = k(\Sigma)$ for all $\Sigma$. Thus $f$ sends each surface type vertex groups whose corresponding surface is in ${\cal S}_{\cal C}(f)$ isomorphically onto such a vertex group. This proves the result. 
\end{proof}

We replace $f$ by the power given by Lemma \ref{IsoOnCoveredSurfaces}.

Let $\Gamma$ be the graph of groups decomposition obtained from $\Lambda$ by collapsing all the edges except those adjacent to surface type vertices whose corresponding surfaces are not in ${\cal S}_{\cal C}(f)$. It naturally inherits a structure of graph of groups with surfaces from $\Lambda$. Note that any edge of $\Gamma$ is adjacent to exactly one surface type vertex and one non surface type vertex.

Denote by $\Gamma_1, \ldots, \Gamma_t$ the subgraphs of groups of $\Lambda$ which are inverse images of the non surface type vertices of $\Gamma$ under the collapse map $\Lambda \to \Gamma$. Note that $\Gamma_j$ is a JSJ-like decomposition. Any surface $\Sigma$ of $\Gamma_j$ is in ${\cal S}_{\cal C}(f)$, thus $f$ is injective on the corresponding surface type vertex groups by Lemma \ref{IsoOnCoveredSurfaces}. In particular, no curves of ${\cal C}$ lie on $\Sigma$. Thus under the refinement of $\Lambda$ to $\Lambda_{\cal C}$, $\Gamma_j$ embeds naturally in $\Lambda_{i_j}$ for some ${i_j}$. 

We can thus choose non surface type vertex groups $H_1, \ldots, H_t$ of $\Gamma$ such that $\rho_{\cal C}(H_j) \leq A_{i_j}$. 

\begin{lemma} \label{AiInHi} For each $i$, the group $f_A(A_i)$ is contained in a non surface type vertex group of $\Gamma$.
\end{lemma}

\begin{proof} By Lemma \ref{OneLambda}, if a surface type vertex $v_S$ of $T_{\Lambda}$ lies in $\phi_i(T^+_i)$, the corresponding vertex group $S$ intersects $f_A(A_i)$ in a finite index subgroup, so the surface corresponding to $v_S$ is in ${\cal S}_{\cal C}(f)$. So the tree $\phi_i(T^+_i)$ does not contain any vertices which are sent to surface type vertices of $T_{\Gamma}$ under the collapse map $ T_{\Lambda} \to T_{\Gamma}$. By bipartism of $T_{\Gamma}$, this implies that $f(A_i)$ is contained in a non surface type vertex group of $\Gamma$.
\end{proof}

We now show
\begin{lemma}\label{DerDesDer} Each $H_j$ is sent isomorphically to a conjugate of itself by $f$, and $f_A(A_{i_j})$ is a conjugate of $H_j$.
\end{lemma}

\begin{proof} We start by remarking that bipartism of $\Gamma$ and $1$-acylindricity at surface type vertices imply that distinct non surface type vertex groups of $\Gamma$ intersect trivially.

Now $\rho_{\cal C}(H_j) \leq A_{i_j}$, and by Lemma \ref{AiInHi}, $f_A(A_{i_j})$ lies in $g_jH_{k_j}g_j^{-1}$ for some $k$. Any vertex group $V$ of $\Gamma_j$ is sent to a conjugate of itself $g_V Vg^{-1}_V$ by $f$, so $f(H_j)$ intersects $g_V H_jg^{-1}_V$ non trivially. Thus we must have $k_j=j$, so that $f_A(A_{i_j}) \leq g_jH_jg_j^{-1}$, and $g^{-1}_j g_V  \in H_j$. 

The map $F = \Conj(g^{-1}_j) \circ f|_{H_j}$ is a morphism $H_j \to H_j$ which restricts to conjugation by an element of $H_j$ on non surface type vertex groups of $\Gamma_j$, and sends surface type vertex groups of $\Gamma_j$ isomorphically to a conjugate of themselves by an element of $H_j$. By Proposition \ref{IsoOnVertexIsIsoOnGroup}, $F$ is an isomorphism, hence so is $f|_{H_j}$. Thus in fact $f_A(A_{i_j}) = g_jH_jg_j^{-1}$
\end{proof}

We can finally finish the proof of Proposition \ref{Retraction}. By Lemma \ref{DerDesDer}, $f_A(A_{i_j}) = g_jH_jg^{-1}_j$, so the map $j \mapsto i_j$ is injective. Each $\Lambda_i$ contains at least one non surface type vertex group, so $i=i_j$ for some $j$, and up to renumbering we can assume $i_j = j$.

The group $f(A)$ acts on $T_{\Gamma}$. We assumed $f(A) = f_A(A_{\cal C})$, so the intersection of $f(A)$ with any surface type vertex group of $\Gamma$ is contained in a boundary subgroup. Thus in the action of $f(A)$ on $T_{\Gamma}$, a surface type vertex of $\Gamma$ has at most one non trivially stabilized adjacent edge. This implies that $f(A)$ is the free product of subgroups of conjugates of the groups $H_i$ with possibly a free group. But $f_A(A_{\cal C})$ is generated by the subgroups $f_A(A_i)$, so $f(A) = g_1H_1g^{-1} * \ldots * g_lH_lg^{-1}_l$.

Denote by $A'$ the subgroup of $A$ generated by the subgroups $H_i$. The morphism $f|_{A'}: A' \to g_1H_1g^{-1} * \ldots * g_lH_lg^{-1}_l$ is injective on each $H_i$, and distinct subgroups $H_i$ are sent to distinct free factors of the image. This implies that $A'$ is in fact the free product of the subgroups $H_i$.

The morphism $r=(f|_{A'})^{-1} \circ f$ restricts to the identity on $A'$, so it is a retraction $A \to A'$, and it is not injective since $f$ isn't. Moreover, it sends surface type vertex groups of $\Gamma$ to non abelian images since $f$ does. Thus $(A, A', r)$ forms a hyperbolic floor with respect to the decomposition $\Gamma$. This finishes the proof.

\nocite{Sel2}
\nocite{Sel3}
\nocite{Sel4}
\nocite{Sel5}
\nocite{Sel6}
\nocite{Sel7}
\bibliography{biblio}

\end{document}